\documentclass{amsart}[utf8, 10pt]%
\usepackage{amsmath}
\usepackage{amsfonts}
\usepackage{amssymb}
\usepackage{mathrsfs}
\usepackage{graphicx}
\usepackage{hyperref}
\usepackage{stmaryrd}
\usepackage{enumerate}
\usepackage{bbm}
\usepackage{color}%
\usepackage{esint}
\usepackage{mathtools}
\usepackage[title,titletoc]{appendix}
\usepackage{xpatch,amsthm}
\makeatletter
\xpatchcmd{\@thm}{\fontseries\mddefault\upshape}{}{}{} 
\makeatother
\providecommand{\U}[1]{\protect\rule{.1in}{.1in}}
\newtheorem{theorem}{Theorem}

\newtheorem{claim}[theorem]{Claim}

\newtheorem{corollary}[theorem]{Corollary}

\newtheorem{definition}[theorem]{Definition}

\newtheorem{lemma}[theorem]{Lemma}
\newtheorem{notation}[theorem]{Notation}

\newtheorem*{question*}{Question}
\newtheorem{proposition}[theorem]{Proposition}

\newcommand{\X}{\mathfrak X}
\newcommand{\Xu}{\mathfrak X u}

\makeatletter
\@namedef{subjclassname@2020}{\textup{2020} Mathematics Subject Classification}
\makeatother

\begin{document}
	\author[A. Mallick]{Arka Mallick}
	\address{Department of Mathematics\\ Indian Institute of Science\\ Bangalore, India.}
	\email{arkamallick@iisc.ac.in}
	\author[S. Sil]{Swarnendu Sil} 
	\address{Department of Mathematics\\ Indian Institute of Science\\ Bangalore, India.}
	\email{swarnendusil@iisc.ac.in}
	\title[Excess decay in Heisenberg groups]{Excess decay for quasilinear equations in the Heisenberg group and consequences}
	\date{\today}
\thanks{S. Sil's research is partially supported by SERB MATRICS MTR/2023/000885.}
	\begin{abstract}
		We study regularity results for the solutions of quasilinear subelliptic $p$-Laplace type equation in Heisenberg groups. We prove somewhat surprising excess decay estimates for the constant coefficient homogeneous equation. Excess decay estimates, while well known in the Euclidean case, due to the celebrated works of Uraltseva and Uhlenbeck, was not known in the setting of Heisenberg groups until now and this lack of excess decay estimate is often attributed to the noncommutativity of the horizontal vector fields. Our results show that, in spite of the this noncommutative feature, excess decay estimates analogous to the Euclidean case hold.\smallskip  
		
		To illustrate the potency of our excess decay estimates, we prove two results. First is a H\"{o}lder continuity result that extends the presently known results to the full range $1< p < \infty.$  The second is a sharp borderline continuity result for the horizontal gradient of the solution. More precisely, we show that if  $u \in HW_{\text{loc}}^{1,p}$ satisfies $$\operatorname{div}_{\mathbb{H}} \left( a(x) \lvert \Xu \rvert^{p-2} \Xu \right) \in L^{(Q,1)}_{\text{loc}}$$ in a 
		domain of $\mathbb{H}_{n},$ where $\mathbb{H}_{n}$ is the Heisenberg group with homogeneous dimension $Q=2n+2,$ for $1 < p < \infty, $ with uniformly positive, bounded, Dini continuous scalar function $a$, 
		then $\Xu$ is continuous. This generalizes the classical result by Stein and the work of Kuusi-Mingione for linear and quasilinear equations, respectively,  in the Euclidean case and Folland-Stein for the linear case in the Heisenberg group setting. This result is new even when either $a$ is constant or when the equation is homogeneous. 
	\end{abstract}
		\maketitle 
	\section{Introduction and main results}
	Let $n \geq 1 $  be an integer and let $\mathbb{H}_{n}$ denote the $n$-th Heisenberg group over the reals. Let $Q:=2n+2$ stands for the homogeneous dimension of $\mathbb{H}_{n}.$ Heisenberg groups $\mathbb{H}_{n}$ are important objects in mathematics and mathematical physics alike. From the point of view of mathematics, they appear in diverse areas such as representation theory, harmonic analysis, complex geometry, subRiemannian geometry and metric geometry. They are the simplest nonAbelian examples of  stratified nilpotent Lie groups, which are homogeneous groups and also metric groups. They are also the simplest examples of Carnot groups and Carnot-Carath\`eodory spaces, apart from Euclidean spaces.  They are \textrm{CR} manifolds, which also serve as models for boundaries of domains in $\mathbb{C}^{n+1}$. The study of subLaplacians on Heisenberg groups are also interesting from the point of view PDEs, as these are differential operators defined by the so-called H\"{o}rmander vector fields and are hypoelliptic, but not elliptic. Nonetheless, the study of these subLaplacians and the related singular and fractional integrals is well developed thanks to several seminal contributions, e.g.  \cite{Kohn_Nirenberg_hypoelliptic}, \cite{Hormander_hypoelliptic}, \cite{Koranyi_Vagi_singularintegralsonhomogeneousspaces},  \cite{Folland_Kohn_delbarNeumann}, \cite{Folland_Stein_delbarNeumann}, \cite{Folland_subelliptic}, \cite{Folland-Stein_book}, \cite{Roncal_Thangavelu} and  many others ( see the book \cite{Bonfigliolietal_subLaplacian} and the references therein ). Subelliptic $p$-Laplace type equations has also been studied intensively for quite some time, see the book \cite{Ricciotti_pLaplaceHeisenberg}, also \cite{Capogna_Danielli_book} and the recent survey \cite{Capogna_Citti_Zhong}. Our present contribution seeks to add significant momentum to the understanding of $p$-subLaplacians on Heisenberg groups.

	The main concern of the present article is to remove perhaps the single most significant roadblock to the progress of regularity theory for $p$-Laplace type equations in Heisenberg groups - namely the absence of an excess decay estimate. 
	
	\textbf{Excess decay in the Euclidean case :} If $u \in W^{1,p}_{loc}\left(\Omega\right)$ is a local weak solution to the homogeneous $p$-Laplace equation
	\begin{align*}
		\operatorname{div} \left( \left\lvert \nabla u \right\rvert^{p-2}\nabla u \right)&= 0 &&\text{ in } \Omega, 
	\end{align*} where $\Omega \subset \mathbb{R}^{n}$ is a domain in the $n$-dimensional Euclidean space, we have the celebrated excess decay estimate ( see \cite{Giaquinta_Modica_Uhlenbeck_result}, \cite{hamburgerregularity}, also \cite{Uralceva}, \cite{uhlenbecknonlinearelliptic} and \cite{Ladyzhenskaya_Uraltseva} ) 
	\begin{align}\label{Uhlenbeck Excess decay}
		\fint_{B_{\rho}} \left\lvert V \left( \nabla u\right) - \left( V \left( \nabla u\right)\right)_{\rho}\right\rvert^{2} \leq C \left( \frac{\rho}{R}\right)^{2\beta} 	\fint_{B_{R}} \left\lvert V \left( \nabla u\right) - \left( V \left( \nabla u\right)\right)_{R}\right\rvert^{2}, 
	\end{align}
	for any balls $B_{\rho} \subset B_{R} \subset \Omega,$ for some $\beta \in (0,1),$ where $V\left(\xi\right):= \left\lvert \xi \right\rvert^{\frac{p-2}{2}}\xi$ and $C\equiv C(n,p)>0$ is a constant. Later, DiBenedetto-Manfredi \cite{DiBenedettoManfredi} obtained a different type of excess decay estimate, in terms of $\nabla u$ itself. They showed the excess decay estimate  
	\begin{align}\label{DiBenedettoManfredi Excess deay}
		\fint_{B_{\rho}} \left\lvert \nabla u - \left( \nabla u \right)_{\rho}\right\rvert^{2} \leq C \left( \frac{\rho}{R}\right)^{2\alpha} 	\fint_{B_{R}} \left\lvert \nabla u - \left(  \nabla u \right)_{R}\right\rvert^{2}, 
	\end{align} 
	again for any balls $B_{\rho} \subset B_{R} \subset \Omega,$ for some $\alpha \in (0,1)$ and for some constant $C\equiv C(n,p)>0$. The H\"{o}lder continuity of $\nabla u$ follows immediately from either of these excess decay estimates. The key to proving \eqref{Uhlenbeck Excess decay} is a type of reverse H\"{o}lder inequality ( see \cite{Giaquinta_Modica_Uhlenbeck_result}, \cite{hamburgerregularity} ), which follows from the relevant Caccioppoli inequality in these case. However, the proof of \eqref{DiBenedettoManfredi Excess deay} goes via an iteration argument, which is adapted in \cite{DiBenedettoManfredi} from its parabolic counterparts in  \cite{DiBenedettoFriedman1}, \cite{DiBenedettoFriedman2}. A Caccioppoli type inequality is also instrumental in this case as well.  These excess decay estimates have also proved pivotal for the  subsequent development of the regularity theory for quasilinear equations and systems with coefficients and non-zero right hand sides ( see e.g. \cite{DiBenedettoManfredi}, \cite{DuzaarMingione_gradient_linearnonlinearpotential} \cite{DuzaarMingione_gradientcontinuityestimates}, \cite{DuzaarMingione_gradientcontinuity_nonlinearpotential}, \cite{KuusiMingione_Steinequationslinearpotentials}, \cite{KuusiMingione_nonlinearStein} etc ).
	  	
	\textbf{Setting of Heisenberg groups :} Coming to the setting of Heisenberg groups, if one considers the analogous equation 
	\begin{align*}
		\operatorname{div}_{\mathbb{H}} \left( \left\lvert \X u\right\rvert^{p-2}\X u\right) &= 0 &&\text{ in } \Omega,
	\end{align*}
	where $\operatorname{div}_{\mathbb{H}}$ denotes the horizontal divergence, $\X u$ denotes the horizontal gradient and now $\Omega \subset \mathbb{H}_{n}$ is a domain in the Heisenberg group $\mathbb{H}_{n},$ the main difference from the Euclidean case is that although $\mathbb{H}_{n} \simeq \mathbb{R}^{2n+1}$ as sets, only the horizontal derivatives $\X_{1}, \ldots, \X_{2n}$ appears in the equation, not the full Euclidean gradient $ \left( X_{1}u, \ldots, X_{2n}u, Tu \right).$ Perhaps the even more significant difference is that unlike the Euclidean derivatives, the horizontal derivatives do not commute and the vertical derivative $Tu$ is the only nontrivial commutator 
	\begin{align}\label{expression of Tu intro}
		Tu = X_{i}X_{n+i}u - X_{n+i}X_{i}u \qquad \text{ for all } 1\leq i \leq n. 
	\end{align}  
	Practically all the difficulties of the Heisenberg group setting, compared to the Euclidean case, can be traced back to the difficulty of estimating the vertical derivative $Tu,$ which appears as soon as one tries to write down the equations satisfied by the horizontal first derivatives $X_{i}u.$ These equations (see \eqref{equation:horizontal} and \eqref{equation:horizontal2} ) are still subelliptic, but they are inhomogeneous, whereas one obtains only homogeneous equations in the Euclidean case. Reverse H\"{o}lder type inequalities, as mentioned above, follows from Caccippoli inequalities for homogeneous elliptic equations, but for inhomogeneous equations, they are considerably harder to obtain and do not, in general, hold. In this case, this task is even more problematic due to the fact that the right hand side contains terms involving the vertical derivative $Tu$ and as the expression for $Tu$ in \eqref{expression of Tu intro} shows, morally, $Tu$ is more like a second order derivative than a first order derivative. 
	
	Due to these difficulties, the progress in the Heisenberg group setting was much slower. First result was due to Capogna \cite{Capogna_regularity}, who proved a regularity result which was enough to conclude $C^{\infty}$ regularity in the nondegenerate case. After this, next significant breakthrough was achieved in Domokos \cite{Domokos_differentiabilityofTu}, who proved a higher integrability estimate for $Tu$ when $1<p<4$. For a long time, all available regularity results were restricted to a small range of $p$, since these relied crucially on the result by Domokos (  \cite{Marchi}, \cite{Manfredi_Mingione_Heisenberg}, \cite{Mingione_ZatorskaGoldstein_Zhong} etc ).

	\textbf{Zhong's mixed Caccioppoli and oscillation control by energy :} The first result to break this barrier is due to Zhong \cite{zhong2018regularityvariationalproblemsheisenberg}. In this remarkable work, Zhong proved a number of mixed Caccioppoli inequalities and used them to prove the the H\"{o}lder continuity of $\X u$ for $p \geq 2.$ Later, this program was carried to completion in Mukherjee-Zhong \cite{Mukherjee_Zhong}, where  H\"{o}lder continuity of $\X u$  was proved also for the range $1 < p < 2$, using the ideas of Tolksdorf \cite{Tolksdorf_regularity}.   
	
	However, both these results established H\"{o}lder continuity by deriving oscillation estimates by `energy'. More precisely, the estimate is 
	\begin{align}\label{oscillation control by energy Zhong}
		\left\lvert \Xu \left(x\right) - \Xu \left(y\right)\right\rvert \leq
		c\Big(\frac{\rho}{R}\Big)^{\alpha_{1}}\Big(\fint_{B_{R}}\left\lvert \X u \right\rvert^{p}\, dx\Big)^{\frac{1}{p}}, \quad \text{ for all } x, y \in B_{\rho},
	\end{align}
	for any metric balls $B_{\rho} \subset B_{R/2} \subset B_R \subset \subset \Omega.$ This is considerably weaker than estimates of the form \eqref{Uhlenbeck Excess decay} or \eqref{DiBenedettoManfredi Excess deay}, as the energy in $B_{R}$ can be much larger than the average osicillation in $B_{R}.$ For proving H\"{o}lder continuity of $\Xu,$ \eqref{oscillation control by energy Zhong} is sufficient, but this is not quite strong enough so that one can derive other results using the same perturbation techniques as in the Euclidean case. Indeed, the only perturbation type results that is known so far is the following: If $2-1/Q < p < \infty$ and  $u$ is a weak solution to the equation   
	\begin{align*}
		\operatorname{div}_{\mathbb{H}} \left( a(x)\left\lvert \X u\right\rvert^{p-2}\X u\right) &= f &&\text{ in } \Omega,
	\end{align*}
	where $a$ is uniformly positive, bounded, H\"{o}lder continuous and $f$ is H\"{o}lder continuous, then  $\X u$ is H\"{o}lder continuous. This was proved in Mukherjee-Sire\cite{Mukherjee_Sire} for the range $p \geq 2$ and in Yu \cite{Holderforpless2} for the range $2-1/Q < p < 2,$ using Wolff potential bounds. In \cite{Mukherjee_Sire}, the authors noted ( see Proposition 3.1 in \cite{Mukherjee_Sire} ) that the mixed Caccioppoli inequalities of Zhong only yields the significantly weaker estimate 
	\begin{align}\label{Mukherjee_Sire_estimate}
		\fint_{B_{\rho}} \left\lvert \nabla u - \left( \nabla u \right)_{\rho}\right\rvert\ \mathrm{d}x \leq C \left( \frac{\rho}{R}\right)^{\alpha} \left[ 	\fint_{B_{R}} \left\lvert \nabla u - \left(  \nabla u \right)_{R}\right\rvert\ \mathrm{d}x + \chi R^{\alpha} \right] ,
	\end{align}   
	where $$\chi = \frac{1}{R_{0}^{\alpha}}\sup\limits_{B_{R_{0}}} \left\lvert \X u \right\rvert, $$ for any metric balls $B_{\rho} \subset B_{R} \subset \subset B_{R_{0}} \subset \Omega.$ The authors in \cite{Mukherjee_Sire} also asserted, both in the Introduction and in Remark 3.2, that the appearance of the extra term with $\chi \neq 0$ is unavoidable and seemed to suggest this as a manifestation of the appearance of the terms involving $Tu$ in the equations for $X_{i}u,$ which are evidently absent in the Euclidean case. 
	
	\textbf{An excess decay estimate in Heisenberg groups :} Obviously, if there were a genuine lack of Euclidean-like excess decay estimates, it is not at all unreasonable to ascribe that lack to the noncommutative nature of the vector fields. Our main contribution in the present work, somewhat surprisingly in this sense, asserts that  Euclidean type excess decay estimates indeed hold in the Heisenberg group setting as well, the noncommutativity of the vector fields notwithstanding.  
	
	\begin{theorem}\label{Main excess decay estimate Intro}
		Suppose $u$ solves the constant coefficient homogeneous $p$-Laplace equation in a domain $\Omega \subset \mathbb{H}_{n},$ with $1< p < \infty,$ 
		\begin{align*}
			\operatorname{div}_{\mathbb{H}} \left( \left\lvert \X u\right\rvert^{p-2}\X u\right) &= 0 &&\text{ in } \Omega. 
		\end{align*}
		Let $1\leq q \leq 2$ be a real number. Then there exists constants $\alpha \in (0, 1),$ and $C_{0} \geq 1,$ both depending only on $n$ and $p$,   such that for any $x_{0} \in \Omega$ and any $0 < \rho < R$ such that $B_{R}\left(x_{0}\right) \subset \subset \Omega,$ we have 
		\begin{align}\label{Main excess decay estimate Intro estimate}
			\left( \fint_{B_{\rho}\left(x_{0}\right)} \left\lvert \X u - \left( \X u\right)_{x_{0}, \rho}\right\rvert^{q} \right)^{\frac{1}{q}}\leq C_{0} \left( \frac{\rho}{R}\right)^{\alpha} 	\left( \fint_{B_{R}\left(x_{0}\right)} \left\lvert \X u - \left( \X u\right)_{x_{0}, R}\right\rvert^{q} \right)^{\frac{1}{q}}.  
		\end{align} 
	\end{theorem}
	We also have a excess decay estimate for $V\left(\X u\right).$
	\begin{theorem}\label{Main excess decay estimate V Intro}
		Suppose $u$ solves the constant coefficient homogeneous $p$-Laplace equation in a domain $\Omega \subset \mathbb{H}_{n},$ with $1< p < \infty,$
		\begin{align*}
			\operatorname{div}_{\mathbb{H}} \left( \left\lvert \X u\right\rvert^{p-2}\X u\right) &= 0 &&\text{ in } \Omega. 
		\end{align*}
		Let $1\leq q \leq 2$ be a real number. There exists constants $\beta \in (0, 1)$ and $C_{0,V} \geq 1,$ both depending only on $n$ and $p$,  such that for any $x_{0} \in \Omega$ and any $0 < \rho < R$ such that $B_{R}\left(x_{0}\right) \subset \subset \Omega,$ we have 
		\begin{multline}\label{Main excess decay estimate Intro estimate V}
			\left( \fint_{B_{\rho}\left(x_{0}\right)} \left\lvert V \left( \X u \right)  - \left( V \left( \X u\right) \right)_{x_{0}, \rho}\right\rvert^{q}\, dx \right)^{\frac{1}{q}} \\ \leq C_{0,V} \left( \frac{\rho}{R}\right)^{\beta} \left( 	\fint_{B_{R}\left(x_{0}\right)} \left\lvert V \left( \X u \right)  - \left( V \left( \X u \right) \right)_{x_{0}, R}\right\rvert^{q}\, dx \right)^{\frac{1}{q}}.  
		\end{multline} 
	\end{theorem}
	The balls above are understood to be metric balls ( see Section \ref{Heisenberg group prelim} and \ref{Heisenberg group function spaces} ) and for any $\mathbb{R}^{2n}$-valued function $g$, $\left(g\right)_{x_0, R}$ denotes the integral average of $g$ over the metric ball $B_R(x_0)$ (see Notation \ref{notation for avarage}). The two theorems above are exact analogues of their Euclidean counterpart. We would also like to point out that while both results are known in the Euclidean setting, the passage from the exponent $2$  in \eqref{DiBenedettoManfredi Excess deay}  to any exponent $1\leq q \leq 2$ was achieved later, by an argument due to Lieberman \cite{Lieberman_expoentdecreasing}. For \eqref{Uhlenbeck Excess decay}, analogous passage can be achived by reverse H\"{o}lder inequalities established in \cite{Giaquinta_Modica_Uhlenbeck_result} and \cite{hamburgerregularity}. We also note that the proof of \eqref{Uhlenbeck Excess decay} for $1< p < 2$ follows via a duality argument in  \cite{hamburgerregularity}. Our argument yields the full result for the full range of $p$ in one go, without needing additional arguments and thus yields a new proof for the Euclidean case as well, as long as an oscillation to energy bound like \eqref{oscillation control by energy Zhong} is available ( for eaxmple, Tolksdorf \cite{Tolksdorf_regularity} proved such estimates independently of excess decay estimates ). Note also that the inequality in Theorem \ref{Main excess decay estimate Intro} for $q=1$ is exactly \eqref{Mukherjee_Sire_estimate} with $\chi = 0.$

	As we mentioned above, in the Euclidean case, H\"{o}lder continuity of gradient is proved using the excess decay estimates. The main idea behind our excess decay estimates is the converse question: Suppose we already know that the solutions are H\"{o}lder continuous and we have an oscillation estimate by energy. When can we derive the excess decay estimate? 
	
	The main insight of our estimate is that this is possible as soon as we have a linear-like Caccioppoli inequality in the nondegenerate regime. The form of the Caccioppoli inequalities in the general case does not  matter at all. While the Caccioppoli inequalities derived by Zhong in the general case is of mixed type, in the nondegenerate regime, we indeed have a Caccioppoli inequality which looks exactly like the one for a linear homogeneous equation ( see \eqref{estimate of XXu in nondegenerate case} ).  This itself is somewhat surprising, as the equations \eqref{equation:horizontal} and \eqref{equation:horizontal2} are not homogeneous.  These Caccioppoli inequalities imply reverse H\"{o}lder inequalities in the standard manner ( see \eqref{rev holder bound 1 to 2 star} ), but of course, again for the nondegenerate regime only.  
	
	Our iteration lemma ( Lemma \ref{iteration of close to constant lemma} ) is similar in spirit to Lemma 4.4 and Lemma 4.5 in \cite{DiBenedettoFriedman1} ( see also \cite{DiBenedettoFriedman2} ). However, unlike in their situations, we do not have a suitable Caccioppoli inequality for the equation in general. Thus our lemma is weaker as we always need the nondegeneracy condition $\lambda/4B \leq \left\lvert \X u \right\rvert \leq A\lambda.$ Due to this extra hypotheses, our proofs of Theorems \ref{Main excess decay estimate Intro} and \ref{Main excess decay estimate V Intro} diverge significantly from the proof of \eqref{DiBenedettoManfredi Excess deay} given in \cite{DiBenedettoManfredi}. Our result is sufficient to derive excess decay if oscillation estimate by energy is known beforehand. But for the same reasons, it cannot furnish an independent proof of the H\"{o}lder continuity result, while their's can.  We would also like to point out that the idea of treating the nondegenerate regime and the degenerate regime separately is also there in DiBenedetto-Manfredi \cite{DiBenedettoManfredi} ( also earlier in Evans \cite{Evans_pLaplacian} ). However, in the degenerate regime, one usually obtains a sup decay of the form 
	\begin{align}\label{degenerate sup decay}
		\sup\limits_{B_{\sigma R}} \left\lvert \nabla u \right\rvert \leq \eta \sup\limits_{B_{R}} \left\lvert \nabla u \right\rvert
	\end{align} 
	for some constants $\eta, \sigma \in (0,1).$ By iterating this, one can of course conclude the existence of an integer $i \geq 1$ such that  one has 
	\begin{align*}
		\sup\limits_{B_{\sigma^{i} R}} \left\lvert \nabla u \right\rvert \leq \eta^{i} \sup\limits_{B_{R}} \left\lvert \nabla u \right\rvert < 	\fint_{B_{R}} \left\lvert \nabla u - \left(  \nabla u \right)_{R}\right\rvert\ \mathrm{d}x.
	\end{align*}
	However, apriori, this integer $i$ depends on $\fint_{B_{R}} \left\lvert \nabla u - \left(  \nabla u \right)_{R}\right\rvert\ \mathrm{d}x.$ The novelty in our proof is that one can actually reduce to the case where it is possible to determine an integer $i_{a},$ apriori in terms of $n$ and $p$ alone such that this happens ( compare Claim \ref{nondegenerate case must occur claim} ). 
	
	\textbf{Consequences of excess decay estimates :} Our excess decay estimates make the problem in the Heisenberg setting amenable to Euclidean perturbation techniques. As an illustration, we have the following result.  
	\begin{theorem}[H\"{o}lder estimates]\label{campanato estimates}
		Let  $1< p < \infty$ and $0 < \mu_{1}, \mu_{2} < 1$ be real numbers. Let $\Omega \subset \mathbb{H}_{n}$ be open. Suppose 
		\begin{itemize}
			\item[(i)] $a:\Omega \rightarrow [ \nu, L ]$ is $C^{0,\mu_{1}}_{\text{loc}}\left( \Omega \right)$, where $0 < \nu < L < \infty$ and  
			\item[(ii)]$F \in C^{0,\mu_{2}}_{\text{loc}}\left(\Omega; \mathbb{R}^{2n} \right).$ 
		\end{itemize} 
		Let $u \in HW_{\text{loc}}^{1,p} \left(\Omega \right)$ be a local weak solution to 
		\begin{align}\label{p subLaplace divergenceform}
			\operatorname{div}_{\mathbb{H}} \left( a(x) \lvert \X u \rvert^{p-2} \X u \right)   &= \operatorname{div}_{\mathbb{H}}F   &&\text{ in } \Omega. 
		\end{align}
		Suppose $$ \mu := \min \left\lbrace \mu_{1}, \mu_{2} \right\rbrace < \min \left\lbrace \beta, \alpha \right\rbrace, $$ where  $\alpha, \beta \in (0,1)$ are the exponents given by Theorems \ref{Main excess decay estimate Intro} and \eqref{Main excess decay estimate V Intro} respectively. Then we have $$ \X u \in C^{0,\min\left\lbrace \mu, \frac{\mu}{p-1}\right\rbrace }_{\text{loc}} \left( \Omega; \mathbb{R}^{2n}\right).$$ 
	\end{theorem}
	As an immediate consequence, we have the following, which generalizes the results in \cite{Mukherjee_Sire} and \cite{Holderforpless2} to the full range $1< p < \infty.$  
	\begin{corollary}\label{campanato estimates nondivergence}
		Let  $1< p < \infty$  and let $\Omega \subset \mathbb{H}_{n}$ be open. Suppose 
		\begin{itemize}
			\item[(i)] $a:\Omega \rightarrow [ \nu, L ]$ is $C^{0,\mu_{1}}_{\text{loc}}\left( \Omega \right)$, where $0 < \nu < L < \infty$, $0 < \mu_{1} < 1$ and
			\item[(ii)]$f \in L^{q}_{\text{loc}}\left(\Omega \right)$ with $q>Q.$ 
		\end{itemize} 
		Let $u \in HW_{\text{loc}}^{1,p} \left(\Omega \right)$ be a local weak solution to 
		\begin{align}\label{p subLaplace nondivergenceform}
			\operatorname{div}_{\mathbb{H}} \left( a(x) \lvert \X u \rvert^{p-2} \X u \right)   &= f   &&\text{ in } \Omega. 
		\end{align}
		Suppose $$ \mu := \min \left\lbrace \mu_{1}, 1 - \frac{Q}{q} \right\rbrace < \min \left\lbrace \beta, \alpha \right\rbrace, $$ where  $\alpha, \beta \in (0,1)$ are the exponents given by Theorems \ref{Main excess decay estimate Intro} and \eqref{Main excess decay estimate V Intro} respectively. Then we have $$ \X u \in C^{0,\min\left\lbrace \mu, \frac{\mu}{p-1}\right\rbrace }_{\text{loc}} \left( \Omega; \mathbb{R}^{2n}\right).$$ 
	\end{corollary}
	The H\"{o}lder spaces above are understood to be Folland-Stein classes ( see Section \ref{Heisenberg group function spaces} ). We emphasize that our proof is via the classical Camapanto method and does neither require any perturbation argument nor any potential estimates. \smallskip

	\textbf{Subelliptic nonlinear Stein theorem : }The efficacy of our excess decay estimate goes beyond easy Campanato style perturbation arguments. In \cite{Stein_steintheorem}, Stein proved the proved the borderline Sobolev embedding result which states that for $n \geq 2,$ $u \in L^{1}(\mathbb{R}^{n})$ and $\nabla u \in L^{(n,1)}(\mathbb{R}^{n}; \mathbb{R}^{n})$ implies $u$ is continuous. Coupled with standard Calderon-Zygmund estimates, which extend to Lorentz spaces, this implies $u \in C^{1}(\mathbb{R}^{n})$ if $\Delta u \in L^{(n,1)}(\mathbb{R}^{n}).$ The search for a nonlinear generalization of this result culminated in 
	Kuusi-Mingione \cite{KuusiMingione_nonlinearStein}, where the authors proved the following general \emph{quasilinear} \emph{vectorial} version
	\begin{align}\label{stein pde version p laplacian}
		\operatorname*{div}  \left( a(x)\lvert \nabla u \rvert^{p-2} \nabla u \right) \in L^{(n,1)}(\mathbb{R}^{n}; \mathbb{R}^{N}) \qquad \implies \qquad u \in C^{1}(\mathbb{R}^{n};\mathbb{R}^{N}), 
	\end{align}
	where $a$ is a uniformly positive Dini continuous function and  $1 < p < \infty .$ For the case of Heisenberg groups, even the linear result that 
	\begin{align*}
		\operatorname{div}_{\mathbb{H}} ( \Xu) ) \in L^{(Q,1)} \implies \X u \text{ is continuous }, 
	\end{align*}
	is difficult to find in the literature. However, this can be shown combining the ideas from \cite{Stein_steintheorem} with the ones from Folland-Stein \cite{Folland-Stein_book}, Folland-Stein \cite{Folland_Stein_delbarNeumann} and Folland \cite{Folland_subelliptic}, see Section \ref{Linear subelliptic Stein} for details.  We prove the following nonlinear Stein type theorem for Hesienberg groups. 
	\begin{theorem}[Subelliptic Nonlinear Stein theorem]\label{Nonlinear Stein theorem intro}
		Let $1 < p < \infty$ and let $\Omega \subset \mathbb{H}_{n}$ be open. Suppose that 
		\begin{itemize}
			\item[(i)] $f \in L_{\text{loc}}^{(Q,1)}\left(\Omega \right),$ where $Q=2n+2$ is the homogeneous dimension of $\mathbb{H}_{n},$
			\item[(ii)] $a:\Omega \rightarrow [\gamma, L]$, where $ 0 < \gamma < L < \infty,$ is Dini continuous.
		\end{itemize}
		Let $u \in HW_{\text{loc}}^{1,p} \left(\Omega\right)$ be a local weak solution to the equation
		\begin{align}\label{Main equation Stein Heisenberg}
			\operatorname{div}_{\mathbb{H}} ( a(x) \lvert \Xu \rvert^{p-2} \Xu) )  = f   &&\text{ in } \Omega. 
		\end{align}
		Then $\X u $ is continuous in $\Omega.$ 
	\end{theorem}
	The conditions on the coefficient is sharp already in the Euclidean case, as shown by Jin et al \cite{JinMazyaVanSchaftingen_dinisharp}. Cianchi \cite{Cianchi_SharpLorentzexponent} showed the Lorentz space $L^{(n,1)}$ is also sharp in the Euclidean case. The homogeneous dimension $Q$ is the natural replacement of $n$ in the setting of Heisenberg groups. As far as we are aware, this result is new even when $f \equiv 0$ or when $a$ is a constant function. This result is nontrivial already in the Euclidean case, where it is the end product of a series of works, starting with Duzaar-Mingione \cite{DuzaarMingione_gradientcontinuityestimates} and culminating in Kuusi-Mingione \cite{KuusiMingione_perutbationparabolic}, Kuusi-Mingione \cite{KuusiMingione_Steinequationslinearpotentials}, Kuusi-Mingione \cite{KuusiMingione_nonlinearStein}. In the beautiful work \cite{KuusiMingione_nonlinearStein}, the authors derived suitable form of the comparison estimates in the nondegenerate regime. This, coupled with an exit time argument and the Euclidean excess decay, this is sufficient for the result.

	The delicate comparison estimates in \cite{KuusiMingione_nonlinearStein}, however, extends immediately to the Heisenberg group setting, as they are ultimately a consequence of the structure of the $p$-Laplace equation. The missing pieces of the puzzle that remained are the follwing:
	\begin{itemize}
		\item The excess decay estimate - which were not available, until now that is, in this setting.
		\item An intermediate continuity results for the homogeneous $p$-Laplace equation with Dini continuous coefficients ( see Theorem $2$ in \cite{KuusiMingione_nonlinearStein} ). This is a crucial intermediate step in their proof. In the Euclidean case, this result was obtained in the parabolic case in \cite{KuusiMingione_perutbationparabolic}.  
	\end{itemize} 
	
	We show, that the excess decay estimate is really the only missing key ingredient. With our excess decay estimates at hand, we prove the following. 
	\begin{theorem}\label{dwcontinuityhomogeneousDini}
		Let $1 < p < \infty$ and let $\Omega \subset \mathbb{H}_{n}$ be open. Let $w \in HW_{\text{loc}}^{1,p}\left( \Omega \right)$  be a local weak solution of 
		\begin{align}\label{p subLaplace with dini coeff}
			\operatorname{div}_{\mathbb{H}}	 ( a(x) \lvert \X w \rvert^{p-2} \X w) )  &= 0   &&\text{ in } \Omega, 
		\end{align} 
		where $a:\Omega \rightarrow [\gamma, L]$, where $ 0 < \gamma < L < \infty,$ is Dini continuous.
		Then $\X w$ is continuous in $\Omega$. 
		Moreover, if $\omega(\cdot)$ denotes the modulus of continuity associated with the function $a$, then
		\begin{itemize}
			\item[(i)] there exists a constant 
			$C_{2} \equiv C_{2}(n, p, \gamma, L, \omega( \cdot )) \geq 1$ and a positive radius $R_{1}= R_{1}(n, p, \gamma, L, \omega( \cdot )) >0$ such that if $R \leq R_{1}$, then the estimate 
			\begin{align}\label{sup estimate dini}
				\sup\limits_{B(x, R/2)} \left\lvert \X w \right\rvert \leq C_{2}  \fint_{B(x, R)} \left\lvert \X w \right\rvert,
			\end{align}
			holds whenever $B(x,R) \subset \Omega.$ If $a(\cdot )$ is a constant function, the estimate holds without any restriction on 
			$R.$
			\item[(ii)] Assume that the inequality 
			\begin{align}\label{sup bound for osc dini}
				\sup\limits_{B(x, R/2)} \left\lvert \X  w \right\rvert \leq A\lambda
			\end{align}
			hold for some $A \geq 1$ and $\lambda > 0.$ Then for any $\delta \in (0,1)$ there exists a positive constant 
			$\sigma_{1} \equiv \sigma_{1} (n, p, \gamma, L, \omega( \cdot ), A, \delta ) \in (0, \frac{1}{4})$ such that for every $\sigma \leq \sigma_{1},$ we have,  
			\begin{align}\label{osc estimate dini}
				\sup_{x,y \in B(x,\sigma R)} \left\lvert \X w(x) - \X w(y) \right\rvert \leq \delta \lambda.
			\end{align}
		\end{itemize}
	\end{theorem}
	This now takes care of all the missing pieces and the arguments of Kuusi-Mingione \cite{KuusiMingione_nonlinearStein} can be adapted, essentially verbatim with the obvious notational change, to our setting to establish Theorem \ref{Nonlinear Stein theorem intro}.  
	
\textbf{Future directions :} We fully expect that with the present contribution at hand, a full-blown `nonlinear potential theory' or a `nonlinear Calderon-Zygmund theory', as it is usually called in the Euclidean case, for the Heisenberg group case, is just around the corner now. In most cases, the lack of an excess decay estimate was the only real missing ingredient till date. Consequently, in our opinion, our estimates would enable most, if not all, of the gradient regularity results for the $p$-Laplacian in the Euclidean case to be extended to cover the horizontal gradient in the Heisenberg group setting. This would potentially alter the entire landscape of the field and allow the Heisenberg group case to catch up to the Euclidean case, which right now has a significant lead in development. 
	
	Our excess decay estimate is also instrumental for establishing estimates in mean oscillation spaces, namely $\mathrm{BMO}$ and $\mathrm{VMO}$. This will be taken up in an upcoming work.   
	
	Finally we remark that there is nothing that is inherently scalar in our arguments. As soon as we have an oscillation estimate by energy for the $p$-Laplacian system in Heisenberg groups, our argument would work the same and produce an excess decay estimate and consequently, both Theorem \ref{campanato estimates}, Corollary \ref{campanato estimates nondivergence} and Theorem \ref{Nonlinear Stein theorem intro} for $p$-Laplacian systems as well. In view of our results, we conjecture that these results hold for systems too. 
	
	In the present contribution, we only focus on the $p$-Laplace equation with or without coefficents in the Heisenberg groups. Although we have not paid any particular attention to this at present, but we believe most, if not all,  of our techniques extend to the case of a general nonlinearity 
	\begin{align*}
		\operatorname{div}_{\mathbb{H}} a\left( x, \X u \right) = f, 
	\end{align*}
	where the nonlinearity $a: \Omega \times \mathbb{R}^{2n} \rightarrow \mathbb{R}^{2n}$ statisfies a $p$-Laplace type structural assumptions or perhaps even more general nonlinearities. Whether our estimates also extend to cases beyond the Heisenberg group, for example, to Carnot groups of Step $2$, remains to be seen and is an interesting question for the future.

	\textbf{Organization :} The rest of the article is organized as follows. In Section \ref{prelim section}, we describe our notations and record the preliminary materials. Section \ref{excess decay section} proves our main estimates for the constant coefficient homogeneous equation, namely the excess decay estimate. Section \ref{Holder continuity section} proves Theorem \ref{campanato estimates} and Corollary \ref{campanato estimates nondivergence}. Section \ref{Homogeneous Dini section}  proves Theorem \ref{dwcontinuityhomogeneousDini}. Section \ref{Nonlinear Stein section} provides a sketch of how to adapt the proof of \cite{KuusiMingione_nonlinearStein} in our setting to prove Theorem \ref{Nonlinear Stein theorem intro}. Our presentation in this section is deliberately extremely brief. The argument in \cite{KuusiMingione_nonlinearStein} is itself somewhat long. Armed with our excess decay estimates, the proof in our case is an exact replica of their argument with the obvious changes in notations. Hence, we refrain from  the verbatim repetitions that would add little else apart from additional pages. Thus, we strongly suggest the readers to read Section \ref{Nonlinear Stein section} alongside the article  \cite{KuusiMingione_nonlinearStein} for better comprehension.     
	\section{Preliminaries}\label{prelim section}
	\subsection{Heisenberg groups}\label{Heisenberg group prelim}
	\begin{definition}\label{def:H group}
		For $n\geq 1$, the \textit{Heisenberg Group} denoted by $\mathbb{H}_{n}$, is identified with the Euclidean space 
		$\mathbb{R}^{2n+1}$ with the group operation 
		\begin{equation}\label{eq:group op}
			x\cdot y\, := \Big(x_1+y_1,\ \dots,\ x_{2n}+y_{2n},\ t+s+\frac{1}{2}
			\sum_{i=1}^n (x_iy_{n+i}-x_{n+i}y_i)\Big)
		\end{equation}
		for every $x=(x_1,\ldots,x_{2n},t),\, y=(y_1,\ldots,y_{2n},s)\in \mathbb{H}_{n}$.
	\end{definition}
	$\mathbb{H}_{n}$ with the group operation \eqref{eq:group op} along with the smooth structure of $\mathbb{R}^{2n+1}$ forms a 
	non Abelian simply connected stratified nilpotent Lie group, whose left invariant vector fields corresponding to the
	canonical basis of the Lie algebra, are
	\[ X_i=  \partial_{x_i}-\frac{x_{n+i}}{2}\partial_t, \quad
	X_{n+i}=  \partial_{x_{n+i}}+\frac{x_i}{2}\partial_t,\] 
	for every $1\leq i\leq n$ and the only
	non zero commutator $T = \partial_t$. 
	We have 
	\begin{equation}\label{eq:comm}
		[X_i\,,X_{n+i}]=  T\quad 
		\text{and}\quad [X_i\,,X_{j}] = 0\ \ \forall\ j\neq n+i.
	\end{equation}
	We call $X_1,\ldots, X_{2n}$ as \textit{horizontal
		vector fields} and $T$ as the \textit{vertical vector field}. 
	For a scalar function $ f: \mathbb{H}_{n} \to \mathbb{R}$, we denote
	$\X f  = (X_1f,\ldots, X_{2n}f)$ and $
	\X \X f =  (X_i(X_j f))_{i,j} $
	as the \textit{Horizontal gradient} and \textit{Horizontal Hessian}, 
	respectively. 
	From \eqref{eq:comm}, we have the following
	trivial but nevertheless, an important inequality 
	$
	|Tf|\leq 2|\X \X f|$. 
	For a vector valued function 
	$F = (f_1,\ldots,f_{2n}) : \mathbb{H}_{n}\to \mathbb{R}^{2n}$, the 
	\textit{Horizontal divergence} is defined as 
	$$ \operatorname{div}_{\mathbb{H}} (F)  =  \sum_{i=1}^{2n} X_i f_i .$$
	The Euclidean gradient of a 
	function $g: \mathbb{R}^{k} \to \mathbb{R}$, shall be denoted by
	$\nabla g=(D_1g,\ldots,D_{k} g)$ and the Hessian matrix by $\nabla^2g$.
	
	\subsection{Function spaces on \texorpdfstring{$\mathbb{H}_{n}$}{Hn}}\label{Heisenberg group function spaces}
	The \textit{Carnot-Carath\`eodory metric} ( henceforth CC-metric), which is a left-invariant metric on $\mathbb{H}_{n},$ denoted by $d_{\text{CC}},$ is defined as the length of the shortest horizontal curves, connecting two points. This is equivalent to the \textit{Kor\`anyi metric}, denoted as  
	$$ d_{\mathbb{H}_{n}}(x,y)= \left\lVert y^{-1}\cdot x \right\rVert_{\mathbb{H}_{n}} ,$$  
	where the Kor\`anyi norm for $x=(x_1,\ldots,x_{2n}, t)\in \mathbb{H}_{n}$ is 
	\begin{equation}\label{eq:norm}
		\left\lVert x \right\rVert_{\mathbb{H}_{n}} :=  \Big(\ \sum_{i=1}^{2n} x_i^2+ |t|\ \Big)^\frac{1}{2}.
	\end{equation}
	Throughout this article we use CC-metric balls, which are defined to be 
	\begin{align*}
		B_r(x) := \left\lbrace y\in\mathbb{H}_{n}: d_{\text{CC}}(x,y)<r \right\rbrace \qquad \text{ for }r>0 \text{ and } x \in \mathbb{H}_{n}.
	\end{align*} However, by virtue of the equivalence 
	of the metrics, all assertions for CC-balls can be restated to Kor\`anyi balls. 
	
	The bi-invariant Haar measure of $\mathbb{H}_{n}$ is just the Lebesgue 
	measure of $\mathbb{R}^{2n+1}$. We would consistently use the following notations. 
	\begin{notation}\label{notation for avarage}
		For a measurable set $A \subset \mathbb{H}_{n}$, we denote the $(2n+1)$- dimensional Lebesgue measure as $\left\lvert A \right\rvert$. Let $A \subset \mathbb{H}_{n}$ be any measurable subset with positive measure and let $g:A \rightarrow \mathbb{R}^{N}$ be a measurable function, taking values in any Euclidean space $\mathbb{R}^{N},$ $N \geq 1.$ We denote its integral average by the notation
		\begin{align*}
			\left(g\right)_{A} := \fint_{A} g\left(x\right)\ \mathrm{d}x := \frac{1}{\left\lvert A \right\rvert} \int_{A} g\left(x\right)\ \mathrm{d}x . 
		\end{align*} 
		If $A=B(x_0, R)$ for $x_0\in \mathbb{H}_n$ and $R>0$, then $\left(g\right)_{x_0, R}$ and $\left(g\right)_A$ denotes the same quantity. If $x_0$ is fixed in the context, then $\left(g\right)_{x_0, R}$ and $\left(g\right)_{R}$ denotes the same quantity.
	\end{notation}
	We record an important property of the integral averages that we would use throughout the rest. 
	\begin{proposition}
		Let $A \subset \mathbb{H}_{n}$ be any measurable subset with positive and finite measure and let $g:A \rightarrow \mathbb{R}^{N}$ be a measurable function, taking values in any Euclidean space $\mathbb{R}^{N},$ $N \geq 1.$ Then for any $ \xi \in \mathbb{R}^{N} $ and any $ 1 \leq q < \infty,$ we have 
		\begin{align}\label{minimality of mean}
			\left( \fint_{A} \left\lvert g - \left(g\right)_{A}\right\rvert^{q}\ \mathrm{d}x \right)^{\frac{1}{q}} \leq 2 	\left( \fint_{A} \left\lvert g - \xi \right\rvert^{q}\ \mathrm{d}x \right)^{\frac{1}{q}}. 
		\end{align}
	\end{proposition}
	The non-isotropic dilations are the group homorphisms given by
	
	\begin{equation}\label{dilation}
		\delta_R(x_1,\dots,x_{2n},t): = \left(R x_1,\dots, Rx_{2n}, R^2 t\right),
	\end{equation}
	where $R>0$ and we have the following relation 
	
	\begin{equation}\label{Balls under dialtion}
		B(0,R)= \delta_R B(0,1).
	\end{equation}
	The equivalence between the \textit{Kor\`anyi metric} and the CC-metric along with the natural scaling \eqref{dilation}, gives a natural candidate for the homogeneous dimension of the group $\mathbb H^n$, which is $2n+2$. Throughout this article we denote $Q:=2n+2$. The Hausdorff dimension with respect to the metric $d_{\text{CC}}$ is also $Q$. Since $d_{CC}$ is left invariant, so we have the following crucial property of metric balls. 
	\begin{proposition}
		For any CC-metric ball $B_r \subset \mathbb{H}_{n}$ with $r>0,$ there exists a constant $c=c(n)>0$ such that 
		\begin{align}\label{measure of balls}
			\left\lvert B_{r} \right\rvert = c(n)r^Q. 
		\end{align}  
	\end{proposition}
	Let $\Omega \subset \mathbb{H}_{n}$ be open. For $1 \leq p \leq \infty,$ the usual $L^{p}$ spaces $L^{p}\left(\Omega\right)$ is defined in the usual manner and extended componentwise to vector-valued functions. 
	\begin{definition}[Horizontal Sobolev spaces]
		For $ 1\leq p < \infty$, the \textit{Horizontal Sobolev space} $HW^{1,p}(\Omega)$ consists
		of functions $u\in L^p(\Omega)$ such that the distributional horizontal gradient $\X u$ is in $L^p(\Omega\,,\mathbb{R}^{2n})$.
		$HW^{1,p}(\Omega)$ is a Banach space with respect to the norm
		\begin{equation}\label{eq:sob norm}
			\| u\|_{HW^{1,p}(\Omega)}= \ \| u\|_{L^p(\Omega)}+\| \X u\|_{L^p(\Omega,\mathbb{R}^{2n})}.
		\end{equation}
		We define $HW^{1,p}_{\text{loc}}(\Omega)$ as its local variant and 
		$HW^{1,p}_0(\Omega)$ as the closure of $C^\infty_0(\Omega)$ in 
		$HW^{1,p}(\Omega)$ with respect to the norm in \eqref{eq:sob norm}. The second order Horizontal Sobolev space $HW^{2,p}(\Omega)$ is the Banach space 
		\begin{align*}
			HW^{2,p}(\Omega):= \left\lbrace u \in HW^{1,p}(\Omega): \X u \in HW^{1,p}(\Omega,\mathbb{R}^{2n}) \right\rbrace, 
		\end{align*} equipped with the norm  
		\begin{equation*}
			\| u\|_{HW^{2,p}(\Omega)}= \ \| u\|_{HW^{1,p}(\Omega)}+\| \X \X u\|_{L^p(\Omega,\mathbb{R}^{2n}\times\mathbb{R}^{2n} )}.
		\end{equation*}
	\end{definition}
	The definition is extended componentwise to vector valued functions as well.

	\begin{definition}[Folland-Stein classes]
		For $0 < \alpha \leq 1,$ we define 
		\begin{equation}\label{def:holderspace}
			C^{\,0,\alpha}(\Omega) =
			\left\lbrace u\in L^{\infty}(\Omega): \sup_{x, y\in \Omega, x\neq y} \frac{|u(x)-u(y)|}{d_{\text{CC}}(x,y)^\alpha}< +\infty\right\rbrace 
		\end{equation} 
		for $0<\alpha \leq 1$, 
		which are Banach spaces with the norm 
		\begin{equation}\label{eq:holder norm}
			\|u\|_{C^{\,0,\alpha}(\Omega)}
			=  \|u\|_{L^\infty(\Omega)}+ \left[ u\right]_{C^{0, \alpha}\left( \Omega\right)},
		\end{equation}
		where the Folland-Stein seminorm is defined as 
		\begin{align}\label{eq:Folland-Stein seminorm}
			\left[ u\right]_{C^{0, \alpha}\left( \Omega\right)} := \sup_{x,y\in\Omega, x\neq y} 
			\frac{|u(x)-u(y)|}{d_{\text{CC}}(x,y)^\alpha}. 
		\end{align}
	\end{definition}
	These have standard extensions to classes $C^{k,\alpha}(\Omega)$ 
	for $k\in \mathbb{N}$, which consists of functions having horizontal 
	derivatives up to order $k$ in $C^{\,0,\alpha}(\Omega)$. The local 
	counterparts are denoted as $C^{k,\alpha}_\text{loc}(\Omega)$.
	H\"{o}lder spaces with respect to homogeneous metrics are sometimes called 
	Folland-Stein classes, as they first appeared in Folland-Stein \cite{Folland-Stein_book} and denoted by $\Gamma^{\alpha}$ or 
	$\Gamma^{\,0,\alpha}$ in some literature. However, by a harmless abuse of notation, we would continue to use the classical notation. We also recall the definition of \textbf{modulus of continuity} of a uniformly continuous function in our metric setting. 
	\begin{definition}
		Let $\Omega \subset \mathbb{H}_{n}$ be an open subset and let $a:\Omega \rightarrow \mathbb{R}$ be a uniformly continuous function. The modulus of continuity of $a$, denoted $\omega_{a, \Omega}$ is a concave increasing function $\omega_{a, \Omega}: [0, \infty) \rightarrow [0, \infty),$ defined by 
		\begin{align*}
			\omega_{a, \Omega}\left( r \right):=  \sup\limits_{\substack{x,y \in \Omega,\\ d_{\text{CC}}(x,y) \leq r }} \left\lvert a\left(x\right) - a\left( y\right)\right\rvert \qquad \text{ for all  } r>0,
		\end{align*}
		and we set $\omega_{a, \Omega}\left(0\right)= 0$ by convention. 	
	\end{definition}
	We would often denote the modulus of continuity by $\omega_{a}$ when the domain is clear and also just $\omega$ when the function is also clear from the context.  It is easy to see that by definition, $a \in C^{0, \alpha} \left( \Omega\right)$ if and only if $\omega_{a, \Omega} \left(r\right) \leq C r^{\alpha}$ for some constant $C>0$ and the smallest such constant is equal to the seminorm $\left[ a\right]_{C^{0, \alpha}\left( \Omega\right)}.$ We now recall the definition of \emph{Dini continuity}. 
	\begin{definition}[Dini continuity]
		Let $\Omega \subset \mathbb{H}_{n}$ be an open subset. A function $a:\Omega \rightarrow \mathbb{R}$ is called \emph{Dini continuous} if $a$ is uniformly continuous in $\Omega$ and its modulus of continuity satisfies 
		\begin{align}\label{Dini continuity def}
			\int_{0}^{\infty} \omega_{a, \Omega} \left(\rho\right)\ \frac{\mathrm{d}\rho}{\rho} < \infty. 
		\end{align}
	\end{definition}
	Now, assume $\Omega \subset \mathbb{H}_{n}$ is any open subset such that there exists a constant $A = A\left( \Omega\right)>0$ such that for any $x_{0} \in \Omega$ and any $r < \operatorname{diam} \Omega,$ we have  
	\begin{align*}
		\left\lvert B_{r}\left(x_{0}\right) \cap \Omega \right\rvert \geq Ar^{Q}. 
	\end{align*} 
	This property is clearly satisfied if $\Omega$ itself is a metric ball. 
	\begin{definition}[Morrey and Campanato spaces]
		For $1 \leq p <  \infty$ and $\lambda \geq 0,$  $\mathrm{L}^{p,\lambda}\left(\Omega\right) $ stands for the Morrey space of all $u \in L^{p}\left(\Omega \right)$ such that 
		$$ \lVert u \rVert_{\mathrm{L}^{p,\lambda}\left(\Omega\right)}^{p} := \sup_{\substack{ x_{0} \in \Omega,\\ \rho >0 }} 
		\rho^{-\lambda} \int_{B_{\rho}(x_{0}) \cap \Omega} \lvert u \rvert^{p} < \infty, $$ endowed with the norm 
		$ \lVert u \rVert_{\mathrm{L}^{p,\lambda}}$ and $\mathcal{L}^{p,\lambda}\left(\Omega\right) $ denotes the Campanato space of all $u \in L^{p}\left(\Omega\right)$ such that 
		$$ [u ]_{\mathcal{L}^{p,\lambda}\left(\Omega\right)}^{p} := \sup_{\substack{ x_{0} \in \Omega,\\  \rho >0 }} 
		\rho^{-\lambda} \int_{B_{\rho}(x_{0}) \cap \Omega} \lvert u  - (u)_{ \rho , x_{0},\Omega}\rvert^{p} < \infty, $$ endowed with the norm 
		$ \lVert u \rVert_{\mathcal{L}^{p,\lambda}} := \lVert  u \rVert_{L^{p}} +  
		[u ]_{\mathcal{L}^{p,\lambda}}.$ Here 
		\begin{align*}
			(u)_{ \rho , x_{0},\Omega} = \frac{1}{ \left\lvert \left( B_{\rho}(x_{0}) \cap \Omega \right) \right\rvert}\int_{B_{\rho}(x_{0}) \cap \Omega} u  = \fint_{B_{\rho}(x_{0}) \cap \Omega} u .
		\end{align*}
	\end{definition}
	These definitions are extended componentwise for vector-valued functions. We record the following facts, which are proved exactly in the same manner as their Euclidean counterparts in view of \eqref{measure of balls}. 
	\begin{proposition}\label{Morrey-Campanato theorem}
		For $0 \leq \lambda < Q,$ we have 
		\begin{align*}
			\mathrm{L}^{p,\lambda}\left(\Omega\right) \simeq	\mathcal{L}^{p,\lambda}\left(\Omega\right) \qquad \text{ with equivalent norms}.\end{align*}
		For $Q< \lambda \leq Q+p, $ we have 
		\begin{align*}
			C^{0, \frac{\lambda -Q}{p}}\left(\Omega\right) \simeq	\mathcal{L}^{p,\lambda}\left(\Omega\right) \qquad \text{ with equivalent seminorms}.\end{align*}
	\end{proposition}
	When $p > Q,$ we also have the following analogue of Morrey's theorem. 
	\begin{proposition}\label{Morrey's theorem}
		Let $\Omega \subset \mathbb{H}_{n}$ be open and let $u \in HW^{1,p}_{\text{loc}}\left( \Omega\right)$ with $p > Q.$ Then $u \in C^{0, \frac{p-Q}{p}}\left( \Omega\right)$ and for any metric ball $B_{2R} \subset \Omega,$ we have the estimate 
		\begin{align*}
			\left\lVert u \right\rVert_{C^{0,\frac{p-Q}{p} }\left( \overline{B_{R/2}}\right)} \leq C(R, Q, s) \left\lVert u \right\rVert_{HW^{1,p}\left(\Omega\right)}. 
		\end{align*} 
	\end{proposition}
	\begin{proof}
		By the Poincar\'{e} inequality \eqref{poincarewithmeans}, we have 
		\begin{align*}
			\int_{B_{\rho}} \left\lvert u - (u)_{\rho} \right\rvert^{p} \leq c\rho^{p} \int_{B_{\rho}} \left\lvert \X u \right\rvert^{p} \leq c\rho^{p} \int_{B_{R}} \left\lvert \X u \right\rvert^{p},
		\end{align*}
		for any $0 < \rho < R,$ whenever the ball  $B_{2R} \subset \Omega.$ This estimate implies, in the standard way,  that $u \in \mathcal{L}^{p,p}_{\text{loc}}\left(\Omega\right)$. Since $p >Q,$ the second conclusion of Proposition \ref{Morrey-Campanato theorem} implies the result. 
	\end{proof}
	For details on classical Morrey and Campanato spaces, we refer to 
	\cite{giaquinta-martinazzi-regularity} and for the sub-elliptic setting we refer to \cite{Capogna_Danielli_book}.
	
	\subsection{Lorentz spaces and a series}\label{Lorentz space section} 
	\begin{definition}[Lorentz spaces]
		For $1\leq p <  \infty,$ and $0 < \theta \leqslant \infty,$ A measurable function $u: \Omega \rightarrow \mathbb{R}$ is said to belong to the Lorentz space $L^{(p,\theta)} \left(\Omega\right)$ if 
		$$ \left\lVert u \right\rVert_{L^{(p,\theta)}\left(\Omega\right)}^{\theta} 
		:= \int_{0}^{\infty} \left( t^{p} \left\lvert \left\lbrace x \in \Omega : \left\lvert u \right\rvert > t \right\rbrace\right\rvert \right)^{\frac{\theta}{p}} \frac{\mathrm{d}t}{t} < \infty $$
		for $0 < \theta < \infty$ and if 
		$$ \left\lVert u \right\rVert_{L^{(p,\infty)}\left(\Omega\right)}^{p} 
		:= \sup_{t >0} \left( t^{p} \left\lvert \left\lbrace x \in \Omega : \left\lvert u \right\rvert > t \right\rbrace\right\rvert \right) < \infty .$$
	\end{definition}
	The quantities $\left\lVert u \right\rVert_{L^{(p,\theta)}\left(\Omega\right)}$ and $\left\lVert u \right\rVert_{L^{(p,\infty)}\left(\Omega\right)}$ are not norms. They are quasinorms which makes the Lorentz space a quasi-Banach space. 
	For different properties of Lorentz spaces, see \cite{Hunt_LorentzSpaces}, \cite{SteinWeiss_Fourieranalysis}.  Now we need the following crucial result, which is proved analogously to Lemma 1 in \cite{KuusiMingione_nonlinearStein}.  
	\begin{lemma}\label{Lorenz series estimate lemma}
		Let $f \in L^{(Q,1)}\left(\mathbb{H}_{n}\right)$ and for any $1 < q < Q,$ define 
		\begin{align*}
			S_{q}\left( x_{0}, r, \sigma \right):= \sum\limits_{j=1}^{\infty} r_{j} \left( \fint_{ B_{j}} \left\lvert f \right\rvert^{q} \right)^{\frac{1}{q}}, 
		\end{align*}
		where $x_{0} \in \mathbb{H}_{n},$ $\sigma \in (0, 1/4),$ $ r>0$ and $$B_{j}:= B\left(x_{0}, r_{j}\right)\qquad \text{ and } \qquad r_{j}:= \sigma^{j}r. $$
		Then for any $1 < q< Q,$ there exists a constant $c \equiv c\left( Q, q \right)>0$ such that for any $r >0$ and any $\sigma \in (0, 1/4),$ we have 
		\begin{align}\label{Lorentz series estimate}
			S_{q}\left( x_{0}, r, \sigma \right) \leq c \sigma^{1-\frac{Q}{q}}\left\lVert f \right\rVert_{L^{(Q,1)}\left(\mathbb{H}_{n}\right)}. 
		\end{align}
	\end{lemma}
\subsection{Sobolev inequalities}

	\noindent We would need the Sobolev inequality for Heisenberg groups, which we write in a scale-invariant form. 
	\begin{proposition}[Scale-invariant Sobolev inequality] 
		Let $B_{R} \subset \mathbb{H}_{n}$ be any ball of radius $R>0.$ Let $1 \leq s <  Q.$ Then there exists a constant $c >0,$ depending only on $n$ and $s$ such that 	
		\begin{align}\label{scale invariant Sobolev}
			\left( \int_{B_{R}}\left\lvert u \right\rvert^{\frac{Qs}{Q-s}}\right)^{\frac{Q-s}{Qs}} \leq c \left( \int_{B_{R}}\left\lvert \X u \right\rvert^{s} + \frac{1}{R^{s}}\int_{B_{R}}\left\lvert u \right\rvert^{s}\right)^{\frac{1}{s}},  
		\end{align}	
		for any $u \in HW^{1,s}\left( B_{R} \right).$
	\end{proposition}
	We would need the following Poincar\'{e} inequality.
	\begin{proposition}[Poincar\'{e} inequality with means]\label{poincarewithmeans}
		Let $B_{R} \subset \mathbb{H}_{n}$ be any ball of radius $R>0.$ Let $1 \leq s <  \infty.$ Then  there exists a constant $c >0,$ depending only on $n$ and $s$ such that for any $u \in HW^{1,s}\left( B_{R} \right),$ we have
		\begin{equation}\label{poincareineqwithmeans}
			\left( \int_{B_{R}} \lvert u - \left( u\right)_{B_{R}} \rvert^{s} \right)^{\frac{1}{s}} \leq c R \left( \int_{B_{R}} \lvert \X u \rvert^{s} \right)^{\frac{1}{s}}. 
		\end{equation}
	\end{proposition} 
	We would also use the following Poincar\'{e}-Sobolev type inequalities. 
	\begin{proposition}[Poincar\'{e}-Sobolev inequality]\label{poincaresobolev}
		Let $B_{R} \subset \mathbb{H}_{n}$ be any ball of radius $R>0.$ Let $1 \leq s <  Q.$ Then there exists a constant $c >0,$ depending only on $n$ and $s$ such that for any $u \in HW^{1,s}\left( B_{R} \right),$ we have
		\begin{equation}\label{poincaresobolevineq}
			\left( \fint_{B_{R}} \lvert u \rvert^{\frac{Qs}{Q-s}} \right)^{\frac{Q-s}{Qs}} \leq c R \left( \fint_{B_{R}} \lvert \X u \rvert^{s} \right)^{\frac{1}{s}}. 
		\end{equation}
	\end{proposition}
	
	\begin{proposition}[Poincar\'{e}-Sobolev inequality with means]\label{poincaresobolevwithmeans}
		Let $B_{R} \subset \mathbb{H}_{n}$ be any ball of radius $R>0.$ Let $1 \leq s <  Q.$ Then there exists a constant $c >0,$ depending only on $n$ and $s$ such that for any $u \in HW^{1,s}\left( B_{R} \right),$ we have
		\begin{equation}\label{poincaresobolevineqwithmeans}
			\left( \fint_{B_{R}} \lvert u - \left( u\right)_{B_{R}} \rvert^{\frac{Qs}{Q-s}} \right)^{\frac{Q-s}{Qs}} \leq c R \left( \fint_{B_{R}} \lvert \X u \rvert^{s} \right)^{\frac{1}{s}}. 
		\end{equation}
	\end{proposition}
	
	The scale invariant forms of  \eqref{scale invariant Sobolev}, \eqref{poincareineqwithmeans}, \eqref{poincaresobolevineq}, \eqref{poincaresobolevineqwithmeans} can be derived by using their corresponding versions when $R=1$, along with \eqref{dilation} and \eqref{Balls under dialtion}. For the case $R=1$, we can use the Poincar\'e inequality established in \cite{jerision-poincare} and the global density result in \cite[Theorem 1.7]{Garofalo-Nhieu-Global-Approximation} to derive \eqref{poincareineqwithmeans}. Additionally, in this case, by using \cite[Theorem 2.1]{lu-poincare-sobolev} and \cite[Theorem 1.7]{Garofalo-Nhieu-Global-Approximation} we can obtain  \eqref{scale invariant Sobolev}, \eqref{poincaresobolevineq} and \eqref{poincaresobolevineqwithmeans}.

	\subsection{Linear subelliptic Stein theorem}\label{Linear subelliptic Stein}
	A borderline case of the Sobolev inequality is the following result. 
	\begin{theorem}\label{linear Stein-Sobolev embedding version}
		Let $u \in L^{1}\left( \Omega\right)$ be such that $\Xu \in L^{(Q,1)}_{\text{loc}} \left( \Omega; \mathbb{R}^{2n}\right),$ where $\Omega \subset \mathbb{H}_{n}$ is open. Then $u$ is continuous in $\Omega.$
	\end{theorem}
	A PDE version of the above result is the following, which can be termed as the linear Stein theorem for Heisenberg groups. 
	\begin{theorem}\label{linear Stein theorem}
		Let $\Omega \subset \mathbb{H}_{n}$ be open. Let $u \in HW^{1,2}_{\text{loc}}\left( \Omega\right)$ be such that 
		\begin{align*}
			\operatorname{div}_{\mathbb{H}} \X u \in L^{(Q,1)}_{\text{loc}} \left( \Omega\right). 
		\end{align*}
		Then $\X u$ is continuous in $\Omega.$
	\end{theorem}
	\begin{proof}
		We prove only the case $\Omega = \mathbb{H}_{n}$ under the additional assumption that $u \in HW^{1,2}\left( \mathbb{H}_{n}\right),$  $\operatorname{div}_{\mathbb{H}} \X u \in L^{(Q,1)} \left( \mathbb{H}_{n}\right).$ Since our statements are local, the general case follows from this by usual localization arguments and bootstrapping the $L^{p}$ estimates ( \cite[page 917]{folland-cz-linear}). Now the main point is that the subLaplacian has a fundamental solution ( see \cite{Folland_subelliptic} ), we call it $\mathcal{N}$ in analogy with the Newtonian potential. Thus, we can write 
		\begin{align*}
			u = \mathcal{N}\ast \left( 	\operatorname{div}_{\mathbb{H}} \X u\right)	 \qquad \text{ in } \mathbb{H}_{n}. 
		\end{align*}
		This implies,  
		\begin{align*}
			\X u = \left( \X \mathcal{N} \right) \ast \left( 	\operatorname{div}_{\mathbb{H}} \X u\right)	 \qquad \text{ in } \mathbb{H}_{n}. 
		\end{align*}
		Note that in both the equations above, the convolution is the group convolution in $\mathbb{H}_{n}.$ Now, the kernel $K = \X \mathcal{N}$ is homogeneous of degree $- \left( Q-1\right)$ and is essentially the analogue of the kernels of the Riesz potentials $I_{1}$ in the Euclidean case. Thus, one can easily prove ( see Folland \cite{Folland_subelliptic} )  
		$K \in L^{(\frac{Q}{Q-1}, \infty )}\left( \mathbb{H}_{n}\right).$ Since $\operatorname{div}_{\mathbb{H}} \X u \in L^{(Q,1)} \left( \mathbb{H}_{n}\right)$ by assumption and $L^{(Q,1)}$ is the dual space of $L^{(\frac{Q}{Q-1}, \infty )},$ for every fixed $y \in \mathbb{H}_{n},$ we have 
		\begin{align*}
			\int_{\mathbb{H}_{n}} K\left(x^{-1}\cdot y\right)\left[ \operatorname{div}_{\mathbb{H}} \X u \right]\left(y\right)\ \mathrm{d}y \leq c \left\lVert K \right\rVert_{L^{(\frac{Q}{Q-1}, \infty )}\left( \mathbb{H}_{n}\right)} \left\lVert \operatorname{div}_{\mathbb{H}} \X u \right\rVert_{L^{(Q,1)}\left( \mathbb{H}_{n}\right)}.  
		\end{align*}
		Thus, the map 
		\begin{align*}
			x \mapsto K \ast \left( \operatorname{div}_{\mathbb{H}} \X u\right) \left( x\right)
		\end{align*} is bounded and consequently, continuous. This completes the proof. 
	\end{proof}
	\begin{proof}[Proof of Theorem \ref{linear Stein-Sobolev embedding version}] The proof is the same as above. We skip the details.	
	\end{proof}
	\subsection{Minimization and weak formulation}
	\begin{proposition}\label{minimizerexistenceprop}
		Let $\Omega \subset \mathbb{H}_{n}$ be open and let $a:\Omega \rightarrow [\gamma, L]$, where $ 0 < \gamma < L < \infty,$ be a measurable map. Let $1 < p < \infty.$ Then for any ball $B_{R} \subset \subset \Omega$ and any given function $u_{0} \in HW^{1,p}\left( B_{R}\right),$ the boundary value problem 
		\begin{align}
			\left\lbrace \begin{aligned}
				\operatorname{div}_{\mathbb{H}} \left( a(x) \left\lvert \X  u \right\rvert^{p-2} \Xu \right) &= 0 &&\text{ in } B_{R}, \\
				u &= u_{0}  &&\text{ on } \partial B_{R}, 
			\end{aligned}\right.
		\end{align}
		admits a unique weak solution $u \in u_{0} + HW^{1,p}_{0}\left( B_{R}\right). $ Moreover, the solution $u$ is the unique minimizer to the minimization problem 
		$$m = \inf\left\lbrace \frac{1}{p}\int_{B_{R}}  a(x)\left\lvert \X u \right\rvert^{p} : u \in u_{0} + HW^{1,p}_{0}\left( B_{R}\right) \right\rbrace . $$ \end{proposition}
	\begin{proof}
		In view of the Poincar\'{e}-Sobolev inequality \eqref{poincaresobolevineq}, the proof in the Euclidean case works mutadis mutandis. 
	\end{proof}
	\subsection{The auxiliary mapping \texorpdfstring{$V$}{V}}
	As in the Euclidean case, we need an auxialiary mapping $V$ to deal with the structure of the $p$-Laplace equation. We define the mapping $V: \mathbb{R}^{2n} \rightarrow \mathbb{R}^{2n}$ by 
	\begin{equation}\label{definition V}
		V(z) : = \left\lvert z \right\rvert^{\frac{p-2}{2}} z , 
	\end{equation}
	which is a locally Lipschitz bijection from $\mathbb{R}^{2n}$ into itself. 
	We summarize the relevant properties of the map in the following. 
	\begin{lemma}\label{prop of V}
		For any $p >1$, there exists a constant $c_{V} \equiv c_{V}(n,p) > 0$ such that 
		\begin{equation}\label{constant cv}
			\frac{\left\lvert z_{1} - z_{2}\right\rvert}{c_{V}} \leq  \frac{\left\lvert V(z_{1}) - V(z_{2})\right\rvert}{\left( \left\lvert z_{1} \right\rvert + 
				\left\lvert z_{2}\right\rvert \right)^{\frac{p-2}{2}}} \leq c_{V}\left\lvert z_{1} - z_{2}\right\rvert,
		\end{equation}
		for any $z_{1}, z_{2} \in \mathbb{R}^{2n},$ not both zero. This implies the classical monotonicity estimate 
		\begin{equation}\label{monotonicity}
			\left( \left\lvert z_{1} \right\rvert + \left\lvert z_{2}\right\rvert \right)^{p-2} \left\lvert z_{1} - z_{2}\right\rvert^{2} \leq c_{M} 
			\left\langle  \left\lvert z_{1} \right\rvert^{p-2} z_{1} - \left\lvert z_{2} \right\rvert^{p-2} z_{2}, z_{1} - z_{2} \right\rangle ,  
		\end{equation}
		with a constant $c_{M} = c_M(n, p) > 0$ for all $p >1$ and all $z_{1}, z_{2} \in \mathbb{R}^{2n}.$
		Moreover, if $1 < p \leq 2,$ there exists a constant $c = c(n, p) > 0$ such that for any  $z_{1}, z_{2} \in \mathbb{R}^{2n},$
		\begin{equation}\label{v estimate p less 2}
			\left\lvert z_{1} - z_{2}\right\rvert \leq c \left\lvert V(z_{1}) - V(z_{2})\right\rvert^{\frac{2}{p}} + c \left\lvert V(z_{1}) - V(z_{2})\right\rvert 
			\left\lvert z_{2}\right\rvert^{\frac{2-p}{2}}.
		\end{equation}
	\end{lemma}
	The estimates \eqref{constant cv} and \eqref{monotonicity} are classical (cf. Lemma 2.1, \cite{hamburgerregularity}). The estimate \eqref{v estimate p less 2} follows from this (cf. Lemma 2, \cite{KuusiMingione_nonlinearStein}). 
\begin{notation}
We use the symbol $c$ to denote  a generic positive constant $c>0$  in our estimates, unless specifically indicated otherwise,  whose value can change from a line to the next. 
\end{notation}	
	\section{Excess decay for subelliptic \texorpdfstring{$p$}{p}-harmonic functions}\label{excess decay section}
	In this section, we are concerned with regularity results for homogeneous $p$-subLaplace equation in Heisenberg groups. Namely, we shall consider 
	\begin{align}\label{equation:main}
		\operatorname{div}_{\mathbb{H}} \left( \left\lvert \X u \right\rvert^{p-2}\X u \right) &= 0 &&\text{ in } \Omega. 
	\end{align} 
	We note that the equation can also be written as 
	\begin{equation}\label{equation:main alternate form}
		\operatorname{div}_{\mathbb{H}} \big(Df(\X u)\big)=\sum_{i=1}^{2n}X_i\big(D_i f(\X u)\big)=0 \qquad \text{ in } \Omega,
	\end{equation}
	where $Df=(D_1 f, D_2f,\ldots,D_{2n}f)$ is the Euclidean gradient
	of $f$ and $f$ is given by  \begin{align*}
		f\left(z\right):= \frac{1}{p} \left\lvert z \right\rvert^{p}, \quad \forall z\in \mathbb{R}^{2n}.
	\end{align*}
	We compute 
	\begin{equation}\label{hessian of f}
		D_j D_i f(z) = \lvert z\rvert^{p-2} \delta_{ij} + (p-2)\lvert z\rvert^{p-4} z_i z_j, \text{ for } i, j = 1,\dots, 2n \text{ and } z\in \mathbb{R}^{2n},
	\end{equation}
	where $\delta_{ij}$ denotes the Kronecker delta function.
	Hence, we deduce  
	\begin{align}\label{coercivity}
		\sum\limits_{i,j=1}^{2n} 	D_{j}D_{i}f\left(z\right) \eta_{j} \eta_{i} \geq c \left\lvert z \right\rvert^{p-2}\left\lvert \eta \right\rvert^{2} \qquad \text{ for all } z, \eta \in \mathbb{R}^{2n}. 
	\end{align}
	\subsection{\texorpdfstring{H\"{o}lder}{Holder} continuity results}
	We begin by recording the known results that we would use. The following theorem is due to Zhong \cite{zhong2018regularityvariationalproblemsheisenberg}. 
	\begin{theorem}\label{thm:lip}
		Let $1<p<\infty$ and $u\in HW^{1,p}(\Omega)$ be a weak solution of
		equation \eqref{equation:main}. Then $\Xu \in
		L^\infty_{\text{loc}}(\Omega;{\mathbb R}^{2n})$. Moreover, for any
		ball $B_{2r}\left(x_{0}\right)\subset \Omega$, we have that
		\begin{equation}\label{Xu:bdd}
			\sup_{B_r\left(x_{0}\right)}\vert\X u\vert\le c_{p}\Big(\fint_{B_{2r}\left(x_{0}\right)} \left\lvert \X u \right\rvert^{p}\, dx\Big)^{\frac 1 p},
		\end{equation}
		where $c_{p}>0$ depends only on $n,p$.
	\end{theorem}
By a standard interpolation technique, this last result easily implies the following. 
	\begin{theorem}\label{thm:lip Lq bound}
		Let $1<p<\infty$ and $u\in HW^{1,p}(\Omega)$ be a weak solution of
		equation \eqref{equation:main}. Then $\Xu \in
		L^\infty_{\text{loc}}(\Omega;{\mathbb R}^{2n})$. Moreover, for any $0 < q < \infty$ and any 
		ball $B_{2r}\left(x_{0}\right)\subset \Omega$, we have the estimate
		\begin{equation}\label{Xu:bdd by Lq}
			\sup_{B_r\left(x_{0}\right)}\vert\X u\vert\leq c_{q}\Big(\fint_{B_{2r}\left(x_{0}\right)} \left\lvert \X u \right\rvert^{q}\, dx\Big)^{\frac 1 q},
		\end{equation}
		where $c_{q}>0$ depends only on $n,p$ and $q.$
	\end{theorem}
	The following is due to Zhong \cite{zhong2018regularityvariationalproblemsheisenberg} (for $p\geq 2$) and Mukherjee-Zhong \cite{Mukherjee_Zhong} (for $1< p < 2$).
\begin{theorem}\label{thm:holder}
	Let  $1 < p <\infty$ and $u\in HW^{1,p}(\Omega)$ be a weak
	solution of equation \eqref{equation:main}. Then the horizontal gradient $\X u$ is locally H\"older continuous. Moreover, there is a positive exponent $\alpha_{1}=\alpha_{1}(n,p) < 1$ such that for any
	ball $B_{r_0}\subset \subset \Omega$ and any $0<r\le r_0$, we have
	\begin{equation}\label{Xu:holder} \left\lvert \Xu \left(x\right) - \Xu \left(y\right)\right\rvert \leq
		c_{h}\Big(\frac{r}{r_0}\Big)^{\alpha_{1}} \left\lVert \X u \right \rVert _{\infty, r_0}, \quad \text{ for all } x, y \in B_{r},\end{equation} where $c_{h}>0$ depends only on $n,p$, where $\left\lVert \X u \right \rVert _{\infty, r_0} := \sup_{B_{r_0}}\vert\X u\vert.$
\end{theorem}
Theorem \ref{thm:holder} and Theorem \ref{thm:lip} easily implies the following result. 
\begin{theorem}\label{thm:holder V}
	Let  $1 <  p <\infty$ and $u\in HW^{1,p}(\Omega)$ be a weak
	solution of equation \eqref{equation:main}
	Then for any ball $B_{r_0}\subset \subset \Omega$ and any $0<r\le r_0$, we have
	\begin{equation}\label{Xu:holder V} \sup\limits_{x, y \in B_{r}}\left\lvert V \left( \Xu  \right)\left(x\right) - V \left( \Xu \right) \left(y\right)\right\rvert \leq
		c_{h,V}\Big(\frac{r}{r_0}\Big)^{\beta_{1}} \left\lVert \X u \right \rVert _{\infty, r_0}^{\frac{p}{2}}, 
	\end{equation} where the constants $c_{h,V}>0$ and $\beta_{1} \in (0,1),$ both depends only on $n$ and $p$. 
\end{theorem}
\begin{proof}
	We have 
	\begin{align*}
		&\sup\limits_{x, y \in B_{r}}\left\lvert V \left( \Xu  \right)\left(x\right) - V \left( \Xu \right) \left(y\right)\right\rvert^{2} \\&\stackrel{\eqref{constant cv}}{\leq} 	c_{V}\sup\limits_{x, y \in B_{r}} \left[ \left( \left\lvert \X u \left(x\right)\right\rvert  + \left\lvert \X u \left(y\right)\right\rvert \right)^{p-2}\left\lvert  \Xu  \left(x\right) - \Xu \left(y\right)\right\rvert^{2}\right]  \\
		&\stackrel{\eqref{monotonicity}}{\leq} c_{M}c_{V} \sup\limits_{x, y \in B_{r}} \left[ \left\langle  \left\lvert \X u \left(x\right)\right\rvert^{p-2}\X u \left(x\right)  - \left\lvert \X u \left(y\right)\right\rvert^{p-2}\X u \left(y\right), \X u \left(x\right)  - \X u \left(y\right)  \right\rangle  \right] \\
		&\leq c_{M}c_{V} \sup\limits_{x, y \in B_{r}} \left[ \left\lvert \left\lvert \X u \left(x\right)\right\rvert^{p-2}\X u \left(x\right)  - \left\lvert \X u \left(y\right)\right\rvert^{p-2}\X u \left(y\right)\right\rvert \left\lvert \X u \left(x\right)  - \X u \left(y\right)  \right\rvert  \right] \\
		&\leq c_{M}c_{V} \sup\limits_{x, y \in B_{r}} \left[ \left( \left\lvert \X u \left(x\right)\right\rvert^{p-1} + \left\lvert \X u \left(y\right)\right\rvert^{p-1} \right)\left\lvert \X u \left(x\right)  - \X u \left(y\right)  \right\rvert  \right] \\
		&\leq 2^{p-1}c_{M}c_{V}\left\lVert  \X u \right\rVert_{\infty, r_{0}}^{p-1} \left( \sup\limits_{x, y \in B_{r}} \left\lvert  \Xu  \left(x\right) - \Xu \left(y\right)\right\rvert\right) \\
		&\stackrel{\eqref{Xu:holder}}{\leq} 2^{p-1}c_{M}c_{V}c_{h}\left( \frac{r}{r_{0}}\right)^{\alpha_{1}} \left\lVert \X u \right \rVert _{\infty, r_0}^{p}.
	\end{align*}
	Thus, setting $\beta_{1} := \alpha_{1}/2,$ we have 
	\begin{align*}
		\sup\limits_{x, y \in B_{r}}\left\lvert V \left( \Xu  \right)\left(x\right) - V \left( \Xu \right) \left(y\right)\right\rvert &\leq c \left( \frac{r}{r_{0}}\right)^{\beta_{1}}\left\lVert \X u \right \rVert _{\infty, r_0}^{\frac{p}{2}}. 
	\end{align*}
	This completes the proof. 
\end{proof}
\subsection{Caccioppoli inequalities in the nondegenerate regime}
We begin by deriving some Caccioppoli type inequalities under the assumption that $\left\lvert \X u \right\rvert $ is both bounded above and below in a fixed scale $\lambda >0.$
\begin{lemma}\label{Cac in the nondeg regime}
	Let $1 < p < \infty$ and let $\Omega \subset \mathbb{H}_{n}$ be open. Let $u \in HW^{1,p}_{\text{loc}}\left(\Omega\right)$ be a weak solution to 
	\begin{align*}
		\operatorname{div}_{\mathbb{H}} \left( \left\lvert \X u\right\rvert^{p-2}\X u\right) &= 0 &&\text{ in } \Omega. 
	\end{align*}
	Let $\tilde{\Omega} \subset \subset \Omega$ be any open subset. Suppose we have \begin{align}\label{two sided bound}
		\frac{\lambda}{4B} \leq |\Xu| \leq A\lambda \qquad \text{ in } \tilde{\Omega}, 
	\end{align}
	for some constants $A, B \geq 1$ and $\lambda >0.$ Let $B_{\rho} \subset B_{R} \subset \tilde{\Omega}$ be concentric balls of radius $0 < \rho < R <1 .$ Let $\eta \in C_{c}^{\infty}\left(B_{R}\right)$ satisfy $0 \leq \eta \leq 1$ in $B_{R}$ with $\eta \equiv 1$ in $B_{\rho}$ and $\left\lvert \X \eta \right\rvert \leq c/\left( R - \rho\right)$  for some absolute constant $c>1.$ Then there exists a constant $C_{0} \equiv C_{0}\left( A, B, n, p \right) \geq 1,$ but independent of $\lambda$, $R$, $\rho$ or $u,$ such that for any $\zeta \in \mathbb{R}$ and any $\xi \in \mathbb{R}^{2n},$ we have the following inequalitites. 
	\begin{align}
		\int_{B_{R}} \eta^{4} \left\lvert \X T u \right\rvert^{2}\ \mathrm{d}x &\leq \frac{C}{\left( R - \rho\right)^{2}} 	\int_{B_{R}} \eta^{2} \left\lvert Tu -\zeta \right\rvert^{2}, \label{Cac for XTu by Tu}
	\end{align}
	\begin{align}
		\int_{B_{R}}\eta^{4}\left\lvert \X Tu \right\rvert^{2}\ \mathrm{d}x &\leq \frac{C}{\left( R - \rho\right)^4}\int_{B_{R}}\left\lvert  \X u- \xi\right\rvert^2\ \mathrm{d}x , \label{XTu estimate nondegenerate} \\
		\int_{B_{R}}\eta^{2}\left\lvert Tu \right\rvert^{2}\ \mathrm{d}x &\leq \frac{C}{\left( R - \rho\right)^{2}} \int_{B_{R}}\left\lvert  \X u- \xi\right\rvert^2\ \mathrm{d}x,  \label{Tu estimate nondegenerate}\\
		\int_{B_{R}}\eta^{2}\lvert\X \X u\rvert^2\ \mathrm{d}x &\leq \frac{C}{\left( R - \rho\right)^{2}}\int_{B_{R}}\left\lvert  \X u- \xi\right\rvert^2\ \mathrm{d}x.  \label{estimate of XXu in nondegenerate case}
	\end{align}
\end{lemma}
\begin{proof}
	Note that by \eqref{two sided bound} and Capogna's regularity result in \cite{Capogna_regularity}, we have 
	\begin{align*}
		\X u \in HW^{1,2}_{\text{loc}}\left( B_{R}; \mathbb{R}^{2n}\right) \text{ and } Tu \in  HW^{1,2}_{\text{loc}}\left( B_{R}\right) \cap L^{\infty}_{\text{loc}}\left( B_{R}\right). 
	\end{align*}
	One can easily check that these regularity assertions are enough to justify our calculations in this lemma. Observe also that it is easy to verify using integration by parts and the definition of weak derivative that 
	\begin{align*}
		Tu \in  HW^{1,2}_{\text{loc}}\left( B_{R}\right) &\Rightarrow \X Tu \in L^{2}_{\text{loc}}\left( B_{R}; \mathbb{R}^{2n}\right) \\&\Rightarrow T \X u \in L^{2}_{\text{loc}}\left( B_{R}; \mathbb{R}^{2n}\right) \text{ as well and } \X Tu = T\X u. 
	\end{align*} We would not comment on this further. 
	\begin{claim}
		For every $l=1,2,\ldots,n,$ $\psi_{l}:= X_{l}u$ is a weak solution of 	\begin{equation}\label{equation:horizontal}
			\sum_{i,j=1}^{2n}X_i\big(D_{j}D_if(\mathfrak X u)
			X_j\psi_l\big)+\sum_{i=1}^{2n}X_i\big(D_{n+l}D_i f(\mathfrak X u)Tu\big)+T\big(D_{n+l}f(\mathfrak X
			u)\big)=0, 
		\end{equation}
		and  $\psi_{n+l}=X_{n+l}u$ is a weak solution of
		\begin{equation}\label{equation:horizontal2}
			\sum_{i,j=1}^{2n}X_i\big(D_{j}D_if(\mathfrak X u)
			X_j\psi_{n+l}\big)-\sum_{i=1}^{2n}X_i\big(D_{l}D_i f(\mathfrak X u)Tu\big)-T\big(D_{l}f(\mathfrak X
			u)\big)=0. 
		\end{equation}
		Furthermore, $Tu$ is a weak solution of 
		\begin{equation}\label{equation:T}
			\sum_{i,j=1}^{2n} X_i\big(D_jD_i f(\mathfrak X u)X_j (Tu)\big)=0.
		\end{equation}
	\end{claim}
	This can be verified by a direct calculation, see Lemma 3.1 and Lemma 3.2 in \cite{zhong2018regularityvariationalproblemsheisenberg}. Now we plugg $\varphi=\eta^4\left( Tu- \zeta \right)$ with $\zeta \in \mathbb{R}$  as a test function in the weak formulation of \eqref{equation:T} to deduce 
	\begin{multline*}
		0 =\int_{\Omega}	\eta^{4}\sum_{i,j=1}^{2n} D_jD_i f(\mathfrak X u)X_j (Tu)X_{i}\left(Tu\right)\, dx \\+ 4\int_{\Omega}	\eta^{3} \sum_{i,j=1}^{2n} D_jD_i f(\mathfrak X u)X_j (Tu)X_{i}\eta\left( Tu- \zeta \right)\, dx.
	\end{multline*}
	Hence, using \eqref{coercivity} and Young's inequality with $\varepsilon>0,$ we have 
	\begin{align*}
		\int_{\Omega}&\eta^{4}|\Xu|^{p-2}\left\lvert \X Tu \right\rvert^{2}\, dx \\ & \stackrel{\eqref{hessian of f}}{\leq} c\int_{\Omega}\eta^{3}|\Xu|^{p-2}\left\lvert \X \eta \right\rvert \left\lvert \X Tu \right\rvert \left\lvert Tu- \zeta \right\rvert\, dx \\
		&\leq c \varepsilon 	\int_{\Omega}\eta^{4}|\Xu|^{p-2}\left\lvert \X Tu \right\rvert^{2}\, dx + 	C_{\varepsilon}\int_{\Omega}\eta^{2}|\Xu|^{p-2} \left\lvert \X \eta \right\rvert^{2}\left\lvert Tu- \zeta  \right\rvert^{2}\, dx. 
	\end{align*}
	Choosing $\varepsilon>0$ small enough, we have 
	\begin{align*}
		\int_{\Omega}\eta^{4}|\Xu|^{p-2}\left\lvert \X Tu \right\rvert^{2}\, dx \leq C \int_{\Omega}\eta^{2}\left\lvert \X \eta \right\rvert^{2}|\Xu|^{p-2} \left\lvert Tu- \zeta  \right\rvert^{2}\ \mathrm{d}x.
	\end{align*}
	In view of \eqref{two sided bound}, this immediately implies \eqref{Cac for XTu by Tu}. Now to derive \eqref{XTu estimate nondegenerate}, we begin by deducing a crucial estimate for $Tu.$ For any $\xi \in \mathbb{R}^{2n},$ integrating by parts and using Young's inequality with $\varepsilon>0,$ we have 
	\begin{align*}
		\int_{\Omega} &\eta^{2} \left\lvert Tu \right\rvert^{2}\, dx \\&= 	\int_{\Omega} \eta^{2}  Tu \left( X_{l}X_{n+l}u - X_{n+l}X_{l}u\right)\, dx \\
		&= 	\int_{\Omega} \eta^{2}  Tu \left[  X_{l}\left( X_{n+l}u - \xi_{n+l}\right) - X_{n+l}\left( X_{l}u - \xi_{l} \right) \right]\, dx \\
		&=  -\int_{\Omega} X_{l}\left( \eta^{2}  Tu \right)\left( X_{n+l}u - \xi_{n+l}\right) \, dx  +\int_{\Omega} X_{n+l}\left( \eta^{2}  Tu \right)\left( X_{l}u - \xi_{l}\right) \, dx \\
		&\begin{multlined}[t]
			= -2\int_{\Omega} \eta X_{l} \eta \left(  Tu \right)\left( X_{n+l}u - \xi_{n+l}\right) \, dx - \int_{\Omega}\eta^{2} X_{l} Tu \left( X_{n+l}u - \xi_{n+l}\right) \, dx \\ +2\int_{\Omega} \eta X_{n+l} \eta \left(  Tu \right)\left( X_{l}u - \xi_{l}\right) \, dx + \int_{\Omega}\eta^{2} X_{n+l} Tu \left( X_{l}u - \xi_{l}\right) \, dx
		\end{multlined} \\
		&\leq 4\int_{\Omega} \eta \left\lvert \X \eta \right\rvert \left\lvert Tu\right\rvert \left\lvert \X u - \xi \right\rvert + 2 \int_{\Omega} \eta^{2} \left\lvert \X Tu \right\rvert \left\lvert \X u - \xi \right\rvert \\ 
		&\leq \varepsilon 	\int_{\Omega} \eta^{2} \left\lvert Tu \right\rvert^{2}\, dx + C_{\varepsilon}\int_{\Omega}\left\lvert \X \eta \right\rvert^{2}\left\lvert  \X u- \xi\right\rvert^2\, dx + 2 \int_{\Omega}\eta^{2} \left\lvert \X Tu\right\rvert \left\lvert  \X u- \xi\right\rvert. 
	\end{align*}
	Choosing $\varepsilon =1/2$, this implies 
	\begin{align}\label{Estimate for Tu by XTu and Xu}
		\int_{\Omega} \eta^{2} \left\lvert Tu \right\rvert^{2}\, dx &\leq C \left( \int_{\Omega}\left\lvert \X \eta \right\rvert^{2}\left\lvert  \X u- \xi\right\rvert^2\, dx + \int_{\Omega}\eta^{2} \left\lvert \X Tu\right\rvert \left\lvert  \X u- \xi\right\rvert \right). 
	\end{align}
	Combining this with \eqref{Cac for XTu by Tu} with the choice $\zeta =0$ and using Young's inequality with $\delta >0,$ we arrive at 
	\begin{align*}
		\int_{B_{R}}&\eta^{4}\left\lvert \X Tu \right\rvert^{2}\, dx \\&\leq \frac{C}{\left( R - \rho \right)^{2}} \left( \int_{B_{R}}\left\lvert \X \eta \right\rvert^{2}\left\lvert  \X u- \xi\right\rvert^2\, dx + \int_{B_{R}}\eta^{2} \left\lvert \X Tu\right\rvert \left\lvert  \X u- \xi\right\rvert\, dx \right) \\
		&\begin{multlined}[t]
			\leq\frac{C}{\left( R - \rho \right)^{2}} \int_{B_{R}}\left\lvert \X \eta \right\rvert^{2}\left\lvert  \X u- \xi\right\rvert^2\, dx + C\delta^2\int_{B_{R}}\eta^{4}\left\lvert \X Tu \right\rvert^{2}\, dx \\+ \frac{C}{\left( R - \rho \right)^{4}\delta^2}\int_{B_{R}} \left\lvert  \X u- \xi\right\rvert^2\, dx  . 
		\end{multlined}
	\end{align*}
	Choose $\delta^2 = \frac{1}{2C}$  we have 
	\begin{align*}
		\int_{B_{R}}\eta^{4}\left\lvert \X Tu \right\rvert^{2}\, dx \leq \frac{C}{\left( R - \rho \right)^{2}} \int_{B_{R}}\left\lvert \X \eta \right\rvert^{2}\left\lvert  \X u- \xi\right\rvert^2\, dx + \frac{C}{\left( R - \rho \right)^{4}}\int_{B_{R}} \left\lvert  \X u- \xi\right\rvert^2\, dx  . 
	\end{align*}
	Thus we have proved \eqref{XTu estimate nondegenerate}. On the other hand, using this back in \eqref{Estimate for Tu by XTu and Xu}, we deduce 
	\begin{align*}
		&\int_{B_{R}}\eta^{2} \left\lvert Tu \right\rvert^{2}\, dx \notag\\&\leq C \left( \int_{B_{R}}\left\lvert \X \eta \right\rvert^{2}\left\lvert  \X u- \xi\right\rvert^2\, dx + \int_{B_{R}}\eta^{2} \left\lvert \X Tu\right\rvert \left\lvert  \X u- \xi\right\rvert \right) \notag \\
		&\leq \frac{C}{\left( R - \rho \right)^{2}}\int_{B_{R}}\left\lvert  \X u- \xi\right\rvert^2\, dx + \left( \int_{B_{R}}\eta^{4}\left\lvert \X Tu \right\rvert^{2}\, dx \right)^{\frac{1}{2}}\left( \int_{B_{R}} \left\lvert  \X u- \xi\right\rvert^2\, dx\right)^{\frac{1}{2}} \notag \\
		&\stackrel{\eqref{XTu estimate nondegenerate}}{\leq} \frac{C}{\left( R - \rho \right)^{2}}\int_{B_{R}}\left\lvert  \X u- \xi\right\rvert^2\, dx + \frac{C}{\left( R - \rho \right)^{2}} \int_{B_{R}}\left\lvert  \X u- \xi\right\rvert^2\, dx.
	\end{align*}
	This proves \eqref{Tu estimate nondegenerate}. \smallskip 
	
	Now we finally turn our attention to \eqref{estimate of XXu in nondegenerate case}. Fix $l\in \{ 1, 2,\ldots,n\}$  and use $\varphi=\eta^2\left( X_l u- \xi_{l} \right)$ as a test function in \eqref{equation:horizontal} to deduce 
	\begin{equation}\label{weak5}
		\begin{aligned}[b]
			\int_{B_{R}} \sum_{i,j=1}^{2n}&\eta^2D_jD_i f(\Xu)X_j X_l u X_i\left(X_l u- \xi_{l}\right)\, dx\\
			=&-2\int_{B_{R}}\sum_{i,j=1}^{2n}\eta D_jD_if(\Xu)X_jX_l u X_i\eta\left(X_l u- \xi_{l}\right)\, dx\\
			&-\int_{B_{R}} \sum_{i=1}^{2n} D_{n+l}D_i f(\Xu)TuX_i\big(\eta^2\left(X_l u- \xi_{l}\right)\big)\, dx\\
			&+\int_{B_{R}} T\big(D_{n+l}f(\Xu)\big)\big(\eta^2\left( X_l u- \xi_{l} \right)\big)\, dx:= I_1^l+I_2^l+I_3^l.
		\end{aligned}
	\end{equation}
	Similarly, we can deduce from \eqref{equation:horizontal2} that for all $l\in \{ n+1,\ldots,2n\}$,
	\begin{equation}\label{weak6}
		\begin{aligned}[b]
			\int_{B_{R}} \sum_{i,j=1}^{2n}&\eta^2D_jD_i f(\Xu)X_jX_l uX_i\left( X_l u- \xi_{l} \right)\, dx\\
			=&-2\int_{B_{R}}\sum_{i,j=1}^{2n}\eta D_jD_if(\Xu)X_jX_l uX_i\eta\left( X_l u- \xi_{l} \right)\, dx\\
			&+\int_{B_{R}} \sum_{i=1}^{2n} D_{l-n}D_i f(\Xu)TuX_i\big(\eta^2\left( X_l u- \xi_{l} \right)\big)\, dx\\
			&-\int_{B_{R}} T\big(D_{l-n}f(\Xu)\big)\big(\eta^2\left( X_l u- \xi_{l} \right)\big)\, dx:= I_1^l+I_2^l+I_3^l.
		\end{aligned}
	\end{equation}
	Summing the above equations for all $l$ from $1$ to $2n$, we arrive at
	\begin{equation}\label{weak7}
		\int_{B_{R}} \sum_{i,j, l} \eta^2D_jD_i f(\Xu)X_jX_l uX_i
		\left( X_l u- \xi_{l} \right)\, dx=\sum_{l} \left(I_1^l+I_2^l+I_3^l\right),
	\end{equation}
	where we sum for  $i,j,l$ from 1 to $2n$. By \eqref{coercivity}, we deduce 
	\begin{align}\label{coercivity estimate for Xu eq}
		\int_{{B_{R}}}\eta^{2}|\Xu|^{p-2}\vert\X \X u\vert^2\, dx \leq C \sum_{l} \left(I_1^l+I_2^l+I_3^l\right)
	\end{align}
	Now, we have 
	\begin{align}
		\left\lvert I^{l}_{1}\right\rvert &\leq C \int_{B_{R}} \eta |\Xu|^{p-2}\left\lvert \X \X u\right\rvert \left\lvert \X \eta \right\rvert \left\lvert  \X u- \xi \right\rvert\, dx \notag\\
		&\leq \varepsilon \int_{B_{R}}\eta^{2}|\Xu|^{p-2}\vert\X \X u\vert^2\, dx + C_{\varepsilon}\int_{B_{R}}\left\lvert \X \eta \right\rvert^{2}|\Xu|^{p-2}\left\lvert  \X u- \xi\right\rvert^2\, dx. \label{estimate for Il1}
	\end{align}
	Now for $I^{l}_{2},$ using the fact that $\left\lvert Tu \right\rvert \leq C \left\lvert \X \X u \right\rvert,$ we have 
	\begin{align}
		\left\lvert I^{l}_{2}\right\rvert &\leq C \int_{B_{R}} \eta |\Xu|^{p-2}\left\lvert Tu\right\rvert \left\lvert \X \eta \right\rvert \left\lvert  \X u- \xi \right\rvert\, dx  +  C\int_{B_{R}} \eta^{2} |\Xu|^{p-2}\left\lvert Tu\right\rvert \left\lvert  \X \X u\right\rvert\, dx \notag\\
		&\begin{multlined}[t]
			\leq C \int_{B_{R}} \eta |\Xu|^{p-2}\left\lvert \X \X u\right\rvert \left\lvert \X \eta \right\rvert \left\lvert  \X u- \xi \right\rvert\, dx + \varepsilon \int_{B_{R}}\eta^{2}|\Xu|^{p-2}\vert\X \X u\vert^2\, dx \\ + C_{\varepsilon}\int_{B_{R}}\eta^{2}|\Xu|^{p-2}\left\lvert  Tu\right\rvert^2\, dx
		\end{multlined} \notag \\
		&\begin{multlined}[t]
			\leq \varepsilon \int_{\Omega}\eta^{2}|\Xu|^{p-2}\vert\X \X u\vert^2\, dx + C_{\varepsilon}\int_{\Omega}\eta^{2}|\Xu|^{p-2}\left\lvert  Tu\right\rvert^2\, dx \\ + C_{\varepsilon}\int_{\Omega}\left\lvert \X \eta \right\rvert^{2}|\Xu|^{p-2}\left\lvert  \X u- \xi\right\rvert^2\, dx. 
		\end{multlined}\label{estimate for Il2}
	\end{align}
	Now, for $I^{l}_{3}$, when $1 \leq l \leq n, $ we compute 
	\begin{align*}
		T \left( D_{n+l}f(\Xu)\right) &= T \left( \left\lvert \X u \right\rvert^{p-2}\right) X_{n+l}u + \left\lvert \X u \right\rvert^{p-2} T X_{n+l}u \\&=\left( p-2\right)\left\lvert \X u \right\rvert^{p-4}\left\langle \X u, T\X u \right\rangle X_{n+l}u +  \left\lvert \X u \right\rvert^{p-2} T X_{n+l}u.
	\end{align*}
	Similarly, when $n+1 \leq l \leq 2n,$ we have 
	\begin{align*}
		T \left( D_{l-n}f(\Xu)\right) &= T \left( \left\lvert \X u \right\rvert^{p-2}\right) X_{l-n}u + \left\lvert \X u \right\rvert^{p-2} T X_{l-n}u \\&=\left( p-2\right)\left\lvert \X u \right\rvert^{p-4}\left\langle \X u, T\X u \right\rangle X_{l-n}u+  \left\lvert \X u \right\rvert^{p-2} T X_{l-n}u.
	\end{align*}
	Hence, in either case we have,  
	\begin{align}\label{estimate for Il3}
		\left\lvert I^{l}_{3}\right\rvert &\leq C \int_\Omega \eta^{2}|\Xu|^{p-2}\left\lvert \X T u \right\rvert\left\lvert \X u- \xi \right\rvert\, dx.  
	\end{align}
	Combining \eqref{coercivity estimate for Xu eq}, \eqref{estimate for Il1}, \eqref{estimate for Il2} and \eqref{estimate for Il3} and choosing $\varepsilon>0$ sufficiently small, we deduce 
	\begin{align}
		\int_{\Omega}&\eta^{2}|\Xu|^{p-2}\vert\X \X u\vert^2\, dx \notag \\&\begin{multlined}[b]
			\leq C \int_{\Omega}\eta^{2}|\Xu|^{p-2}\left\lvert  Tu\right\rvert^2\, dx  + C\int_{\Omega}\left\lvert \X \eta \right\rvert^{2}|\Xu|^{p-2}\left\lvert  \X u- \xi\right\rvert^2\, dx  \\ +C\int_\Omega \eta^{2}|\Xu|^{p-2}\left\lvert \X T u \right\rvert\left\lvert \X u- \xi \right\rvert\, dx.   
		\end{multlined} \label{estimate for XXu}
	\end{align} 
	Now, combining \eqref{estimate for XXu}, \eqref{XTu estimate nondegenerate},  \eqref{Tu estimate nondegenerate} and \eqref{two sided bound}, we deduce 
	\begin{align*}
		\int_{B_{R}}&\eta^{2}\lvert\X \X u\rvert^2\, dx \notag \\
		&\leq \frac{C}{\left( R - \rho \right)^{2}}\int_{B_{R}}\left\lvert  \X u- \xi\right\rvert^2\, dx + C\int_{B_{R}}\eta^{2}\left\lvert \X T u \right\rvert\left\lvert \X u- \xi \right\rvert\, dx \notag \\
		&\leq \frac{C}{\left( R - \rho \right)^{2}}\int_{B_{R}}\left\lvert  \X u- \xi\right\rvert^2\, dx +  C\left( \int_{B_{R}}\eta^{4}\left\lvert \X T u \right\rvert^{2}\, dx \right)^{\frac{1}{2}}\left( \int_{B_{R}}\left\lvert  \X u- \xi\right\rvert^2\, dx\right)^{\frac{1}{2}} \notag \\
		&\leq \frac{C}{\left( R - \rho \right)^{2}}\int_{B_{R}}\left\lvert  \X u- \xi\right\rvert^2\, dx.  
	\end{align*}
	This proves  \eqref{estimate of XXu in nondegenerate case} and completes the proof of the Lemma. 
\end{proof}
\subsection{Reverse H\texorpdfstring{\"{o}}{o}lder inequalities}
Our Caccioppoli inequalities in Lemma \ref{Cac in the nondeg regime} leads us to a few reverse H\"{o}lder inequalities in the nondegenerate regime. 
\begin{lemma}\label{reverse holder inequality lemma}
	Let $1 < p < \infty$ and let $\Omega \subset \mathbb{H}_{n}$ be open. Let $u \in HW^{1,p}_{\text{loc}}\left(\Omega\right)$ be a weak solution to 
	\begin{align*}
		\operatorname{div}_{\mathbb{H}} \left( \left\lvert \X u\right\rvert^{p-2}\X u\right) &= 0 &&\text{ in } \Omega. 
	\end{align*}
	Let $\tilde{\Omega} \subset \subset \Omega$ be any open subset. Suppose we have 
	\begin{align}\label{two sided bound rev holder}
		\frac{\lambda}{4B} \leq |\Xu| \leq A\lambda \qquad \text{ in } \tilde{\Omega}, 
	\end{align}
	for some constants $A, B \geq 1$ and $\lambda >0.$ Then there exists a constant $C_{RH} \equiv C_{RH}\left( A, B, n, p \right) \geq 1,$ such that for any $\xi \in \mathbb{R}^{2n},$ and for any ball $B_{R} \subset \tilde{\Omega}$ with $R>0,$ we have the following reverse H\"{o}lder inequality 
	\begin{align}\label{rev holder bound 1 to 2 star}
		\left( \fint_{B_{R/2}}\left\lvert  \X u- \xi\right\rvert^\frac{2Q}{Q-2}\ \mathrm{d}x \right)^{\frac{Q-2}{2Q}} &\leq C_{RH}\fint_{B_{R}}\left\lvert  \X u- \xi\right\rvert\ \mathrm{d}x. 
	\end{align}
	In particular, for any exponent $1 \leq q \leq 2Q/(Q-2),$ we have the following estimate 
	\begin{align}\label{rev holder bound}
		\left( \fint_{B_{R/2}}\left\lvert  \X u- \xi\right\rvert^{q}\ \mathrm{d}x \right)^{\frac{1}{q}} &\leq C_{RH}\left( \fint_{B_{R}}\left\lvert  \X u- \xi\right\rvert^{q}\ \mathrm{d}x \right)^{\frac{1}{q}}. 
	\end{align}
\end{lemma}
\begin{proof}
	In view of \eqref{two sided bound rev holder}, the Caccippoli inequality \eqref{estimate of XXu in nondegenerate case} holds for any ball $B \subset \tilde{\Omega}$. Now fix concentric balls $B_{\rho} \subset B_{r} \subset \tilde{\Omega}.$ Let  $\eta \in C_{c}^{\infty}\left(B_{R}\right)$, $0 \leq \eta \leq 1$ in $B_{R}$ with $\eta \equiv 1$ in $B_{\rho}$ and $\left\lvert \X \eta \right\rvert \leq c/\left( R - \rho\right)$  for some constant $c>1.$ Then  $\eta \left( \X u -\xi \right) \in HW^{1,2}_{0}\left(B_{R}; \mathbb{R}^{2n}\right)$ for any $\xi \in \mathbb{R}^{2n}.$ Thus we have
	\begin{align}
		\int_{B_{\rho}} \left\lvert   \X u  - \xi \right\rvert^{\frac{2Q}{Q-2}} 
		&\leq \int_{B_{R}} \left\lvert  \eta\left( \X u  - \xi \right) \right\rvert^{\frac{2Q}{Q-2}} \notag\\
		& \stackrel{\eqref{poincaresobolevineq} }{\leq} C\left( \int_{B_{R}} \left\lvert \X  \left[ \eta\left( \X u  - \xi \right) \right]\right\rvert^{2} \right)^{\frac{Q}{Q-2}} \notag\\&\leq C\left(	\int_{B_{R}} \eta^{2}\left\lvert \X \X u \right\rvert^{2}  + \int_{B_{R}} \left\lvert \left\langle \X \eta , \X u  - \xi \right\rangle  \right\rvert^{2} \right)^{\frac{Q}{Q-2}} \notag \\
		&\stackrel{\eqref{estimate of XXu in nondegenerate case}}{\leq} C\left( R-\rho\right)^{-\frac{2Q}{Q-2}} \left( 	\int_{B_{R}} \left\lvert \X u - \xi  \right\rvert^{2} \right)^{\frac{Q}{Q-2}}. \label{rev holder at int for Xu}
	\end{align} 
	This implies, using the H\"{o}lder inequality and the Young's inequality, 
	\begin{align*}
		\left( \int_{B_{\rho}} \left\lvert   \X u  - \xi \right\rvert^{\frac{2Q}{Q-2}} \right)^{\frac{Q-2}{2Q}} &\leq \frac{C}{\left(R-\rho\right)}\left( 	\int_{B_{R}} \left\lvert \X u - \xi  \right\rvert^{2} \right)^{\frac{1}{2}} \\
		&\leq  \frac{C}{\left(R-\rho\right)}\left( 	\int_{B_{R}} \left\lvert \X u - \xi  \right\rvert^{\frac{2Q}{Q+2}} \left\lvert \X u - \xi  \right\rvert^{\frac{4}{Q+2}}\right)^{\frac{1}{2}} \\
		&\leq\frac{C}{\left(R-\rho\right)}\left( \int_{B_{R}} \left\lvert   \X u  - \xi \right\rvert^{\frac{2Q}{Q-2}}\right)^{\frac{Q-2}{2\left(Q+2\right)}}\left( \int_{B_{R}} \left\lvert   \X u  - \xi \right\rvert\right)^{\frac{2}{Q+2}} \\
		&\leq \frac{1}{2}\left( \int_{B_{R}} \left\lvert   \X u  - \xi \right\rvert^{\frac{2Q}{Q-2}}\right)^{\frac{Q-2}{2Q}} + \frac{C}{\left( R-\rho\right)^{\frac{Q+2}{2}}}\int_{B_{R}} \left\lvert   \X u  - \xi \right\rvert.
	\end{align*}
	Thus, setting 
	\begin{align*}
		\phi \left(r\right):= \left( \int_{B_{r}} \left\lvert   \X u  - \xi \right\rvert^{\frac{2Q}{Q-2}} \right)^{\frac{Q-2}{2Q}}
	\end{align*} and appying Lemma $8.18$ in \cite{giaquinta-martinazzi-regularity}, we deduce 
	\begin{align*}
		\left( \int_{B_{\rho}} \left\lvert   \X u  - \xi \right\rvert^{\frac{2Q}{Q-2}} \right)^{\frac{Q-2}{2Q}} \leq C\left( R-\rho\right)^{-\frac{Q+2}{2}}\int_{B_{R}} \left\lvert   \X u  - \xi \right\rvert
	\end{align*}
	for every $0 < \rho <R.$ Setting $\rho = R/2$, we arrive at 
	\begin{align*}
		\left( \fint_{B_{R/2}} \left\lvert   \X u  - \xi \right\rvert^{\frac{2Q}{Q-2}} \right)^{\frac{Q-2}{2Q}} \leq C\fint_{B_{R}} \left\lvert   \X u  - \xi \right\rvert. 
	\end{align*}
	This proves \eqref{rev holder bound 1 to 2 star} and \eqref{rev holder bound} follows from this using H\"{o}lder inequality on both sides. This completes the proof. 
\end{proof}
As a consequence, we have the following result for $V \left( \X u\right),$ which is not quite as sharp when $1<p<2,$ but this would suffice for our purposes.   
\begin{lemma}\label{reverse holder inequality V lemma}
	Let $1 < p < \infty$ and let $\Omega \subset \mathbb{H}_{n}$ be open. Let $u \in HW^{1,p}_{\text{loc}}\left(\Omega\right)$ be a weak solution to 
	\begin{align*}
		\operatorname{div}_{\mathbb{H}} \left( \left\lvert \X u\right\rvert^{p-2}\X u\right) &= 0 &&\text{ in } \Omega. 
	\end{align*}
	Let $\tilde{\Omega} \subset \subset \Omega$ be any open subset. Suppose we have 
	\begin{align}\label{two sided bound rev holder V}
		\frac{\lambda}{4B} \leq |\Xu| \leq A\lambda \qquad \text{ in } \tilde{\Omega}, 
	\end{align}
	for some constants $A, B \geq 1$ and $\lambda >0.$ Let $\xi \in \mathbb{R}^{2n}$ be a constant vector if $2\le p <\infty$ and   
	\begin{align}\label{upper bound for xi for 1<p<2}
		\left\lvert \xi \right\rvert \leq A\lambda \qquad \text{ if } 1 < p < 2.
	\end{align}
	Then there exists a constant $C_{RH} = C_{RH, V}\left( A, B, n, p \right) \geq 1,$ such that for any ball $B_{R} \subset \tilde{\Omega}$ with $R>0,$ we have 
	\begin{align}
		\left( \fint_{B_{R/2}}\left\lvert  V \left( \X u \right) - V \left( \xi\right) \right\rvert^{\frac{2Q}{Q-2}}\ \mathrm{d}x \right)^{\frac{Q-2}{2Q}} &\leq C_{RH, V}\fint_{B_{R}}\left\lvert V \left(  \X u \right) - V \left(\xi\right)\right\rvert\ \mathrm{d}x \label{rev holder bound V 1 to 2 star}. 
	\end{align}
	In particular, for any exponent $1 \leq q \leq 2Q/(Q-2),$ we have the following estimate 
	\begin{align}
		\left( \fint_{B_{R/2}}\left\lvert  V \left( \X u \right) - V \left( \xi\right) \right\rvert^{q}\ \mathrm{d}x \right)^{\frac{1}{q}} &\leq C_{RH, V}\left( \fint_{B_{R}}\left\lvert V \left(  \X u \right) - V \left(\xi\right)\right\rvert^{q}\ \mathrm{d}x  \right)^{\frac{1}{q}}\label{rev holder bound V}, 
	\end{align}
	for any $\xi \in \mathbb{R}^{n}$ which satisfies \eqref{upper bound for xi for 1<p<2}.
	\end{lemma}
\begin{proof}
	We first tackle the case $p \geq 2,$ where no bound on $\xi$ is needed. To derive the inequalities for $V\left(\X u\right),$ observe that $V\left(\cdot\right)$ is smooth away from the origin and in view of the pointwise bounds \eqref{two sided bound rev holder V}, we can justify the computation of derivatives and deduce the inequality
	\begin{align*}
		\left\lvert \X \left[ V \left(\X u\right)\right] \right\rvert &\leq  c  \left\lvert \X u \right\rvert^{\frac{p-2}{2}}\left\lvert \X \X u \right\rvert 
	\end{align*}
	for some constant $c>0,$ which depends only on $n$ and $p$.  This, combined with the Caccippoli inequality \eqref{estimate of XXu in nondegenerate case} implies,  again in view of the pointwise bounds, that $V \left(\X u\right) \in HW^{1,2}_{\text{loc}}\left(\tilde{\Omega}\right).$ Thus, for $B_{\rho}, B_{R}$ and $\eta$ as defined earlier, we have $\eta\left( V \left(\X u\right) - V\left(\xi\right) \right)\in HW^{1,2}_{0}\left(B_{R}; \mathbb{R}^{2n}\right)$ for any $\xi \in \mathbb{R}^{2n}.$ Thus we have
	\begin{align*}
		\int_{B_{\rho}} &\left\lvert V \left(\X u\right) -V\left(\xi\right)\right\rvert^{\frac{2Q}{Q-2}} \\
		&\leq  	\int_{B_{R}} \eta^{\frac{2Q}{Q-2}}\left\lvert V \left(\X u\right) -V\left(\xi\right)\right\rvert^{\frac{2Q}{Q-2}} \\
		&\stackrel{\eqref{poincaresobolevineq}}{\le}  	C\left( \int_{B_{R}} \left\lvert \X \left[ \eta \left( V \left(\X u\right) -V\left(\xi\right)\right)\right] \right\rvert^{2}\right) ^{\frac{Q}{Q-2}} \\
		&= C\left( \int_{B_{R}} \eta^{2} \left\lvert \X \left[ V \left(\X u\right)\right] \right\rvert^{2}  + \int_{B_{R}} \left\lvert \left\langle \X \eta,  V \left(\X u\right) -V\left(\xi\right)\right\rangle \right\rvert^{2}\right)^{\frac{Q}{Q-2}} \\
		&\leq C \left( \int_{B_{R}} \eta^{2} \left\lvert \X u \right\rvert^{p-2}\left\lvert \X \X u \right\rvert^{2}  + \int_{B_{R}} \left\lvert \left\langle \X \eta,  V \left(\X u\right) -V\left(\xi\right)\right\rangle \right\rvert^{2}\right)^{\frac{Q}{Q-2}} \\
		&\stackrel{\eqref{two sided bound rev holder}}{\leq} C \left( \lambda^{p-2}\int_{B_{R}} \eta^{2} \left\lvert \X \X u \right\rvert^{2}  + \frac{1}{\left( R - \rho\right)^{2}}\int_{B_{R}} \left\lvert V \left(\X u\right) -V\left(\xi\right) \right\rvert^{2}\right)^{\frac{Q}{Q-2}} \\
		&\stackrel{\eqref{estimate of XXu in nondegenerate case}}{\leq} \frac{C}{\left( R - \rho\right)^{\frac{2Q}{Q-2}}}\left( \int_{B_{R}}\lambda^{p-2} \left\lvert \X u - \xi \right\rvert^{2}  + \int_{B_{R}} \left\lvert V \left(\X u\right) -V\left(\xi\right) \right\rvert^{2}\right)^{\frac{Q}{Q-2}}.
	\end{align*}
	Now in view of \eqref{two sided bound rev holder V} and \eqref{upper bound for xi for 1<p<2}, we deduce 
	\begin{align*}
		\lambda^{p-2}\left\lvert \X u - \xi \right\rvert^{2}
		&\leq \max \left\{(4B)^{p-2}, (2A)^{2-p}\right\} \left( \left\lvert \X u \right\rvert + \left\lvert \xi\right\rvert\right)^{p-2}\left\lvert \X u - \xi \right\rvert^{2}\\
		&\stackrel{\eqref{constant cv}}{\leq} c^2_{V}\max \left\{(4B)^{p-2}, (2A)^{2-p}\right\}\left\lvert V \left(\X u\right) -V\left(\xi\right) \right\rvert^{2}. 
	\end{align*}
	Using this in the last estimate, we arrive at 
	\begin{align*}
		\int_{B_{\rho}} \left\lvert V \left(\X u\right) -V\left(\xi\right)\right\rvert^{\frac{2Q}{Q-2}} \leq C \left( R - \rho\right)^{-\frac{2Q}{Q-2}}\left( \int_{B_{R}} \left\lvert V \left(\X u\right) -V\left(\xi\right) \right\rvert^{2}\right)^{\frac{Q}{Q-2}}. 
	\end{align*}
	The desired estimates now follow from this in the same way as in Lemma \ref{reverse holder inequality lemma}.
\end{proof}

\subsection{Estimates in the nondegenerate regime}
Some of the Lemmas we would prove in this subsection are similar in spirit to some Lemmas in \cite{DiBenedettoFriedman1} ( compare Lemma \ref{comparison with linearized lemma}, \ref{if xi0 then x1 exists} and \ref{iteration of close to constant lemma} with Lemma 4.3, 4.4 and 4.5 in \cite{DiBenedettoFriedman1} respectively ). However, our results are valid only the nondegenerate regime. 
\begin{lemma}\label{comparison with linearized lemma}
	Let $1<p<\infty$. Suppose $u \in HW^{1,p}_{\text{loc}}\left( B_{2R}\right)$ be a weak solution of
		\begin{align*}
		\operatorname{div}_{\mathbb{H}} \left( \left\lvert \X u\right\rvert^{p-2}\X u\right) &= 0 &&\text{ in } B_{2R}. 
	\end{align*}
	Let $A, B \geq 1$ and $\lambda >0$ be constants and suppose  
	\begin{align}\label{two sided bound in comparison with linearized}
		\frac{\lambda}{4B} \leq  \left\lvert \X u \right\rvert \leq A\lambda \qquad \text{ in } B_{R}.  
	\end{align} 
	Let $\xi \in \mathbb{R}^{2n}$ be any vector satisfying 
	\begin{align*}
		\frac{\lambda}{4B} \leq |\xi| \leq A \lambda. 
	\end{align*}
	Let $v \in  HW^{1,p}(B_{R/2})$ be the unique weak solution of the Dirichlet boundary value problem for the linear equation 
	\begin{align*}
		\left\lbrace \begin{aligned}
			\operatorname{div}_{\mathbb{H}} \left[ \left\lvert \xi \right\rvert^{p-2} \X v + \left( p-2\right) \left\lvert \xi \right\rvert^{p-4} \left\langle \xi, \X v \right\rangle \xi \right] &=0 &&\text{ in } B_{R/2}, \\
			v&=u &&\text{ on } \partial B_{R/2}. 
		\end{aligned}\right. 
	\end{align*}
	Then there exists a constant $C\equiv C \left( Q, p, A, B  \right) \geq 1$ such that we have the estimate 
	\begin{align}\label{comparison with linearized}
		\int_{B_{R/2}} \left\lvert \X u - \X v\right\rvert^{2} \leq C\left( \frac{1}{A^{2}\lambda^{2}} \fint_{B_{R}}\left\lvert \X u -\xi \right\rvert^{2}\right)^{\frac{2}{Q-2}}\int_{B_{R}}\left\lvert \X u -\xi \right\rvert^{2}. 
	\end{align}
\end{lemma}
\begin{proof}
	From the equation satisfied by $u$, we  can write 
	\begin{align*}
		&\operatorname{div}_{\mathbb{H}} \left[ \left\lvert \xi \right\rvert^{p-2} \X u + \left( p-2\right) \left\lvert \xi \right\rvert^{p-4} \left\langle \xi, \X u \right\rangle \xi \right] \\&= \operatorname{div}_{\mathbb{H}} \left[ \left\lvert \xi \right\rvert^{p-2} \X u + \left( p-2\right) \left\lvert \xi \right\rvert^{p-4} \left\langle \xi, \X u \right\rangle \xi - \left\lvert \X u \right\rvert^{p-2}\X u \right] \\
		&\begin{aligned}
			=&\operatorname{div}_{\mathbb{H}}\left[ \left\lvert \xi \right\rvert^{p-2}\xi-\left\lvert \X u \right\rvert^{p-2}\X u  + \left\lvert \xi \right\rvert^{p-2}\left( \X u -\xi \right) + \left(p-2\right)\left\lvert \xi \right\rvert^{p-4} \left\langle \xi, \X u - \xi \right\rangle \xi\right] \\&\qquad \qquad\qquad\qquad \qquad\qquad\qquad\qquad\qquad\qquad\qquad+\left(p-2\right)\operatorname{div}_{\mathbb{H}} \left( \left\lvert \xi \right\rvert^{p-2}\xi \right)
		\end{aligned}\\
		&=\operatorname{div}_{\mathbb{H}}\left[ \left\lvert \xi \right\rvert^{p-2}\xi-\left\lvert \X u \right\rvert^{p-2}\X u  + \left\lvert \xi \right\rvert^{p-2}\left( \X u -\xi \right) + \left(p-2\right)\left\lvert \xi \right\rvert^{p-4} \left\langle \xi, \X u - \xi \right\rangle \xi\right]\\
		&:= \operatorname{div}_{\mathbb{H}} F.
	\end{align*}
	Setting $w:= u-v \in HW_{0}^{1,p}(B_{R/2}),$ we see that this implies 
	\begin{align*}
		\operatorname{div}_{\mathbb{H}} \left[ \left\lvert \xi \right\rvert^{p-2} \X w + \left( p-2\right) \left\lvert \xi \right\rvert^{p-4} \left\langle \xi, \X w \right\rangle \xi \right]  &= \operatorname{div}_{\mathbb{H}} F &&\text{ in } B_{R/2}. 
	\end{align*}
	Plugging $w$ as a test function, we deduce 
	\begin{align*}
		\left\lvert \xi \right\rvert^{p-2}\int_{B_{R/2}} \left\lvert \X w \right\rvert^{2} +  \left( p-2\right)\left\lvert \xi \right\rvert^{p-4} \int_{B_{R/2}}\left\lvert \left\langle \xi, \X w \right\rangle \right\rvert^{2} = \int_{B_{R/2}}\left\langle F, \X w \right\rangle .
	\end{align*}
	Now, if $p \geq 2,$ clearly we have 
	\begin{align*}
		\left\lvert \xi \right\rvert^{p-2}\int_{B_{R/2}} \left\lvert \X w \right\rvert^{2} &\leq \left\lvert \xi \right\rvert^{p-2}\int_{B_{R/2}} \left\lvert \X w \right\rvert^{2} +  \left( p-2\right)\left\lvert \xi \right\rvert^{p-4} \int_{B_{R/2}}\left\lvert \left\langle \xi, \X w \right\rangle \right\rvert^{2}. 
	\end{align*}
	On the other hand, if $1 < p <2,$ then we have 
	\begin{align*}
		\left(p-2\right)\left\lvert \xi \right\rvert^{p-4} \int_{B_{R/2}}\left\lvert \left\langle \xi, \X w \right\rangle \right\rvert^{2} \geq \left( p-2\right)	\left\lvert \xi \right\rvert^{p-2}\int_{B_{R/2}} \left\lvert \X w \right\rvert^{2}.
	\end{align*}
	Hence in this case, we have 
	\begin{align*}
		\left( p-1\right)	\left\lvert \xi \right\rvert^{p-2}\int_{B_{R/2}} \left\lvert \X w \right\rvert^{2} \leq \left\lvert \xi \right\rvert^{p-2}\int_{B_{R/2}} \left\lvert \X w \right\rvert^{2} +  \left( p-2\right)\left\lvert \xi \right\rvert^{p-4} \int_{B_{R/2}}\left\lvert \left\langle \xi, \X w \right\rangle \right\rvert^{2}.
	\end{align*}
	Thus, in either case, we have,  
	\begin{align*}
		\int_{B_{R/2}} \left\lvert \X w \right\rvert^{2}
		&\leq \frac{C}{\lambda^{p-2}}	\min\left\lbrace 1, p-1  \right\rbrace \left\lvert \xi \right\rvert^{p-2}\int_{B_{R/2}} \left\lvert \X w \right\rvert^{2} \\
		&\leq \frac{C}{\lambda^{p-2}}\left[ \left\lvert \xi \right\rvert^{p-2}\int_{B_{R/2}} \left\lvert \X w \right\rvert^{2} +  \left( p-2\right)\left\lvert \xi \right\rvert^{p-4} \int_{B_{R/2}}\left\lvert \left\langle \xi, \X w \right\rangle \right\rvert^{2}\right] \\
		&\leq  \frac{C}{\lambda^{p-2}}\int_{B_{R/2}}\left\langle F, \X w \right\rangle \\
		&\leq  \frac{C}{\lambda^{p-2}}\int_{B_{R/2}}\left\lvert F \right\rvert \left\lvert \X w \right\rvert \leq \delta^{2}\frac{C}{\lambda^{p-2}}\int_{B_{R/2}} \left\lvert \X w \right\rvert^{2} + \frac{1}{4\delta^{2}}\frac{C}{\lambda^{p-2}}\int_{B_{R/2}} \left\lvert F \right\rvert^{2},
	\end{align*}
	where $C$ depends on $A, B$, $ p.$ Choosing $\delta >0$ such that $C\delta^{2}\lambda^{2-p} = 1/2,$ we deduce 
	\begin{align}\label{Xw bound}
		\int_{B_{R/2}} \left\lvert \X w \right\rvert^{2} \leq \frac{C}{\lambda^{2p-4}}\int_{B_{R/2}} \left\lvert F \right\rvert^{2}.
	\end{align}
	Now, from the definition of $F$, we have ( see Lemma 4.3. in \cite{DiBenedettoFriedman1} ), 
	\begin{align}\label{pointwise bound for F}
		\left\lvert F \right\rvert &\leq \frac{c \left(A, B\right)}{\lambda}\left( \left\lvert \X u \right\rvert + \left\lvert \xi \right\rvert \right)^{p-2}\left\lvert \X u -\xi \right\rvert^{2}. 
	\end{align}
	Thus, we have, with constants $c\equiv c\left(A, B\right) >0,$ 
	\begin{align}
		\int_{B_{R/2}} \left\lvert F \right\rvert^{2} &\leq c \left(A, B\right)\lambda^{2p-6}\int_{B_{R/2}}\left\lvert \X u -\xi \right\rvert^{4} \notag\\
		&\leq c \lambda^{2p-2 - \frac{2Q}{Q-2}}\int_{B_{R/2}}\left\lvert \X u -\xi \right\rvert^{\frac{2Q}{Q-2}}.\label{estimate:L2 of F}
	\end{align}
	Now, applying \eqref{rev holder bound 1 to 2 star}, we have   
	\begin{align*}
		\fint_{B_{R/2}}\left\lvert \X u -\xi \right\rvert^{\frac{2Q}{Q-2}} \leq C_{RH}^\frac{2Q}{Q-2}\left( \fint_{B_{R}}\left\lvert \X u -\xi \right\rvert^{2}\right)^{\frac{Q}{Q-2}}. 
	\end{align*}
	Using this in \eqref{estimate:L2 of F}, we conclude that there exists a constant $C\equiv C(A,B, n, p)>0$ such that 
	\begin{align*}
		\int_{B_{R/2}} \left\lvert F \right\rvert^{2} &\leq C\lambda^{2p-2 - \frac{2Q}{Q-2}}\left( \fint_{B_{R}}\left\lvert \X u -\xi \right\rvert^{2}\right)^{\frac{2}{Q-2}}\int_{B_{R}}\left\lvert \X u -\xi \right\rvert^{2}.
	\end{align*}
	Combining this with \eqref{Xw bound}, we have 
	\begin{align*}
		\int_{B_{R/2}} \left\lvert \X w \right\rvert^{2} &\leq C\lambda^{\left[ \left( 2p-2 - \frac{2Q}{Q-2}\right) + \left( 4 -2p\right)\right]} \left(  \fint_{B_{R}}\left\lvert \X u -\xi \right\rvert^{2}\right)^{\frac{2}{Q-2}}\int_{B_{R}}\left\lvert \X u -\xi \right\rvert^{2} \\
		&\leq C\lambda^{ - \frac{4}{Q-2}} \left(  \fint_{B_{R}}\left\lvert \X u -\xi \right\rvert^{2}\right)^{\frac{2}{Q-2}}\int_{B_{R}}\left\lvert \X u -\xi \right\rvert^{2}.
	\end{align*}
	This implies \eqref{comparison with linearized} and completes the proof. 
\end{proof}
\begin{lemma}\label{osc decay for linearized eq lemma}
	Let  $1<p<\infty$, $A, B \geq 1$, and $\lambda >0$ be constants and let $\xi \in \mathbb{R}^{2n}$ be any vector satisfying 
	\begin{align*}
		\frac{\lambda}{4B} \leq |\xi| \leq A\lambda. 
	\end{align*}
	Let $v \in  HW^{1,p}(B_{R/2})$ be  a weak solution of the linear equation 
	\begin{align*}
		\operatorname{div}_{\mathbb{H}} \left[ \left\lvert \xi \right\rvert^{p-2} \X v + \left( p-2\right) \left\lvert \xi \right\rvert^{p-4} \left\langle \xi, \X v \right\rangle \xi \right] &=0 &&\text{ in } B_{R/2}.
	\end{align*}
	Then there exists a constant $C \equiv C \left(Q, p, A, B \right)>0 $ such that for any $0<\rho \leq R/2$, we have the estimate
	\begin{align}\label{linear excess decay}
		\int_{B_{\rho}} \left\lvert \X v - \left( \X v\right)_{\rho} \right\rvert^{2} \leq C \left(\frac{\rho}{R}\right)^{Q+2}\int_{B_{R/2}} \left\lvert \X v - \left( \X v\right)_{R/2} \right\rvert^{2}. 
	\end{align}
\end{lemma}
\begin{proof}
	The proof  can be found in Theorem 3.2 in Xu-Zuily \cite{Xu_Zuily_subellipticdecayestimate}. As is well known, the constant in the estimate, depends only upon the \emph{ellipticity ratio}, which in this case is $A/B,$ i.e. independent of $\lambda.$
\end{proof}
\begin{lemma}\label{if xi0 then x1 exists}
	Let $1<p<\infty$. Suppose $u \in HW^{1,p}_{\text{loc}}\left( B_{2R}\right)$ be a weak solution of
	\begin{align*}
		\operatorname{div}_{\mathbb{H}} \left( \left\lvert \X u\right\rvert^{p-2}\X u\right) &= 0 &&\text{ in } B_{2R}. 
	\end{align*} 
	Let $A, B \geq 1$ be constants and suppose  
	\begin{align*}
		\frac{\lambda}{4B} \leq  \left\lvert \X u \right\rvert \leq A\lambda \qquad \text{ in } B_{R}.  
	\end{align*} Then for any $\delta_{0} >0$ with $64A^{2}B^{2}\delta_{0} < 1,$ there exist $\varepsilon, \tau \in \left(0,1\right),$ depending only on $Q$, $p$, $A,$ $B,$ and $\delta_{0},$ satisfying 
	\begin{align}\label{log bound single}
		0 < \frac{\log \delta_{0}}{2 \log \tau} < 1, 
	\end{align} such that if $\xi^{0} \in \mathbb{R}^{2n}$ is a vector satisfying 
	\begin{enumerate}[(i)]
		\item $ \displaystyle \frac{\lambda}{4B} \leq \left\lvert \xi^{0} \right\rvert \leq A\lambda,$
		\item $ \displaystyle \fint_{B_{R}} \left\lvert \X u - \xi^{0} \right\rvert^{2} \leq \varepsilon A^{2}\lambda^{2},$
	\end{enumerate}
	then there exists a vector $ \xi^{1} \in \mathbb{R}^{2n}$ satisfying  
	\begin{align}\label{lower and upper bound}
		\frac{\lambda}{8B}  \leq \left\lvert \xi^{1} \right\rvert \leq 2A\lambda,  
	\end{align}
	\begin{align}
		\fint_{B_{\tau R}}  \left\lvert \X u - \xi^{1} \right\rvert^{2} &\leq \varepsilon A^{2}\lambda^{2}, \label{smallness at  next level}\\ \intertext{and}
		\int_{B_{\tau R}}  \left\lvert \X u - \xi^{1} \right\rvert^{2} &\leq \delta_{0} \tau^{Q} 	\int_{B_{R}}  \left\lvert \X u - \xi^{0} \right\rvert^{2}. \label{excess decay estimate}
	\end{align}
\end{lemma}
\begin{proof}
	We apply Lemma \ref{comparison with linearized lemma} with the choice $\xi = \xi^{0}$ to deduce 
	
	\begin{align}\label{xu-xv in iteration}
		\int_{B_{R/2}} \left \lvert \Xu- \X v \right \rvert^2 \leq C_1  \varepsilon ^{\frac{2}{Q-2}} \int_{B_{R}} \left \lvert \Xu- \xi_0 \right \rvert^2 
	\end{align}
	for some constant $C_1\equiv C_1(Q, p, A, B)>1.$
	Now we estimate 
	\begin{align}\label{xu-xv in iteration1}
		\int_{B_{R/2}}\left\lvert \X v - \xi^{0} \right\rvert^{2} &\leq  2	\int_{B_{R/2}}\left\lvert \X u - \xi^{0} \right\rvert^{2} + 2	\int_{B_{R/2}}\left\lvert \X u - \X v \right\rvert^{2} \notag \\
		&\stackrel{\eqref{xu-xv in iteration}}{\leq} C_2\left( 1 +\varepsilon^{\frac{2}{Q-2}}\right) \int_{B_{R}}\left\lvert \X u - \xi^{0}\right\rvert^{2},
	\end{align}
	for some constant $C_2\equiv C_2(Q, p, A, B)>1.$ Now, by Lemma \ref{osc decay for linearized eq lemma}, for any $0 < \rho < R/2,$ we have 
	\begin{align}\label{linear excess decay1}
		\int_{B_{\rho}} \left\lvert \X v - \left( \X v\right)_{\rho} \right\rvert^{2} &\leq C_3 \left(\frac{\rho}{R}\right)^{Q+2}\int_{B_{R/2}} \left\lvert \X v - \left( \X v\right)_{R/2} \right\rvert^{2} \notag \\
		&\leq C_3 \left(\frac{\rho}{R}\right)^{Q+2}\int_{B_{R/2}} \left\lvert \X v - \xi^{0} \right\rvert^{2},
	\end{align}
	for some constant $C_3\equiv C_3(Q, p, A, B)>1.$ Hence, setting $\rho = \tau R/2$ and 
	\begin{align*}
		\xi^{1} :=  \left( \X v\right)_{\rho} =  \left( \X v\right)_{\tau R/2},
	\end{align*}
	we have 
	\begin{align*}
		\int_{B_{\tau R/2}} \left\lvert \X u - \xi^{1} \right\rvert^{2} &\leq 	2\int_{B_{\tau R/2}} \left\lvert \X v - \xi^{1} \right\rvert^{2} + 	2\int_{B_{\tau R/2}} \left\lvert \X u - \X v \right\rvert^{2} \\
		&\stackrel{\eqref{linear excess decay1}}{\leq} 2C_3 \tau^{Q+2}\int_{B_{R/2}} \left\lvert \X v - \xi^{0} \right\rvert^{2} + 	2\int_{B_{R/2}} \left\lvert \X u - \X v \right\rvert^{2} \\
		&\stackrel{\eqref{xu-xv in iteration}, \eqref{xu-xv in iteration1}}{\leq} C_4 \left[ \tau^{Q+2} \left( 1 + \varepsilon^{\frac{2}{Q-2}}\right)  + \varepsilon^{\frac{2}{Q-2}} \right] \int_{B_{R}} \left\lvert \X u - \xi^{0} \right\rvert^{2} \\
		&\leq C_4 \left[ \tau^{Q+2}   + \varepsilon^{\frac{2}{Q-2}} \right] \int_{B_{R}} \left\lvert \X u - \xi^{0} \right\rvert^{2}.
	\end{align*}
	Here, again $C_4\equiv C_4(Q, p, A, B)>1$ is a constant.   Now we choose $\varepsilon >0$ sufficiently small ( depending on $\tau$  ) such that we have 
	\begin{align}\label{choice of epsilon}
		\varepsilon^{\frac{2}{Q-2}} = \tau^{Q+2}. 
	\end{align}
	This choice implies that
	\begin{align}\label{choice of tau}
		\int_{B_{\tau R/2}} \left\lvert \X u - \xi^{1} \right\rvert^{2} &\leq C_M \tau^{Q+2} \int_{B_{R}} \left\lvert \X u - \xi^{0} \right\rvert^{2},
	\end{align}
	where 
	\begin{align}\label{max constant in xi0 to xi1 lem}
	C_M: = \max\{ C_1, C_2, C_3, 2C_4\}.
	\end{align}
	Now we choose $\tau >0$ small enough to have 
	\begin{align}\label{choice of tau in iteration}
		2^{Q+2}C_M \tau^{2} \leq \delta_{0}.
	\end{align}
	Note that this implies \eqref{log bound single}. 
	This proves \eqref{excess decay estimate}.  Assumption $(ii)$ and \eqref{excess decay estimate} implies \eqref{smallness at  next level}. Now observe that since  $Q \geq 4,$ we have 
	\begin{align*}
		\frac{(Q-2)(Q+2)}{2} - Q = \frac{Q^{2} -2Q -4}{2} = \frac{\left( Q -1 \right)^{2} -5}{2} \geq 2.
	\end{align*}
	Thus, we deduce 
	\begin{align}\label{bound of tau in terms of delta}
		2^{Q+2}\tau^{\frac{(Q-2)(Q+2)}{2} - Q} \leq 2^{Q+2}\tau^{2} \leq 2^{Q+2}C_M\tau^{2} \leq \delta_{0}.
	\end{align}
	Using this, we deduce 
	\begin{align}\label{xi1-xi0}
		\left\lvert \xi^{1} - \xi^{0} \right\rvert^{2} &\leq \fint_{B_{\tau R/2}} \left\lvert \X v - \xi^{0} \right\rvert^{2} \notag \\
		&\leq  2\fint_{B_{\tau R/2}} \left\lvert \X u - \xi^{0} \right\rvert^{2} + 2\fint_{B_{\tau R/2}} \left\lvert \X u - \X v \right\rvert^{2} \notag\\
		&\leq  2^{Q+1}\tau^{-Q}\fint_{B_{ R}} \left\lvert \X u - \xi^{0} \right\rvert^{2} +  2\tau^{-Q}\fint_{B_{ R/2}} \left\lvert \X u - \X v \right\rvert^{2} \notag\\
		&\stackrel{\eqref{xu-xv in iteration}, \eqref{choice of epsilon}, \eqref{choice of tau in iteration}}{\leq}  \left( 2^{Q+1}\tau^{-Q} + \delta_{0} \right) \fint_{B_{R}} \left\lvert \X u - \xi^{0} \right\rvert^{2} \\
		&\leq 2^{Q+2}\tau^{-Q} \varepsilon A^{2}\lambda^{2} \stackrel{\eqref{choice of epsilon}}{\le} 2^{Q+2}\tau^{\frac{(Q-2)(Q+2)}{2} - Q}A^{2}\lambda^{2} \stackrel{\eqref{bound of tau in terms of delta}}{\le}\delta_{0}A^{2}\lambda^{2}. \notag
	\end{align}
Hence, we have 
	\begin{align*}
		\left\lvert \xi^{1} - \xi^{0} \right\rvert^{2} \leq \delta_{0}A^{2}\lambda^{2}
	\end{align*}
	Then we have 
	\begin{align*}
		\left( \frac{1}{4B}- \sqrt{\delta_{0}} A \right)\lambda 	\leq \left\lvert \xi^{0} \right\rvert - \left\lvert \xi^{1} - \xi^{0} \right\rvert  \leq 	\left\lvert \xi^{1} \right\rvert \leq 	\left\lvert \xi^{0} \right\rvert + \left\lvert \xi^{1} - \xi^{0} \right\rvert \leq \left( 1 + \sqrt{\delta_{0}} \right)A\lambda. 
	\end{align*}
	Since $64A^{2}B^{2}\delta_{0} < 1,$ we have 
	\begin{gather*}
		\left( \frac{1}{4B}- \sqrt{\delta_{0}} A \right)\lambda > \frac{1}{B}\left( \frac{1}{4} - \frac{1}{8} \right)\lambda > \frac{\lambda}{8B} \qquad \text{ and } \\ \left( 1 + \sqrt{\delta_{0}} \right)A\lambda \leq \left( 1 + \frac{1}{8AB} \right)A\lambda \leq \left( 1 + \frac{1}{8} \right)A\lambda \leq 2A\lambda. 
	\end{gather*}
	This proves \eqref{lower and upper bound} and completes the proof. 
\end{proof}
\begin{lemma}\label{iteration of close to constant lemma}
	Let $1<p<\infty$. Suppose $u \in HW^{1,p}_{\text{loc}}\left( B_{2R}\right)$ be a weak solution of
		\begin{align*}
		\operatorname{div}_{\mathbb{H}} \left( \left\lvert \X u\right\rvert^{p-2}\X u\right) &= 0 &&\text{ in } B_{2R}. 
	\end{align*} 
	Let $A, B \geq 1$ be constants and suppose  
	\begin{align*}
		\frac{\lambda}{4B} \leq  \left\lvert \X u \right\rvert \leq A\lambda \qquad \text{ in } B_{R}.  
	\end{align*} Then there exist constants $\delta_{0}, \varepsilon, \tau  \in \left(0,1\right)$, depending only on $Q$, $p$, $A,$ $B,$ satisfying 
	\begin{align}\label{log bound sequence}
		0 < \frac{\log \delta_{0}}{2 \log \tau} < 1, 
	\end{align} such that if $\xi^{0} \in \mathbb{R}^{2n}$ is a vector satisfying 
	\begin{enumerate}[(i)]
		\item $ \displaystyle \frac{\lambda}{4B} \leq \left\lvert \xi^{0} \right\rvert \leq A\lambda,$
		\item $ \displaystyle \fint_{B_{R}} \left\lvert \X u - \xi^{0} \right\rvert^{2} \leq \varepsilon A^{2}\lambda^{2},$
	\end{enumerate}
	then there exists a sequence of vectors $\left\lbrace \xi^{s} \right\rbrace_{s \in \mathbb{N}}$ in $\mathbb{R}^{2n}$ such that for all integers $s \geq 0,$ we have 
	\begin{align}\label{lower and upper bound sequence}
		\frac{\lambda}{8B} \leq \left\lvert \xi^{s} \right\rvert \leq 2A\lambda,  
	\end{align}
	\begin{align}
		\fint_{B_{\tau^{s} R}}  \left\lvert \X u - \xi^{s} \right\rvert^{2} &\leq \varepsilon A^{2}\lambda^{2}, \label{smallness at  next level sequence}\\ \intertext{and}
		\int_{B_{\tau^{s+1} R}}  \left\lvert \X u - \xi^{s+1} \right\rvert^{2} &\leq \delta_{0} \tau^{Q} 	\int_{B_{\tau^{s} R}}  \left\lvert \X u - \xi^{s} \right\rvert^{2}. \label{excess decay estimate sequence}
	\end{align}
\end{lemma}
\begin{proof}
	We plan to construct the sequence inductively using Lemma \ref{if xi0 then x1 exists}. To do this, we only need to show that it is possible to choose $\delta_{0} >0$ such that given $\xi^{s},$ the vector $\xi_{s+1}$ constructed by Lemma \ref{if xi0 then x1 exists} also satisfies the weaker bounds 
	\begin{align*}
		\frac{\lambda}{8B} \leq \left\lvert \xi^{s+1} \right\rvert \leq 2A\lambda,   
	\end{align*}
	instead of the significantly worse bounds 
	\begin{align*}
		\frac{\lambda}{8^{s}B} \leq \left\lvert \xi^{s+1} \right\rvert \leq 2^{s}A\lambda.    
	\end{align*}
	This is easy to achieve. Indeed, recalling the estimates in Lemma \ref{if xi0 then x1 exists}, particularly \eqref{xi1-xi0}, we deduce 
	\begin{align*}
		\left\lvert \xi^{s+1} - \xi^{s} \right\rvert^{2} &\leq 2^{Q+2}\tau^{-Q} \fint_{B_{\tau^{s}R}} \left\lvert \X u - \xi^{s} \right\rvert^{2} \\
		&\leq  2^{Q+2}\tau^{-Q}  \delta_{0}^{s}\fint_{B_{R}} \left\lvert \X u - \xi^{0} \right\rvert^{2} \\
		&\leq 2^{Q+2}\tau^{-Q}  \delta_{0}^{s}\varepsilon A^{2}\lambda^{2} \stackrel{\eqref{choice of epsilon}}{\leq} 2^{Q+2}\tau^{\left( \frac{(Q+2)(Q-2)}{2} -Q\right)}\delta_{0}^{s} A^{2}\lambda^{2} \stackrel{\eqref{bound of tau in terms of delta}}{\leq} \delta_{0}^{s+1} A^{2}\lambda^{2}.
	\end{align*}
	Taking square root and summing over $s,$ we deduce 
	\begin{align*}
		\left\lvert \xi^{s+1} - \xi^{0} \right\rvert \leq \left(\sum_{i=1}^{\infty} \delta_{0}^{\frac{i}{2}}\right) A\lambda.
	\end{align*}
 Now, \eqref{lower and upper bound sequence} is satisfied for all $s \geq 0$ with the choice  $\delta_{0}=1/{\left(1+8AB\right)^2}$. 
\end{proof}
Now we are ready to prove our main lemma. 
\begin{lemma}\label{decay estimate in the nondegenerate regime}
	Let $1<p<\infty$. Suppose $u \in HW^{1,p}_{\text{loc}}\left( B_{2R}\right)$ be a weak solution of 
	\begin{align*}
		\operatorname{div}_{\mathbb{H}} \left( \left\lvert \X u\right\rvert^{p-2}\X u\right) &= 0 &&\text{ in } B_{2R}. 
	\end{align*}
	Let $A, B \geq 1$ be constants and suppose  
	\begin{align}\label{lower upper bound full ball main lemma}
		\frac{\lambda}{4B} \leq  \left\lvert \X u \right\rvert \leq A\lambda \qquad \text{ in } B_{R}.  
	\end{align}
	Then there exist constants $\varepsilon_{0}, \alpha_{0} \in (0,1)$ and $C_{0} \geq 1,$ all depending only on $Q,$ $p$, $A$ and $B$ such that if  
	\begin{align}\label{lower upper bound main lemma}
		\frac{\lambda}{8B}	\leq \left\lvert \left( \X u\right)_{\sigma R}\right\rvert \leq A\lambda
	\end{align}
	and 
	\begin{align}\label{oscillation by sup main lemma}
		\fint_{B_{\sigma R}} \left\lvert \X u - \left( \X u\right)_{\sigma R}\right\rvert^{2} \leq \varepsilon_{0} A^{2}\lambda^{2}
	\end{align}
	hold for some constant $0<\sigma\leq 1$, then for any $0 < \rho < \sigma R,$ we have 
	\begin{align}
		\fint_{B_{\rho}} \left\lvert \X u - \left( \X u\right)_{\rho}\right\rvert^{2} \leq C_{0} \sigma^{-Q-2\alpha_{0}} \left(\frac{\rho}{R}\right)^{2\alpha_{0}} \fint_{B_{R}} \left\lvert \X u - \left( \X u\right)_{R}\right\rvert^{2}. 
	\end{align}
\end{lemma}
\begin{proof}
	Set 
	\begin{align*}
		\xi^{0} := \left( \X u\right)_{\sigma R}. 
	\end{align*}
	Now note that \eqref{lower upper bound full ball main lemma} clearly implies 
	\begin{align}\label{lower upper bound sigma ball main lemma}
		\frac{\lambda}{8B} \leq \frac{\lambda}{4B} \leq \left\lvert  \X u \right\rvert \leq A\lambda \qquad \text{ in } B_{\sigma R}. 
	\end{align}
	By \eqref{lower upper bound main lemma} we also have 
	\begin{align}\label{lower upper bound average sigma ball main lemma}
		\frac{\lambda}{8B}\leq   \left\lvert \left( \X u\right)_{\sigma R}\right\rvert \leq A\lambda. 
	\end{align}
	Hence, \eqref{lower upper bound sigma ball main lemma}, \eqref{lower upper bound average sigma ball main lemma} and \eqref{oscillation by sup main lemma} implies that the hypotheses of Lemma \ref{iteration of close to constant lemma} are satisfied for our choice of $\xi^{0}$ in $B_{\sigma R}$ with the same $A$ and $B$ being replaced by $2B,$ if $\varepsilon_{0} =\varepsilon,$ where $\varepsilon \in (0,1)$ is the constant given by Lemma  \ref{iteration of close to constant lemma}. Thus, there exists $\delta_{0}, \tau \in (0,1), $ both depending only on $Q$, $p$, $A$ and $B$ such that there exists a sequence of vectors $\left\lbrace \xi^{s} \right\rbrace_{s \in \mathbb{N}}$ such that for all integers $s \geq 0,$ we have $\frac{\lambda}{16B} \leq \left\lvert \xi^s \right\rvert \leq 4A\lambda$ and
	\begin{align*}
		\fint_{B_{\tau^{s+1} \sigma R}}  \left\lvert \X u - \xi^{s+1} \right\rvert^{2} &\leq \left( \delta_{0}\right)^{s+1} 	\fint_{B_{\sigma R}}  \left\lvert \X u -  \left( \X u\right)_{\sigma R} \right\rvert^{2}. 
	\end{align*}
	Now, let $0 < \rho < \sigma R.$ Then there exists an integer $s_{0} \geq 0$ such that 
	\begin{align*}
		\tau^{s_{0}+1}\sigma R < \rho \leq \tau^{s_{0}}\sigma R. 
	\end{align*}
	Suppose $s_{0} \geq 1.$ Then we have 
	\begin{align*}
		\fint_{B_{\rho}} \left\lvert \X u - \left( \X u\right)_{\rho} \right\rvert^{2} 
		&\leq \fint_{B_{\rho}} \left\lvert \X u - \xi^{s_{0}} \right\rvert^{2} \\
		&\leq \left( \frac{\tau^{s_{0}} \sigma R}{\rho}\right)^{Q} \fint_{B_{\tau^{s_{0}}\sigma R}} \left\lvert \X u -  \xi^{s_{0}} \right\rvert^{2} \\
		&\leq \left( \frac{\tau^{s_{0}} \sigma R}{\rho}\right)^{Q} \left( \delta_{0} \right)^{s_{0}} \fint_{B_{\sigma R}} \left\lvert \X u -  \left( \X u\right)_{\sigma R} \right\rvert^{2} \\
		&\leq \tau^{-Q}\left( \delta_{0} \right)^{s_{0}} \fint_{B_{\sigma R}} \left\lvert \X u -  \left( \X u\right)_{\sigma R} \right\rvert^{2}.
	\end{align*}
	Now set 
	\begin{align*}
		\alpha_{0} := \frac{\log \delta_{0}}{2\log \tau } \in \left( 0, 1\right), \text{ in view of \eqref{log bound sequence}. } 
	\end{align*}
	Now, we deduce 
	\begin{align*}
		\fint_{B_{\rho}} \left\lvert \X u - \left( \X u\right)_{\rho} \right\rvert^{2} 
		&\leq  \tau^{-Q}\left( \delta_{0} \right)^{s_{0}} \left(\frac{\sigma R}{\rho}\right)^{2\alpha_{0}} \left(\frac{\rho}{\sigma R}\right)^{2\alpha_{0}}\fint_{B_{\sigma R}} \left\lvert \X u -  \left( \X u\right)_{\sigma R} \right\rvert^{2} \\
		&\leq  \tau^{-Q}\left( \delta_{0} \right)^{s_{0}} \left(\delta_{0}\right)^{-s_{0}-1} \left(\frac{\rho}{\sigma R}\right)^{2\alpha_{0}}\fint_{B_{\sigma R}} \left\lvert \X u -  \left( \X u\right)_{\sigma R} \right\rvert^{2} \\
		&\leq\tau^{-Q}\delta_{0}^{-1}\left(\frac{\rho}{\sigma R}\right)^{2\alpha_{0}}\fint_{B_{\sigma R}} \left\lvert \X u -  \left( \X u\right)_{R} \right\rvert^{2} \\
		&\leq\tau^{-Q}\delta_{0}^{-1} \sigma^{-Q-2\alpha_{0}} \left(\frac{\rho}{R}\right)^{2\alpha_{0}}\fint_{B_{R}} \left\lvert \X u -  \left( \X u\right)_{R} \right\rvert^{2}.
	\end{align*}
	Thus, we conclude by setting $C = C \left(Q, p, A, B\right):= \tau^{-Q}\delta_{0}^{-1}. $
	If $s_{0} =0,$ then  for the same choice of $\alpha_{0}$ as above,  we clearly have,
	\begin{align*}
		\fint_{B_{\rho}} \left\lvert \X u - \left( \X u\right)_{\rho} \right\rvert^{2} &\leq \fint_{B_{\rho}} \left\lvert \X u -  \left( \X u\right)_{\sigma R} \right\rvert^{2} \\&\leq \left( \frac{\sigma R}{\rho}\right)^{Q -2\alpha_{0}} \left( \frac{\rho}{\sigma R}\right)^{2\alpha_{0}}\fint_{B_{\sigma R}} \left\lvert \X u -  \left( \X u\right)_{\sigma R} \right\rvert^{2} \\
		&\leq \tau^{2\alpha_{0} -Q}\sigma^{-Q-2\alpha_{0}} \left(\frac{\rho}{R}\right)^{2\alpha_{0}}\fint_{B_{R}} \left\lvert \X u -  \left( \X u\right)_{R} \right\rvert^{2}.
	\end{align*}
	Thus, we have the estimate in this case as well for the same $$C = C \left(Q, p, A, B\right):= \tau^{-Q}\delta_{0}^{-1} = \tau^{2\alpha_{0} -Q}.$$ This completes the proof. 
\end{proof}
We also have the following result. 
\begin{lemma}\label{decay estimate in the nondegenerate regime for V}
	Let $1<p<\infty$. Suppose $u \in HW^{1,p}_{\text{loc}}\left( B_{2R}\right)$ be a weak solution of 
	\begin{align*}
		\operatorname{div}_{\mathbb{H}} \left( \left\lvert \X u\right\rvert^{p-2}\X u\right) &= 0 &&\text{ in } B_{2R}. 
	\end{align*}
	Let $A, B \geq 1$ be constants and suppose  
	\begin{align}\label{lower upper bound full ball main lemma V}
		\frac{\lambda}{4B} \leq  \left\lvert V \left( \X u \right) \right\rvert^{\frac{2}{p}} \leq A\lambda \qquad \text{ in } B_{R}.  
	\end{align}
	Then there exist constants $\theta_{0}, \alpha_{0} \in (0,1)$ and $C_{1} \geq 1,$ all depending only on $Q$, $p$, $A$ and $B$, such that if	
	\begin{align}\label{lower upper bound main lemma V}
		\frac{\lambda}{8B}	\leq \left\lvert \left( V \left(\X u\right)\right)_{\sigma R}\right\rvert^{\frac{2}{p}} \leq A\lambda
	\end{align}
	and 
	\begin{align}\label{oscillation by sup main lemma V}
		\fint_{B_{\sigma R}} \left\lvert V \left( \X u \right)  - \left( V \left( \X u\right) \right)_{\sigma R}\right\rvert^{2} \leq \theta_{0} A^{p}\lambda^{p},
	\end{align}
	hold for some constant $0<\sigma \leq 1$, then for any $0 < \rho < \sigma R,$ we have 
	\begin{align}
		\fint_{B_{\rho}} \left\lvert V \left( \X u \right) - \left( V \left( \X u \right)\right)_{\rho}\right\rvert^{2} \leq C_{1} \sigma^{-Q-2\alpha_{0}} \left(\frac{\rho}{R}\right)^{2\alpha_{0}} \fint_{B_{R}} \left\lvert V \left( \X u \right) - \left( V \left( \X u\right) \right)_{R}\right\rvert^{2}. 
	\end{align}
\end{lemma}
\begin{proof}
	Set 
	\begin{align}\label{choice of xi0 for V}
		\xi^{0} := V^{-1} \left( \left( V \left(\X u\right)\right)_{\sigma R} \right). 
	\end{align}
	Observe that this implies 
	\begin{align*}
		\left\lvert \xi^{0} \right\rvert^{\frac{p}{2}} = \left\lvert \left( V \left(\X u\right)\right)_{\sigma R}\right\rvert. 
	\end{align*}
	Thus, clearly, we have the bounds 
	\begin{align*}
		\frac{\lambda}{8B}	\leq \left\lvert \xi^{0}\right\rvert \leq A\lambda. 
	\end{align*}
	As $\left\lvert V \left( \X u \right) \right\rvert^{\frac{2}{p}}  = \left\lvert \X u \right\rvert,$ we also have the obvious bounds 
	\begin{align*}
		\frac{\lambda}{8B} \leq	\frac{\lambda}{4B} \leq  \left\lvert  \X u  \right\rvert \leq A\lambda \qquad \text{ in } B_{\sigma R}. 
	\end{align*}
	Now, using \eqref{constant cv}, we have 
	\begin{align*}
		\fint_{B_{\sigma R}} \left\lvert \X u - \xi^{0} \right\rvert^{2} \leq c_{0}\fint_{B_{\sigma R}} \left\lvert V \left(\X u \right) - V \left(\xi^{0} \right) \right\rvert^{2} \leq c_{0}\theta_{0} A^{2}\lambda^{2},
	\end{align*} 
	for a constant $c_{0}= c_{0}\left(Q, p, A, B \right)>0. $ We set 
	\begin{align*}
		\theta_{0} = \frac{\varepsilon}{c_{0}}, 
	\end{align*}
	where $\varepsilon$ is the constant given by Lemma \ref{iteration of close to constant lemma} for our choice of $\xi^{0}$ in $B_{\sigma R}$ with the same $A$ and $B$ being replaced by $2B.$ Thus, there exists $\delta_{0}, \tau \in (0,1), $ both depending only on $Q$, $p$, $A$ and $B$ such that there exists a sequence of vectors $\left\lbrace \xi^{s} \right\rbrace_{s \in \mathbb{N}}$ such that for all integers $s \geq 0,$ we have 
	\begin{align}\label{lower and upper bound sequence V}
		\frac{\lambda}{16B} \leq \left\lvert \xi^{s} \right\rvert \leq 2A\lambda  
	\end{align} and 
	\begin{align*}
		\fint_{B_{\tau^{s+1} \sigma R}}  \left\lvert \X u - \xi^{s+1} \right\rvert^{2} &\leq \left( \delta_{0}\right)^{s+1} 	\fint_{B_{\sigma R}}  \left\lvert \X u -  \xi^{0} \right\rvert^{2}. 
	\end{align*}
	This implies, again by \eqref{constant cv}, for some constant $C\equiv C(Q,p, A, B)>0$ we have 
	\begin{align*}
		\fint_{B_{\tau^{s+1} \sigma R}}  \left\lvert V \left( \X u \right) - V \left( \xi^{s+1} \right)  \right\rvert^{2}  &\leq C \left( \lambda\right)^{2-p}	\fint_{B_{\tau^{s+1} \sigma R}}  \left\lvert \X u - \xi^{s+1} \right\rvert^{2} \\ &\leq C \left( \delta_{0}\right)^{s+1} 	\fint_{B_{\sigma R}} \lambda^{2-p} \left\lvert \X u -  \xi^{0} \right\rvert^{2} \\
		&\leq C \left( \delta_{0}\right)^{s+1} 	\fint_{B_{\sigma R}}  \left\lvert V \left( \X u \right) -  V \left( \xi^{0} \right) \right\rvert^{2}.
	\end{align*}
	Since $V \left(\xi^0\right) = \left( V \left(\X u\right)\right)_{\sigma R},$ we arrive at 
	\begin{align*}
		\fint_{B_{\tau^{s+1} \sigma R}}  \left\lvert V \left( \X u \right) - V \left( \xi^{s+1} \right)  \right\rvert^{2} \leq C \left( \delta_{0}\right)^{s+1} 	\fint_{B_{\sigma R}}  \left\lvert V \left( \X u \right) - \left( V \left(\X u\right)\right)_{\sigma R}  \right\rvert^{2}. 
	\end{align*}
	From this, via an interpolation argument similar to the one used above in Lemma \ref{decay estimate in the nondegenerate regime} , we have 
	\begin{align*}
		\fint_{B_{\rho}}  \left\lvert V \left( \X u \right) - V \left( \xi^{s+1} \right)  \right\rvert^{2} \leq C \sigma^{-Q-2\alpha_{0}}\left( \frac{\rho}{R}\right)^{2\alpha_{0}} 	\fint_{B_{R}}  \left\lvert V \left( \X u \right) - \left( V \left(\X u\right)\right)_{R}  \right\rvert^{2}. 
	\end{align*}
	This completes the proof. \end{proof}
\subsection{Proofs of Theorem \ref{Main excess decay estimate Intro} and Theorem \ref{Main excess decay estimate V Intro}}
We now prove our crucial excess decay estimates. 
\begin{proof}{\textbf{( of Theorem \ref{Main excess decay estimate Intro}} ) }
	We can and henceforth assume $x_{0} =0.$ Note that, it is enough to prove \eqref{Main excess decay estimate Intro estimate}, for $0<\rho<\kappa R$, for some constant $\kappa\in (0,1)$ depending only $n$ and $p$, as the estimate is trivial otherwise.  First, we fix $C_{RH}>0$ which is given by Lemma \ref{reverse holder inequality lemma} with the choice $A=B=1$. So, clearly $C_{RH}$ depends only on $n$ and $p$. Next, let $B=B\left(n, p\right)\geq 1$ be defined by $B:= 2c_{1},$ where $c_{1}$ is the constant given by Theorem \ref{thm:lip Lq bound} with the choice $q=1$. Let $\varepsilon_{0}, \alpha_{0} \in (0,1)$ be the constants given by Lemma  \ref{decay estimate in the nondegenerate regime} for the choices of $A=2C_{RH}$  and $B=2c_{1}.$ Let $\alpha_{1} \in (0,1)$ be the exponent in Theorem \ref{thm:holder}. Set 
	\begin{align*}
		\alpha := \min \left\lbrace \alpha_{0}, \alpha_{1} \right\rbrace. 
	\end{align*}  Now we claim the following. 
	\begin{claim}\label{sup controls oscillation claim}
		Without loss of generality, we can assume that the we have 
		\begin{align}\label{lower upper for average}
			\frac{1}{B} \left\lVert \X u \right\rVert_{\infty, R/2} \leq \left\lvert \left(\X u\right)_{R}\right\rvert \leq 4 \left\lVert \X u \right\rVert_{\infty, R/2}
		\end{align}
		and 
		\begin{align}\label{oscillation is smaller than sup}
			\left( \fint_{B_{R/2}} \left\lvert \X u - \left(\X u\right)_{R} \right\rvert^{q} \right)^{\frac{1}{q}}\leq \varepsilon_{0} \left\lVert \X u \right\rVert_{\infty, R/2}. 
		\end{align}
	\end{claim}
	\emph{Proof of Claim \ref{sup controls oscillation claim}: } Indeed, if \eqref{oscillation is smaller than sup} is violated, then by Theorem \ref{thm:holder}, for any $0 < \rho < R/2,$ we have 
	\begin{align*}
		\left( \fint_{B_{\rho}} \left\lvert \X u - \left( \X u\right)_{\rho}\right\rvert^{q}\, dx \right)^{\frac{1}{q}}
		&\leq \sup\limits_{x, y \in B_{\rho}} \left\lvert \X u \left(x\right) - \X u \left(y\right) \right\rvert \\
		&\leq c_h\left(\frac{\rho}{R}\right)^{\alpha_{1}}\left\lVert \X u \right\rVert_{\infty, R/2} \\&< \frac{c_h}{\varepsilon_{0}}\left(\frac{\rho}{R}\right)^{\alpha} \left( \fint_{B_{R/2}} \left\lvert \X u - \left(\X u\right)_{R} \right\rvert^{q} \right)^{\frac{1}{q}}\\
		&\leq \frac{2^{Q}c_h}{\varepsilon_{0}}\left(\frac{\rho}{R}\right)^{\alpha} \left( \fint_{B_{R}} \left\lvert \X u - \left(\X u\right)_{R} \right\rvert^{q} \right)^{\frac{1}{q}},
	\end{align*}
	which implies our desired estimate. So we can assume \eqref{oscillation is smaller than sup} holds. Now note that 
	\begin{align*}
		\left\lvert \left(\X u\right)_{R}\right\rvert^{2} &= \left\langle \left(\X u\right)_{R}, \left(\X u\right)_{R}\right\rangle =\left\langle \left(\X u\right)_{R}, \left(\X u\right)_{R} -\X u \left(x\right)\right\rangle + \left\langle \left(\X u\right)_{R}, \X u \left(x\right)\right\rangle ,  
	\end{align*}
	for a.e. $x \in B_{R/2}.$ Using Cauchy-Schwarz, integrating and finally using Young's inequality, we deduce 
	\begin{align*}
		\left\lvert \left(\X u\right)_{R}\right\rvert^{2} &\leq \left\lvert \left(\X u\right)_{R}\right\rvert  \left( \fint_{B_{R/2}} \left\lvert \X u - \left(\X u\right)_{R} \right\rvert + \fint_{B_{R/2}} \left\lvert \X u \right\rvert \right) \\
		&\stackrel{\eqref{oscillation is smaller than sup}}{\leq} \left(\varepsilon_{0} +1 \right)\left\lvert \left(\X u\right)_{R}\right\rvert \left\lVert \X u \right\rVert_{\infty, R/2} \\
		&\leq \frac{1}{2}\left\lvert \left(\X u\right)_{R}\right\rvert^{2} + 2\left(\varepsilon_{0} +1 \right)^{2}\left\lVert \X u \right\rVert_{\infty, R/2}^{2}. 
	\end{align*}
	Hence, we deduce 
	\begin{align*}
		\left\lvert \left(\X u\right)_{R}\right\rvert \leq 2 \left(\varepsilon_{0} + 1\right) \left\lVert \X u \right\rVert_{\infty, R/2} \leq 4\left\lVert \X u \right\rVert_{\infty, R/2}. 
	\end{align*}
	Thus, the only way \eqref{lower upper for average} can fail is if we have 
	\begin{align*}
		\left\lvert \left(\X u\right)_{R}\right\rvert  <	\frac{1}{B} \left\lVert \X u \right\rVert_{\infty, R/2}. 
	\end{align*}
	But then  by Theorem \ref{thm:lip Lq bound} (with `$q=1$'), H\"{o}lder's inequality and the triangle inequality for the $L^{q}$ norm, we find,   
	\begin{align*}
		\left\lVert \X u \right\rVert_{\infty, R/2} \leq c_{1} \left(\fint_{B_{R}} \left\lvert \X u \right\rvert^{q}\right)^{\frac{1}{q}} &\leq c_{1} \left(\fint_{B_{R}} \left\lvert \X u  -  \left(\X u\right)_{R}\right\rvert^{q}\right)^{\frac{1}{q}} + c_{1} \left\lvert \left(\X u\right)_{R}\right\rvert \\ &< c_{1} \left(\fint_{B_{R}} \left\lvert \X u  -  \left(\X u\right)_{R}\right\rvert^{q}\right)^{\frac{1}{q}} + \frac{c_{1}}{B} 	\left\lVert \X u \right\rVert_{\infty, R/2}. 
	\end{align*}
	Thus, recalling $B=2c_1$, we have 
	\begin{align*}
		\left\lVert \X u \right\rVert_{\infty, R/2} \leq B \left( \fint_{B_{R}} \left\lvert \X u - \left(\X u\right)_{R} \right\rvert^{q}\right)^{\frac{1}{q}}. 
	\end{align*}
	Hence, as before, for any $0<\rho<R/2$, we have 
	\begin{align*}
		\left( \fint_{B_{\rho}} \left\lvert \X u - \left( \X u\right)_{\rho}\right\rvert^{q}\, dx \right)^{\frac{1}{q}}&\leq c\left(\frac{\rho}{R}\right)^{\alpha_{1}}\left\lVert \X u \right\rVert_{\infty, R/2}\\&\leq cB\left(\frac{\rho}{R}\right)^{\alpha} \left( \fint_{B_{R}} \left\lvert \X u - \left(\X u\right)_{R} \right\rvert^{q}\right)^{\frac{1}{q}},
	\end{align*}
	which again yields the desired estimate. This completes the proof of Claim \ref{sup controls oscillation claim}.  
	
	From now on, we shall assume that \eqref{lower upper for average} and \eqref{oscillation is smaller than sup} holds. Now we claim
	\begin{claim}\label{the alternatives claim}
		There exists a constant $\tau_{1} \in (0,1),$ depending only on $n$ and $p$ such that at least one of the following alternatives occur. 
		\begin{enumerate}[(i)]
			\item \textbf{The nondegenerate alternative: } We have 
			\begin{align}\label{the nondegenerate alt}
				\left\lvert \X u \right\rvert \geq \frac{1}{4}\left\lVert \X u \right\rVert_{\infty, R/2} \qquad \text{ in  } B_{\tau_{1} R/2}. 
			\end{align}
			\item \textbf{The degenerate alternative: } We have 
			\begin{align}\label{the degenerate alt}
				\left\lVert \X u \right\rVert_{\infty, \tau_{1} R/2} \leq \frac{1}{2}\left\lVert \X u \right\rVert_{\infty, R/2}. 
			\end{align}
		\end{enumerate}
	\end{claim} 
	\emph{Proof of Claim \ref{the alternatives claim}: } By Theorem \ref{thm:holder}, for any $0 < \tau < 1,$ we have 
	\begin{align*}
		\sup\limits_{x, y \in B_{\tau R/2}} \left\lvert \X u \left(x\right) - \X u \left(y\right)\right\rvert \leq c_{h}\tau^{\alpha_{1}}\left\lVert \X u \right\rVert_{\infty, R/2}. 
	\end{align*}
	Now, choose $\tau_{1} \in (0,1)$ sufficiently small such that 
	\begin{align*}
		c_{h}\tau_{1}^{\alpha_{1}} = \frac{1}{4}. 
	\end{align*}
	Now, if \eqref{the nondegenerate alt} is violated, then there exists a point $ \bar{x} \in B_{\tau_{1} R/2}$ such that 
	\begin{align*}
		\left\lvert \X u \left( \bar{x} \right) \right\rvert < \frac{1}{4}\left\lVert \X u \right\rVert_{\infty, R/2} .  
	\end{align*}
	But then, for any $x \in B_{\tau_{1} R/2},$ by Theorem \ref{thm:holder}, we deduce 
	\begin{align*}
		\left\lvert \X u \left( x \right) \right\rvert &\leq 	\left\lvert \X u \left( \bar{x} \right) \right\rvert + 	\left\lvert \X u \left( x \right) - \X u \left( \bar{x} \right) \right\rvert \\
		&\leq \frac{1}{4}\left\lVert \X u \right\rVert_{\infty, R/2} + 	c_{h}\tau_{1}^{\alpha_{1}} \left\lVert \X u \right\rVert_{\infty, R/2} = \frac{1}{2}\left\lVert \X u \right\rVert_{\infty, R/2}. 
	\end{align*}
	This completes the proof of Claim \ref{the alternatives claim}.  
	
	Now we define 
	\begin{align*}
		R_{i}:= \tau_{1}^{i}\frac{R}{2} \qquad \text{ for all integers } i \geq 0. 
	\end{align*}
	Thus, by Claim \ref{the alternatives claim}, at each level $i$, at least one of the following alternatives must occur. 
	\begin{enumerate}[(a)]
		\item \textbf{The nondegenerate alternative: } We have 
		\begin{align}\label{the nondegenerate alt at level i}
			\left\lvert \X u \right\rvert \geq \frac{1}{4}\left\lVert \X u \right\rVert_{\infty, R_{i}} \qquad \text{ in  } B_{R_{i+1}}. 
		\end{align}
		\item \textbf{The degenerate alternative: } We have 
		\begin{align}\label{the degenerate alt at level i}
			\left\lVert \X u \right\rVert_{\infty, R_{i+1}} \leq \frac{1}{2}\left\lVert \X u \right\rVert_{\infty, R_{i}}. 
		\end{align}
	\end{enumerate}
	Next, we choose $i_{a} \geq 1$ to be first integer for which we have 
	\begin{align}\label{choice of ia}
		2^{i_{a}-1} \geq B. 
	\end{align}
	Clearly, $i_a$ depends only on $B$ and thus ultimately, only on $n$ and $p$.
	\begin{claim}\label{nondegenerate case must occur claim}
		Without loss of generality, we can assume that \eqref{the nondegenerate alt at level i} occurs for the first time for some $0 \leq i < i_{a}.$ 
	\end{claim}
	\emph{Proof of Claim \ref{nondegenerate case must occur claim}: }  If Claim \ref{nondegenerate case must occur claim} is false, then \eqref{the degenerate alt at level i} must have occured for all $ 0 \leq i \leq i_{a}-1.$ But this implies,  we have,  
	\begin{align*}
		\left\lVert \X u \right\rVert_{\infty, R_{i_{a}}} &\leq \left(\frac{1}{2}\right)^{i_{a}}\left\lVert \X u \right\rVert_{\infty, R/2}.  
	\end{align*}
	This now implies 
	\begin{align}\label{osc control energy in very small ball}
		\left( \fint_{B_{R_{i_{a}}}} \left\lvert \X u - \left( \X u \right)_{R}\right\rvert^{q}  \right)^{\frac{1}{q}} \geq \left( \fint_{B_{R_{i_{a}}}} \left\lvert \X u \right\rvert^{q} \right)^{\frac{1}{q}}.
	\end{align}
	Indeed, if \eqref{osc control energy in very small ball} is false, then recalling \eqref{lower upper for average}, we have 
	\begin{align*}
		\frac{1}{B} \left\lVert \X u \right\rVert_{\infty, R/2} &\leq \left\lvert \left(\X u\right)_{R}\right\rvert \\
		&=\fint_{B_{R_{i_{a}}}}\left\lvert \left(\X u\right)_{R}\right\rvert\ \mathrm{d}x \\
		&\stackrel{\text{H\"{o}lder}}{\leq} \left( \fint_{B_{R_{i_{a}}}}\left\lvert \X u - \left(\X u\right)_{R}\right\rvert^{q}\ \mathrm{d}x \right)^{\frac{1}{q}}+ \left( \fint_{B_{R_{i_{a}}}} \left\lvert \X u \right\rvert^{q} \right)^{\frac{1}{q}} \\
		&< 2 \left( \fint_{B_{R_{i_{a}}}} \left\lvert \X u \right\rvert^{q} \right)^{\frac{1}{q}}\leq 2\left\lVert \X u \right\rVert_{\infty, R_{i_{a}}} \leq 2\left(\frac{1}{2}\right)^{i_{a}}\left\lVert \X u \right\rVert_{\infty, R/2}, 
	\end{align*}
	which contradicts \eqref{choice of ia}. Hence \eqref{osc control energy in very small ball} holds. Now using Theorem \ref{thm:lip Lq bound} (with `$q=1$') together with \eqref{osc control energy in very small ball} and the fact $1\leq q\leq 2$, we find  
	\begin{align*}
		\left\lVert \X u \right\rVert_{\infty, R_{i_{a}}/2} \leq c_{1} \fint_{B_{R_{i_{a}}}} \left\lvert \X u \right\rvert &\leq c_{1} \left( \fint_{B_{R_{i_{a}}}} \left\lvert \X u \right\rvert^{q}\right)^{\frac{1}{q}} \\
		&\leq c_{1} \left( \fint_{B_{R_{i_{a}}}} \left\lvert \X u - \left( \X u \right)_{R}\right\rvert^{q} \right)^{\frac{1}{q}} \\
		&\leq c_{1} \left( \frac{R}{R_{i_{a}}}\right)^{\frac{Q}{q}}\left( \fint_{B_{R}} \left\lvert \X u - \left( \X u \right)_{R}\right\rvert^{q} \right)^{\frac{1}{q}} \\
		&\leq c_{1}2^Q\tau_{1}^{-i_{a}Q}\left( \fint_{B_{R}} \left\lvert \X u - \left( \X u \right)_{R}\right\rvert^{q} \right)^{\frac{1}{q}}.  
	\end{align*}
	Combining this with Theorem \ref{thm:holder}, for any $0 < \rho < R_{i_{a}}/2,$ we have 
	\begin{align*}
		\left( \fint_{B_{\rho}} \left\lvert \X u - \left( \X u\right)_{\rho}\right\rvert^{q} \, dx \right)^{\frac{1}{q}} 
		&\leq \sup\limits_{x, y \in B_{\rho}} \left\lvert \X u \left(x\right) - \X u \left(y\right) \right\rvert \\
		&\leq c_h\left(\frac{\rho}{R_{i_{a}}}\right)^{\alpha_{1}}  \left\lVert \X u \right\rVert_{\infty, R_{i_{a}}/2} \\
		&\leq c_h2^{\alpha_1}\tau_1^{-\alpha_{1}i_{a}}\left(\frac{\rho}{R}\right)^{\alpha_{1}}\left\lVert \X u \right\rVert_{\infty, R_{i_{a}}/2} \\
		&\leq c_hc_{1}2^{2Q}\tau_1^{-i_{a}\left(Q + \alpha_{1}\right)}\left(\frac{\rho}{R}\right)^{\alpha_{1}}\left( \fint_{B_{R}} \left\lvert \X u - \left( \X u \right)_{R}\right\rvert^{q} \right)^{\frac{1}{q}}. 
	\end{align*}
	This together with the fact that if $\frac{R_{i_a}}{2} \leq \rho \leq R$, then $\frac{\tau_1^{i_a}}{2}\leq \frac{\rho}{R} \leq 1$, yields our desired estimate. This proves Claim \ref{nondegenerate case must occur claim}.  

	Thus, we can assume that there exists a smallest integer $0 \leq i_{0} < i_{a}$ such that \eqref{the nondegenerate alt at level i} occurs for  $i=i_{0}.$ Thus, we have 
	\begin{align}\label{lower bound in any smaller ball}
		\frac{1}{4}\left\lVert \X u \right\rVert_{\infty, R_{i_{0}}}  \leq \left\lvert \X u \right\rvert \leq \left\lVert \X u \right\rVert_{\infty, R_{i_{0}}} \qquad \text{ in } B_{R_{i_{0}+1}}.
	\end{align}
	Now we claim 
	\begin{claim}\label{existence of sigma claim}
		There exists a constant $\sigma \in (0,1),$ depending only on $n$, $p$ and $\varepsilon_{0},$ and thus ultimately only on $n$ and $p$,  such that we have 
		\begin{align}\label{oscillation by sup in sigma ball}
			\left( \fint_{B_{\sigma R_{i_{0}+1}}} \left\lvert \X u - \left( \X u\right)_{\sigma R_{i_{0}+1}}\right\rvert^{q} \right)^{\frac{1}{q}} \leq \frac{\varepsilon_{0}}{16} \left\lVert \X u \right\rVert_{\infty, R_{i_{0}}}.
		\end{align}
	\end{claim}
	\emph{Proof of Claim \ref{existence of sigma claim}:} Let $c_h>0$ be the constant given by Theorem \ref{thm:holder}. Choose $\sigma \in (0,1)$ sufficiently small such that 
	\begin{align*}
		c_{h}\sigma^{\alpha_{1}} = \frac{\varepsilon_{0}}{16}. 
	\end{align*}
	Then, by Theorem \ref{thm:holder}, we deduce 
	\begin{align*}
		\left( \fint_{B_{\sigma R_{i_{0}+1}}} \left\lvert \X u - \left( \X u\right)_{\sigma R_{i_{0}+1}}\right\rvert^{q} \right)^{\frac{1}{q}} 
		&\leq \sup\limits_{x, y \in B_{\sigma R_{i_{0}+1}}} \left\lvert \X u \left( x\right) - \X u \left( y\right)\right\rvert \\
		&\leq  	c_{h}\sigma^{\alpha_{1}} \left\lVert \X u \right\rVert_{\infty, R_{i_{0}+1}} \\
		&\leq \frac{\varepsilon_{0}}{16}\left\lVert \X u \right\rVert_{\infty, R_{i_{0}+1}}  \leq \frac{\varepsilon_{0}}{16}\left\lVert \X u \right\rVert_{\infty, R_{i_{0}}}.
	\end{align*}
	This proves Claim \ref{existence of sigma claim}.  
	
	Now we are in a position to complete the proof of the Theorem. Observe that exactly one of the following can occur. Either we have  
	\begin{align}\label{lower bound for sigma average}
		\left\lvert \left( \X u\right)_{\frac{1}{2}\sigma R_{i_{0}+1}}\right\rvert  \geq 	\frac{1}{8B}\left\lVert \X u \right\rVert_{\infty, R_{i_{0}}},
	\end{align}
	or we must have 
	\begin{align}\label{no lower bound for sigma average}
		\left\lvert \left( \X u\right)_{\frac{1}{2}\sigma R_{i_{0}+1}}\right\rvert  < 	\frac{1}{8B}\left\lVert \X u \right\rVert_{\infty, R_{i_{0}}}.
	\end{align}
	If \eqref{no lower bound for sigma average} holds, then by Theorem \ref{thm:lip Lq bound} (with `$q=1$'), we have 
	\begin{align}
		\left\lVert \X u \right\rVert_{\infty, \frac{1}{4}\sigma R_{i_{0}+1}} &\leq c_{1}\fint_{B_{\frac{1}{2}\sigma R_{i_{0}+1}}} \left\lvert \X u \right\rvert \notag\\
		&\leq c_{1}\left( \fint_{B_{\frac{1}{2}\sigma R_{i_{0}+1}}} \left\lvert \X u \right\rvert^{q}\right)^{\frac{1}{q}} \notag\\
		&\leq c_{1}\left( \fint_{B_{\frac{1}{2}\sigma R_{i_{0}+1}}} \left\lvert \X u  - \left( \X u\right)_{\frac{1}{2}\sigma R_{i_{0}+1}}\right\rvert^{q}\right)^{\frac{1}{q}} + c_{1}\left\lvert \left( \X u\right)_{\frac{1}{2}\sigma R_{i_{0}+1}} \right\rvert \notag\\
		&\leq c_{1}\left( \fint_{B_{\frac{1}{2}\sigma R_{i_{0}+1}}} \left\lvert \X u  - \left( \X u\right)_{\frac{1}{2}\sigma R_{i_{0}+1}}\right\rvert^{q}\right)^{\frac{1}{q}} + \frac{c_{1}}{8B}\left\lVert \X u \right\rVert_{\infty, R_{i_{0}}} \notag\\
		&\leq c_{1}\left( \fint_{B_{\frac{1}{2}\sigma R_{i_{0}+1}}} \left\lvert \X u  - \left( \X u\right)_{\frac{1}{2}\sigma R_{i_{0}+1}}\right\rvert^{q}\right)^{\frac{1}{q}} + \frac{c_{1}}{2B}\left\lVert \X u \right\rVert_{\infty, \frac{1}{4}\sigma R_{i_{0}+1}} , \label{prelim estimate for sup in 1/2sigma}
	\end{align}
	where in the last line, we have used the bound 
	\begin{align}\label{lower sup controls upper sup}
		\frac{1}{4}\left\lVert \X u \right\rVert_{\infty, R_{i_{0}}} \leq \left\lVert \X u \right\rVert_{\infty, \frac{1}{4}\sigma R_{i_{0}+1}},
	\end{align}
	which is implied by \eqref{lower bound in any smaller ball}. Now by our choice of $B=2c_1$, \eqref{prelim estimate for sup in 1/2sigma} implies 
	\begin{align*}
		\left\lVert \X u \right\rVert_{\infty, \frac{1}{4}\sigma R_{i_{0}+1}/2} \leq \frac{4c_{1}}{3} \left( \fint_{B_{\frac{1}{2}\sigma R_{i_{0}+1}}} \left\lvert \X u  - \left( \X u\right)_{\frac{1}{2}\sigma R_{i_{0}+1}}\right\rvert^{q}\right)^{\frac{1}{q}}.  
	\end{align*}
	Once again, combined with Theorem \ref{thm:holder}, this implies, for any $0 < \rho \le\sigma R_{i_{a}}/4\le \sigma R_{i_{0}+1}/4,$ we have  
	\begin{align}
		&\left( \fint_{B_{\rho}} \left\lvert \X u - \left( \X u\right)_{\rho}\right\rvert^{q}\, dx\right)^{\frac{1}{q}} \notag\\&\qquad \qquad \leq c_h\left(\frac{\rho}{\sigma R_{i_{0}+1}}\right)^{\alpha_{1}}\left\lVert \X u \right\rVert_{\infty, \frac{1}{4}\sigma R_{i_{0}+1}} \notag\\&\qquad \qquad\leq \frac{4c_hc_{1}}{3}\left(\frac{\rho}{\sigma R_{i_{0}+1}}\right)^{\alpha_{1}} \left( \fint_{B_{\frac{1}{2}\sigma R_{i_{0}+1}}} \left\lvert \X u  - \left( \X u\right)_{\frac{1}{2}\sigma R_{i_{0}+1}}\right\rvert^{q} \right)^{\frac{1}{q}} \notag
		\\&\qquad \qquad\stackrel{\eqref{minimality of mean}}{\leq} \frac{8c_hc_{1}}{3}\left(\frac{\rho}{\sigma R_{i_{0}+1}}\right)^{\alpha_{1}} \left( \fint_{B_{\frac{1}{2}\sigma R_{i_{0}+1}}} \left\lvert \X u  - \left( \X u\right)_{R}\right\rvert^{q} \right)^{\frac{1}{q}} \notag
		\\		&\qquad \qquad\leq c\left(\frac{R}{\sigma R_{i_{0}+1}}\right)^{\frac{Q}{q} +\alpha_{1}} \left(\frac{\rho}{R}\right)^{\alpha_{1}} \left( \fint_{B_{R}} \left\lvert \X u  - \left( \X u\right)_{R}\right\rvert^{q}\right)^{\frac{1}{q}}. \label{prefinal estimate for with no lower bound on sigma average}
	\end{align}
	Now, since $\sigma,$ $\tau_{1}, \alpha_{1}$ are fixed constants and $i_{a}$ is a fixed integer, all depending only on $n$ and $p,$ with $\tau_{1} \in (0,1)$ and $i_{0} < i_{a},$ we note that 
	\begin{align*}
		\left(\frac{R}{\sigma R_{i_{0}+1}}\right)^{\frac{Q}{q} +\alpha_{1}} = 	\left(\frac{1}{\sigma \tau_{1}^{i_{0}+1}}\right)^{\frac{Q}{q} +\alpha_{1}} \leq \left(\frac{1}{\sigma \tau_{1}^{i_{a}}}\right)^{\frac{Q}{q} +\alpha_{1}} \leq \left(\frac{1}{\sigma \tau_{1}^{i_{a}}}\right)^{Q +\alpha_{1}} ,
	\end{align*}
	where the last term on the right is again a fixed constant, depending only on $n$ and $p.$ Hence, \eqref{prefinal estimate for with no lower bound on sigma average} now easily yields our desired estimate. Hence it only remains to settle the case when \eqref{lower bound for sigma average} holds. Observe that \eqref{oscillation by sup in sigma ball} implies 
	\begin{align*}
		\left( \fint_{B_{\sigma R_{i_{0}+1}}} \left\lvert \X u - \left( \X u\right)_{\sigma R_{i_{0}+1}}\right\rvert^{q} \right)^{\frac{1}{q}} \leq \frac{\varepsilon_{0}}{16} \left\lVert \X u \right\rVert_{\infty, R_{i_{0}}} \leq \varepsilon_{0} \left\lVert \X u \right\rVert_{\infty, R_{i_{0}+1}},
	\end{align*}
	where once again we have used the bound \eqref{lower sup controls upper sup}. Combining this with the H\"{o}lder and reverse H\"{o}lder inequality, we deduce  
	\begin{align*}
		&\left( \fint_{B_{\frac{1}{2}\sigma R_{i_{0}+1}}} \left\lvert \X u - \left( \X u\right)_{\frac{1}{2}\sigma R_{i_{0}+1}}\right\rvert^{2} \right)^{\frac{1}{2}} \\&\qquad \qquad  \stackrel{\eqref{minimality of mean}}{\leq} 2\left( \fint_{B_{\frac{1}{2}\sigma R_{i_{0}+1}}} \left\lvert \X u - \left( \X u\right)_{\sigma R_{i_{0}+1}}\right\rvert^{2} \right)^{\frac{1}{2}}\notag 
		\\&\qquad \qquad \stackrel{\text{H\"{o}lder}}{\leq} 2\left( \fint_{B_{\frac{1}{2}\sigma R_{i_{0}+1}}} \left\lvert \X u - \left( \X u\right)_{\sigma R_{i_{0}+1}}\right\rvert^{\frac{2Q}{Q-2}} \right)^{\frac{Q-2}{2Q}}\\&\qquad \qquad \stackrel{\eqref{lower bound in any smaller ball},\eqref{rev holder bound 1 to 2 star}}{\leq} 2C_{RH}\fint_{B_{\sigma R_{i_{0}+1}}} \left\lvert \X u - \left( \X u\right)_{\sigma R_{i_{0}+1}}\right\rvert \\&\qquad \qquad \stackrel{\text{H\"{o}lder}}{\leq} 2C_{RH}\left( \fint_{B_{\sigma R_{i_{0}+1}}} \left\lvert \X u - \left( \X u\right)_{\sigma R_{i_{0}+1}}\right\rvert^{q} \right)^{\frac{1}{q}} \leq 2C_{RH}\varepsilon_{0} \left\lVert \X u \right\rVert_{\infty, R_{i_{0}+1}}.
	\end{align*}
		Thus, recalling $\varepsilon_{0} \in (0,1),$ we have 
	\begin{align}\label{oscillation control for using main lemma}
		\fint_{B_{\frac{1}{2}\sigma R_{i_{0}+1}}} \left\lvert \X u - \left( \X u\right)_{\frac{1}{2}\sigma R_{i_{0}+1}}\right\rvert^{2} &\leq \varepsilon_{0}^{2} (2C_{RH})^2\left\lVert \X u \right\rVert_{\infty,  R_{i_{0}+1}}^{2} \notag\\
		&\leq \varepsilon_{0} (2C_{RH})^2\left\lVert \X u \right\rVert_{\infty,  R_{i_{0}+1}}^{2}. 
	\end{align}
	By virtue of  \eqref{lower bound in any smaller ball}, we also trivially have 
	\begin{align}\label{upper and lower bound for Xu}
		\frac{1}{4B}\left\lVert \X u \right\rVert_{\infty, R_{i_{0}+1}}	\leq \frac{1}{4}\left\lVert \X u \right\rVert_{\infty, R_{i_{0}}}	\leq \left\lvert \X u \right\rvert \leq 2C_{RH}\left\lVert \X u \right\rVert_{\infty, R_{i_{0}+1}} \quad \text{ in } B_{R_{i_{0}+1}}. 
	\end{align}
	Now note that \eqref{lower bound for sigma average} immediately implies 
	\begin{align}\label{lower upper bound sigma average}
		\frac{1}{8B}\left\lVert \X u \right\rVert_{\infty, R_{i_{0}+1}}	\leq \frac{1}{8B}\left\lVert \X u \right\rVert_{\infty, R_{i_{0}}}	\leq \left\lvert \left( \X u\right)_{\frac{1}{2}\sigma R_{i_{0}+1}}\right\rvert \leq 2C_{RH}\left\lVert \X u \right\rVert_{\infty, R_{i_{0}+1}}. 
	\end{align} Now, \eqref{upper and lower bound for Xu}, \eqref{lower upper bound sigma average} and \eqref{oscillation control for using main lemma} implies that the hypotheses of Lemma \ref{decay estimate in the nondegenerate regime} are satisfied with $R$ replaced by $R_{i_{0}+1}/2,$ $A=2C_{RH}$, $B=2c_1$ and $\lambda = \left\lVert \X u \right\rVert_{\infty, R_{i_{0}+1}},$ which we can always assume to be positive, as otherwise there is nothing to prove. Hence, Lemma \ref{decay estimate in the nondegenerate regime} implies that for any $0 < \rho \le\sigma R_{i_{a}}/2\le \sigma R_{i_{0}+1}/2,$ we have 
	\begin{align}\label{holder rev holder depending on q}
		\fint_{B_{\rho}} \left\lvert \X u - \left( \X u\right)_{\rho}\right\rvert^{2} \leq C_0 \sigma^{-Q-2\alpha_{0}} \left(\frac{\rho}{R_{i_{0}+1}}\right)^{2\alpha_{0}} \fint_{B_{\frac{1}{2}R_{i_{0}+1}}} \left\lvert \X u - \left( \X u\right)_{\frac{1}{2}R_{i_{0}+1}}\right\rvert^{2}.
			\end{align}
	From here, to derive our desired estimate, we use H\"{o}lder inequality on one side and reverse H\"{o}lder inequality on the other. More precisely, by \eqref{holder rev holder depending on q}, for any $0 < \rho \le\sigma R_{i_{a}}/2\le \sigma R_{i_{0}+1}/2,$ we have  
	\begin{align*}
		&\left( 	\fint_{B_{\rho}} \left\lvert \X u - \left( \X u\right)_{\rho}\right\rvert^{q}  \right)^{\frac{1}{q}}\\
		&\stackrel{\text{H\"{o}lder}}{\leq}	\left( 	\fint_{B_{\rho}} \left\lvert \X u - \left( \X u\right)_{\rho}\right\rvert^{2}  \right)^{\frac{1}{2}} \\
		&\stackrel{\eqref{holder rev holder depending on q}}{\leq} C_0 \sigma^{-\frac{Q}{2}-\alpha_{0}} \left(\frac{\rho}{R_{i_{0}+1}}\right)^{\alpha_{0}} \left(\fint_{B_{\frac{1}{2}R_{i_{0}+1}}} \left\lvert \X u - \left( \X u\right)_{\frac{1}{2}R_{i_{0}+1}}\right\rvert^{2}\right)^\frac{1}{2}\notag\\
		&\stackrel{\text{H\"{o}lder}}{\leq} C_0\sigma^{-\frac{Q}{2}-\alpha_{0}} \left(\frac{\rho}{R_{i_{0}+1}}\right)^{\alpha_{0}} \left(\fint_{\frac{1}{2}B_{R_{i_0+1}}} \left\lvert \X u - \left( \X u\right)_{\frac{1}{2}R_{i_0+1}}\right\rvert^{\frac{2Q}{Q-2}}\right)^{\frac{Q-2}{2Q}} \\
		&\stackrel{\eqref{minimality of mean}}{\leq} 2 C_0\sigma^{-\frac{Q}{2}-\alpha_{0}} \left(\frac{\rho}{R_{i_{0}+1}}\right)^{\alpha_{0}} \left(\fint_{B_{\frac{1}{2}R_{i_0+1}}} \left\lvert \X u - \left( \X u\right)_{R_{i_0+1}}\right\rvert^{\frac{2Q}{Q-2}}\right)^{\frac{Q-2}{2Q}}\\
		&\stackrel{\eqref{lower bound in any smaller ball},\eqref{rev holder bound}}{\leq} 2C_{RH} C_0\sigma^{-\frac{Q}{2}-\alpha_{0}} \left(\frac{\rho}{R_{i_{0}+1}}\right)^{\alpha_{0}} \fint_{B_{R_{i_0+1}}} \left\lvert \X u - \left( \X u\right)_{R_{i_0+1}}\right\rvert \\
		&\stackrel{\text{H\"{o}lder}}{\leq} 2C_{RH} C_0\sigma^{-\frac{Q}{2}-\alpha_{0}} \left(\frac{\rho}{R_{i_{0}+1}}\right)^{\alpha_{0}} \left(\fint_{B_{R_{i_0+1}}} \left\lvert \X u - \left( \X u\right)_{R_{i_0+1}}\right\rvert^{q}\right)^{\frac{1}{q}}\\
		&\stackrel{\eqref{minimality of mean}}{\leq} 4C_{RH}C_0\sigma^{-\frac{Q}{2}-\alpha_{0}} \left(\frac{\rho}{R_{i_{0}+1}}\right)^{\alpha_{0}} \left(\fint_{B_{R_{i_0+1}}} \left\lvert \X u - \left( \X u\right)_{R}\right\rvert^{q}\right)^{\frac{1}{q}} \\
		& \leq 4C_{RH}C_0\left(\frac{R}{R_{i_0+1}}\right)^\frac{Q}{q}\sigma^{-\frac{Q}{2}-\alpha_{0}} \left(\frac{\rho}{R_{i_{0}+1}}\right)^{\alpha_{0}} \left(\fint_{B_R}\left\lvert \X u - \left( \X u\right)_{R}\right\rvert^{q}\right)^{\frac{1}{q}}\\
		& \leq 4C_{RH}C_0\left(\frac{R}{R_{i_0+1}}\right)^{\frac{Q}{q}+\alpha_0}\sigma^{-\frac{Q}{2}-\alpha_{0}} \left(\frac{\rho}{R}\right)^{\alpha_{0}} \left(\fint_{B_R}\left\lvert \X u - \left( \X u\right)_{R}\right\rvert^{q}\right)^{\frac{1}{q}}\\
		& \leq 4C_{RH}C_0\left(\frac{2}{\tau_1^{i_0+1}}\right)^{\frac{Q}{q}+\alpha_0}\sigma^{-\frac{Q}{2}-\alpha_{0}} \left(\frac{\rho}{R}\right)^{\alpha_{0}} \left(\fint_{B_R}\left\lvert \X u - \left( \X u\right)_{R}\right\rvert^{q}\right)^{\frac{1}{q}}\\
		&\leq 4C_{RH}C_0\left(\frac{2}{\tau_1^{i_a+1}}\right)^{\frac{Q}{q}+\alpha_0}\sigma^{-\frac{Q}{2}-\alpha_{0}} \left(\frac{\rho}{R}\right)^{\alpha_{0}} \left(\fint_{B_R}\left\lvert \X u - \left( \X u\right)_{R}\right\rvert^{q}\right)^{\frac{1}{q}}.
	\end{align*}
Since $C_{RH}, C_0, \sigma, \alpha_{0}, \tau_{1}, i_{a}$ are all fixed, which ultimately depend only on $n$ and $p,$ we derive our desired estimate. This completes the proof.  
\end{proof}
Now we prove the corresponding result for $V(\X u).$  The proof is completely analogous to the proof of Theorem \ref{Main excess decay estimate Intro}. One only has to use Theorem \ref{thm:holder V} instead of Theorem \ref{thm:holder} and Lemma \ref{decay estimate in the nondegenerate regime for V} instead of Lemma \ref{decay estimate in the nondegenerate regime}.  For the convenience of readers we provide complete details.
\begin{proof}{\textbf{( of Theorem \ref{Main excess decay estimate V Intro}} ) }
	We can and henceforth assume $x_{0} =0.$ As remarked in the proof of Theorem \ref{Main excess decay estimate Intro}, it is enough to prove \eqref{Main excess decay estimate Intro estimate V}, for $0<\rho<\kappa R$, for some constant $\kappa\in (0,1)$ depending only $n$ and $p$, as the estimate is trivial otherwise. First, we fix $C_{RH, V}>0$ which is given by Lemma \ref{reverse holder inequality V lemma} with the choice $A=B=1$. So, clearly $C_{RH, V}$ depends only on $n$ and $p$. Next, let $B=B\left(n, p\right)\geq 1$ be defined by $B:= 2^{1+\frac{1}{p}}c_{p/2},$ where $c_{p/2}$ is the constant given by Theorem \ref{thm:lip Lq bound} with the choice $q=p/2$. Let $\theta_{0}, \alpha_{0} \in (0,1)$ be the constants given by Lemma  \ref{decay estimate in the nondegenerate regime for V} for the choices of $A=2C_{RH, V}$  and $B=2^{1+\frac{1}{p}}c_{p/2}.$ Let $\beta_{1} \in (0,1)$ be the exponent in Theorem \ref{thm:holder V}. Set 
	\begin{align*}
		\beta := \min \left\lbrace \alpha_{0}, \beta_{1} \right\rbrace. 
	\end{align*} 
	\begin{claim}\label{sup controls oscillation claim V}
		Without loss of generality, we can assume that the we have 
		\begin{align}\label{lower upper for average V}
			\frac{1}{B} \left\lVert \X u \right\rVert_{\infty, R/2} \leq \left\lvert \left (V\left (\X u\right)\right)_{R}\right\rvert^{\frac{2}{p}} \leq 16 \left\lVert \X u \right\rVert_{\infty, R/2}
		\end{align}
		and 
		\begin{align}\label{oscillation is smaller than sup V}
			\left( \fint_{B_{R/2}} \left\lvert V\left(\X u \right) - \left(V\left(\X u\right)\right)_{R} \right\rvert^{q} \right)^{\frac{1}{q}}\leq \theta_{0} \left\lVert \X u \right\rVert_{\infty, R/2}^\frac{p}{2}. 
		\end{align}
	\end{claim}
	\emph{Proof of Claim \ref{sup controls oscillation claim V}: } Indeed, if \eqref{oscillation is smaller than sup V} is violated, then by Theorem \ref{thm:holder V}, for any $0 < \rho < R/2,$ we have 
	\begin{align*}
		&\left( \fint_{B_{\rho}} \left\lvert V\left(\X u\right) - \left(V\left( \X u\right)\right)_{\rho}\right\rvert^{q}\, dx \right)^{\frac{1}{q}}\\ 		
		&\qquad \qquad \leq \sup\limits_{x, y \in B_{\rho}} \left\lvert V\left(\X u\right) \left(x\right) - V\left(\X u \right)\left(y\right) \right\rvert \\
		&\qquad \qquad \leq c_{h,V}\left(\frac{\rho}{R}\right)^{\beta_{1}}\left\lVert \X u \right\rVert_{\infty, R/2}^{\frac{p}{2}} \\
		& \qquad \qquad \leq \frac{c_{h,V}}{\theta_{0}}\left(\frac{\rho}{R}\right)^{\beta} \left( \fint_{B_{R/2}} \left\lvert V\left(\X u\right) - \left(V\left(\X u\right)\right)_{R} \right\rvert^{q} \right)^{\frac{1}{q}}\\
		&\qquad \qquad \leq \frac{2^{Q}c_{h,V}}{\theta_{0}}\left(\frac{\rho}{R}\right)^{\beta} \left( \fint_{B_{R}} \left\lvert V\left(\X u\right) - \left(V\left(\X u\right)\right)_{R} \right\rvert^{q} \right)^{\frac{1}{q}},
	\end{align*}
	which implies our desired estimate. So we can assume \eqref{oscillation is smaller than sup V} holds. Now note that 
	\begin{align*}
		\left\lvert \left(V\left(\X u\right)\right)_{R}\right\rvert^{2} &= \left\langle \left(V\left(\X u\right)\right)_{R}, \left(V\left(\X u\right)\right)_{R}\right\rangle \\&=\left\langle \left(V\left(\X u\right)\right)_{R}, \left(V\left(\X u\right)\right)_{R} -V\left(\X u\right) \left(x\right)\right\rangle  + \left\langle \left(V\left(\X u\right)\right)_{R}, V\left(\X u\right) \left(x\right)\right\rangle,
	\end{align*}
	for a.e. $x \in B_{R/2}.$ Using Cauchy-Schwarz, integrating and H\"older,  we deduce 
	\begin{align*}
		\left\lvert \left(V\left(\X u\right)\right)_{R}\right\rvert^{2} 
		&\leq \left\lvert \left(V\left(\X u\right)\right)_{R}\right\rvert  \left( \fint_{B_{R/2}} \left\lvert V\left(\X u\right) - \left(V\left(\X u\right)\right)_{R} \right\rvert + \fint_{B_{R/2}} \left\lvert V\left(\X u \right)\right\rvert \right) \\
		&\stackrel{\eqref{oscillation is smaller than sup V}}{\leq} \left(\theta_{0} +1 \right)\left\lvert \left(V\left(\X u\right)\right)_{R}\right\rvert \left\lVert \X u \right\rVert_{\infty, R/2}^\frac{p}{2}. 
	\end{align*} Since, $0<\theta_0<1$ so,  using Young's inequality in the last estimate we derive 
	\begin{align*}
		\left\lvert \left(V\left(\X u\right)\right)_{R}\right\rvert^{\frac{2}{p}} 
		&\leq  \left(\theta_{0} +1 \right)^\frac{1}{p}\left\lvert \left(V\left(\X u\right)\right)_{R}\right\rvert^\frac{1}{p} \left\lVert \X u \right\rVert_{\infty, R/2}^\frac{1}{2}\\
		& \leq 2^\frac{1}{p}\left\lvert \left(V\left(\X u\right)\right)_{R}\right\rvert^\frac{1}{p} \left\lVert \X u \right\rVert_{\infty, R/2}^\frac{1}{2} \\ 
		&\leq \frac{1}{2}\left\lvert \left(V\left(\X u\right)\right)_{R}\right\rvert^{\frac{2}{p}} +  2^{\frac{2}{p}+1}\left\lVert \X u \right\rVert_{\infty, R/2}.
	\end{align*}
Hence, we deduce 
	\begin{align*}
		\left\lvert \left(V\left(\X u\right)\right)_{R}\right\rvert^\frac{2}{p} \leq  2^{\frac{2}{p}+2} \left\lVert \X u \right\rVert_{\infty, R/2} \leq 16\left\lVert \X u \right\rVert_{\infty, R/2}. 
	\end{align*}
	Thus, the only way \eqref{lower upper for average V} can fail is if we have 
	\begin{align*}
		\left\lvert \left( V\left(\X u\right)\right)_{R}\right\rvert^\frac{2}{p}  < 	\frac{1}{B} \left\lVert \X u \right\rVert_{\infty, R/2}. 
	\end{align*}
	But then  by Theorem \ref{thm:lip Lq bound} (with `$q=p/2$'), H\"{o}lder's inequality and the Minkowski inequality, we find, 
	\begin{align*}
		\left\lVert \X u \right\rVert_{\infty, R/2} 
		&\leq c_{p/2} \left(\fint_{B_{R}} \left\lvert \X u \right\rvert ^{\frac{p}{2}}\right)^\frac{2}{p}  \\ 
		&\leq c_{p/2} \left(\fint_{B_{R}} \left\lvert V\left(\X u\right) \right\rvert^{q}\right)^{\frac{2}{pq}} \\
		&\leq c_{p/2} \left[\left(\fint_{B_{R}} \left\lvert V\left(\X u\right)  -  \left(V\left(\X u\right)\right)_{R}\right\rvert^{q}\right)^{\frac{1}{q}} + \left\lvert \left(V\left(\X u\right)\right)_{R}\right\rvert \right]^\frac{2}{p}\\ 
		& \leq 2^\frac{1}{p}c_{p/2} \left[\left(\fint_{B_{R}} \left\lvert V\left(\X u\right)  -  \left(V\left(\X u\right)\right)_{R}\right\rvert^{q}\right)^{\frac{2}{q}} + \left\lvert \left(V\left(\X u\right)\right)_{R}\right\rvert^2 \right]^\frac{1}{p}\\
		&\leq 2^\frac{1}{p}c_{p/2} \left(\fint_{B_{R}} \left\lvert V\left(\X u\right)  -  \left(V\left(\X u\right)\right)_{R}\right\rvert^{q}\right)^{\frac{2}{pq}} + 2^\frac{1}{p}c_{p/2}\left\lvert \left(V\left(\X u\right)\right)_{R}\right\rvert^{\frac{2}{p}} \\
		&< 2^\frac{1}{p}c_{p/2}  \left(\fint_{B_{R}} \left\lvert V\left(\X u\right)  -  \left(V\left(\X u\right)\right)_{R}\right\rvert^{q}\right)^{\frac{2}{pq}} + \frac{2^\frac{1}{p}c_{p/2}}{B} 	\left\lVert \X u \right\rVert_{\infty, R/2}. 
	\end{align*}
	Thus, using the choice of $B=2^{1+\frac{1}{p}}c_{p/2}$, we have 
	\begin{align*}
		\left\lVert \X u \right\rVert_{\infty, R/2} \leq B \left(\fint_{B_{R}} \left\lvert V\left(\X u\right)  -  \left(V\left(\X u\right)\right)_{R}\right\rvert^{q}\right)^{\frac{2}{pq}}. 
	\end{align*}
	Hence, as before, for any $0 < \rho < R/2,$ we have 
	\begin{align*}
		&\left( \fint_{B_{\rho}} \left\lvert V\left(\X u\right) - \left(V\left( \X u\right)\right)_{\rho}\right\rvert^{q}\, dx \right)^{\frac{1}{q}}\\
		&\qquad \leq c_{h, V}\left(\frac{\rho}{R}\right)^{\beta_{1}}\left\lVert \X u \right\rVert_{\infty, R/2}^{\frac{p}{2}}\leq c_{h,V}B^{\frac{p}{2}}\left(\frac{\rho}{R}\right)^{\beta_{1}}\left(\fint_{B_{R}} \left\lvert V\left(\X u\right)  -  \left(V\left(\X u\right)\right)_{R}\right\rvert^{q}\right)^{\frac{1}{q}},
	\end{align*}
	which again yields the desired estimate. This completes the proof of Claim \ref{sup controls oscillation claim V}.  
	
	From now on, we shall assume that \eqref{lower upper for average V} and \eqref{oscillation is smaller than sup V} holds. Now we  by Claim \ref{the alternatives claim} we assert the existence of a constant $\tau_{1} \in (0,1),$ depending only on $n$ and $p$ such that at least one of the following alternatives occur. 
	\begin{enumerate}[(i)]
		\item \textbf{The nondegenerate alternative: } We have 
		\begin{align}\label{the nondegenerate alt V}
			\left\lvert \X u \right\rvert \geq \frac{1}{4}\left\lVert \X u \right\rVert_{\infty, R/2} \qquad \text{ in  } B_{\tau_{1} R/2}. 
		\end{align}
		\item \textbf{The degenerate alternative: } We have 
		\begin{align}\label{the degenerate alt V}
			\left\lVert \X u \right\rVert_{\infty, \tau_{1} R/2} \leq \frac{1}{2}\left\lVert \X u \right\rVert_{\infty, R/2}. 
		\end{align}
	\end{enumerate}
Now we define 
	\begin{align*}
		R_{i}:= \tau_{1}^{i}\frac{R}{2} \qquad \text{ for all integers } i \geq 0. 
	\end{align*}
	Thus, by Claim \ref{the alternatives claim}, at each level $i$, at least one of the following alternatives must occur. 
	\begin{enumerate}[(a)]
		\item \textbf{The nondegenerate alternative: } We have 
		\begin{align}\label{the nondegenerate alt at level i V}
			\left\lvert \X u \right\rvert \geq \frac{1}{4}\left\lVert \X u \right\rVert_{\infty, R_{i}} \qquad \text{ in  } B_{R_{i+1}}. 
		\end{align}
		\item \textbf{The degenerate alternative: } We have 
		\begin{align}\label{the degenerate alt at level i V}
			\left\lVert \X u \right\rVert_{\infty, R_{i+1}} \leq \frac{1}{2}\left\lVert \X u \right\rVert_{\infty, R_{i}}. 
		\end{align}
	\end{enumerate}
Next, we choose $i_{a} \geq 1$ to be first integer for which we have 
	\begin{align}\label{choice of ia V}
		2^{i_{a}-1} \geq 2^{\frac{2}{p}}B. 
	\end{align}
	Clearly, $i_a$ depends only on $B$ and thus ultimately, only on $n$ and $p$.
	Now we claim the following. 
	\begin{claim}\label{nondegenerate case must occur claim V}
		Without loss of generality, we can assume that \eqref{the nondegenerate alt at level i V} occurs for the first time for some $0 \leq i < i_{a}.$ 
	\end{claim}
	\emph{Proof of Claim \ref{nondegenerate case must occur claim V}: }  If Claim \ref{nondegenerate case must occur claim V} is false, then \eqref{the degenerate alt at level i V} must have occured for all $ 0 \leq i \leq i_{a}-1.$ But this implies,  we have,  
	\begin{align*}
		\left\lVert \X u \right\rVert_{\infty, R_{i_{a}}} &\leq \left(\frac{1}{2}\right)^{i_{a}}\left\lVert \X u \right\rVert_{\infty, R/2}.  
	\end{align*}
	This now implies 
	\begin{align}\label{osc control energy in very small ball V}
		\left( \fint_{B_{R_{i_{a}}}} \left\lvert V\left(\X u\right) - \left(V\left( \X u \right)\right)_{R}\right\rvert^{q}  \right)^{\frac{1}{q}} \geq \left( \fint_{B_{R_{i_{a}}}} \left\lvert V\left(\X u\right) \right\rvert^{q} \right)^{\frac{1}{q}}.
	\end{align}
	Indeed, if \eqref{osc control energy in very small ball V} is false, then recalling \eqref{lower upper for average V}, we have 
	\begin{align*}
		\frac{1}{B^\frac{p}{2}} \left\lVert \X u \right\rVert_{\infty, R/2}^{\frac{p}{2}} 
		&\leq \left\lvert \left(V\left(\X u\right)\right)_{R}\right\rvert \\
		&=\fint_{B_{R_{i_{a}}}}\left\lvert \left(V\left(\X u\right)\right)_{R}\right\rvert\ \mathrm{d}x \\
		&\stackrel{\text{H\"{o}lder}}{\leq} \left( \fint_{B_{R_{i_{a}}}}\left\lvert V\left(\X u\right) - \left(V\left(\X u\right)\right)_{R}\right\rvert^{q}\ \mathrm{d}x \right)^{\frac{1}{q}} + \left( \fint_{B_{R_{i_{a}}}} \left\lvert V\left(\X u\right) \right\rvert^{q} \right)^{\frac{1}{q}} \\
		&< 2 \left( \fint_{B_{R_{i_{a}}}} \left\lvert V\left(\X u\right) \right\rvert^{q} \right)^{\frac{1}{q}}\leq 2\left\lVert \X u \right\rVert_{\infty, R_{i_{a}}}^{\frac{p}{2}} \leq 2\left(\frac{1}{2}\right)^{\frac{p}{2}i_{a}}\left\lVert \X u \right\rVert_{\infty, R/2}^\frac{p}{2}, 
	\end{align*}
	which is a contradiction, by \eqref{choice of ia V}. Hence \eqref{osc control energy in very small ball V} holds. Now using Theorem \ref{thm:lip Lq bound} (with `q=p/2') together with \eqref{osc control energy in very small ball V} and the fact $1\leq q\leq 2$, we find  
	\begin{align*}
		\left\lVert \X u \right\rVert_{\infty, R_{i_{a}}/2} 
		&\leq c_{p/2} \left(\fint_{B_{R_{i_{a}}}} \left\lvert \X u \right\rvert^{\frac{p}{2}} \right)^\frac{2}{p}\\ 
		&\leq c_{p/2} \left( \fint_{B_{R_{i_{a}}}} \left\lvert V\left(\X u \right)\right\rvert^{q}\right)^{\frac{2}{pq}} \\
		&\leq c_{p/2} \left( \fint_{B_{R_{i_{a}}}} \left\lvert V\left(\X u\right) - \left(V\left( \X u \right)\right)_{R}\right\rvert^{q} \right)^{\frac{2}{pq}} \\
		&\leq c_{p/2} \left( \frac{R}{R_{i_{a}}}\right)^{\frac{2Q}{pq}}\left( \fint_{B_{R}} \left\lvert V\left(\X u\right) - \left(V\left( \X u \right)\right)_{R}\right\rvert^{q} \right)^{\frac{2}{pq}} \\
		&\leq c_{p/2}2^{\frac{2Q}{p}}\tau_{1}^{\frac{-2i_{a}Q}{p}}\left( \fint_{B_{R}} \left\lvert V\left(\X u\right) - \left(V\left( \X u \right)\right)_{R}\right\rvert^{q} \right)^{\frac{2}{pq}}.  
	\end{align*}
	Combining this with Theorem \ref{thm:holder V}, for any $0 < \rho < R_{i_{a}}/2,$ we have 
	\begin{align*}
		&\left( \fint_{B_{\rho}} \left\lvert V\left(\X u\right) - \left(V\left( \X u\right)\right)_{\rho}\right\rvert^{q} \, dx \right)^{\frac{1}{q}}\\ 
		&\qquad \qquad \leq \sup\limits_{x, y \in B_{\rho}} \left\lvert V\left(\X u\right) \left(x\right) - V\left(\X u\right) \left(y\right) \right\rvert \\
		& \qquad \qquad \leq c\left(\frac{\rho}{R_{i_{a}}}\right)^{\beta_{1}} \left\lVert \X u \right\rVert_{\infty, R_{i_{a}}/2}^{\frac{p}{2}}  \\
		& \qquad \qquad \leq c\tau_1^{-\beta_{1}i_{a}}\left(\frac{\rho}{R}\right)^{\beta_{1}}\left\lVert \X u \right\rVert_{\infty, R_{i_{a}}/2}^{\frac{p}{2}} \\
		& \qquad \qquad \leq cc_{p/2}^{\frac{p}{2}}\tau_1^{-i_{a}\left(Q + \beta_{1}\right)}\left(\frac{\rho}{R}\right)^{\beta_{1}}\left( \fint_{B_{R}} \left\lvert V\left(\X u\right) - \left(V\left( \X u \right)\right)_{R}\right\rvert^{q} \right)^{\frac{1}{q}}. 
	\end{align*}
	This together with the fact that if $\frac{R_{i_a}}{2} \leq \rho \leq R$, then $\frac{\tau_1^{i_a}}{2}\leq \frac{\rho}{R} \leq 1$, yields our desired estimate. This proves Claim \ref{nondegenerate case must occur claim V}.

	Thus, we can assume that there exists a smallest integer $0 \leq i_{0} < i_{a}$ such that \eqref{the nondegenerate alt at level i V} occurs for  $i=i_{0}.$ Thus, we have 
	\begin{align}\label{lower bound in any smaller ball V}
		\frac{1}{4}\left\lVert \X u \right\rVert_{\infty, R_{i_{0}}}  \leq \left\lvert \X u \right\rvert \leq \left\lVert \X u \right\rVert_{\infty, R_{i_{0}}} \qquad \text{ in } B_{R_{i_{0}+1}}.
	\end{align}
	Now we claim 
	\begin{claim}\label{existence of sigma claim V}
		There exists a constant $\sigma \in (0,1),$ depending only on $n$, $p$ and $\theta_{0},$ and thus ultimately only on $n$ and $p$,  such that 
		we have 
		\begin{align}\label{oscillation by sup in sigma ball V}
			\left( \fint_{B_{\sigma R_{i_{0}+1}}} \left\lvert V\left(\X u\right) - \left(V\left( \X u\right)\right)_{\sigma R_{i_{0}+1}}\right\rvert^{q} \right)^{\frac{1}{q}} \leq \frac{\theta_{0}}{16} \left\lVert \X u \right\rVert_{\infty, R_{i_{0}}}^{\frac{p}{2}}.
		\end{align}
	\end{claim}
	\emph{Proof of Claim \ref{existence of sigma claim V}:} Let $c_{h, V}>0$ be the constant given by Theorem \ref{thm:holder V}. Choose $\sigma \in (0,1)$ sufficiently small such that 
	\begin{align*}
		c_{h, V}\sigma^{\beta_{1}} = \frac{\theta_{0}}{16}. 
	\end{align*}
	Clearly, $\sigma$ ultimately depends only on $n$ and $p$. Then, by Theorem \ref{thm:holder V}, we deduce 
	\begin{align*}
		&\left( \fint_{B_{\sigma R_{i_{0}+1}}} \left\lvert V\left(\X u\right) - \left(V\left( \X u\right)\right)_{\sigma R_{i_{0}+1}}\right\rvert^{q} \right)^{\frac{1}{q}}\\ 
		&\qquad \qquad\leq \sup\limits_{x, y \in B_{\sigma R_{i_{0}+1}}} \left\lvert V\left(\X u\right) \left( x\right) - V\left(\X u\right) \left( y\right)\right\rvert \\
		& \qquad \qquad \leq  	c_{h,V}\sigma^{\beta_{1}}\left\lVert \X u \right\rVert_{\infty, R_{i_{0}+1}}^{\frac{p}{2}} \leq \frac{\theta_{0}}{16}\left\lVert \X u \right\rVert_{\infty, R_{i_{0}+1}}^{\frac{p}{2}}  \leq \frac{\theta_{0}}{16}\left\lVert \X u \right\rVert_{\infty, R_{i_{0}}}^{\frac{p}{2}}.
	\end{align*}
	This proves Claim \ref{existence of sigma claim V}.  
	Now we are in a position to complete the proof of the Theorem. Note that, we trivially have for any $0<\kappa\leq 1$
	\begin{align}\label{bound for xi}
		\left \lvert V^{-1}\left[\left( V\left( \X u \right)\right)_{\kappa R_{i_0+1}}\right]\right\rvert= \left \lvert\left( V\left( \X u \right)\right)_{\kappa R_{i_0+1}}\right\rvert^{\frac{2}{p}} \leq \left\lVert \X u \right\rVert_{\infty, R_{i_0}}.
	\end{align}
	Now, observe that exactly one of the following can occur. Either we have  
	\begin{align}\label{lower bound for sigma average V}
		\left\lvert \left(V\left( \X u\right)\right)_{\frac{1}{2}\sigma R_{i_{0}+1}}\right\rvert ^{\frac{2}{p}} \geq 	\frac{1}{8B}\left\lVert \X u \right\rVert_{\infty, R_{i_{0}}},
	\end{align}
	or we must have 
	\begin{align}\label{no lower bound for sigma average V}
		\left\lvert \left(V\left( \X u\right)\right)_{\frac{1}{2}\sigma R_{i_{0}+1}}\right\rvert^{\frac{2}{p}}  < 	\frac{1}{8B}\left\lVert \X u \right\rVert_{\infty, R_{i_{0}}}.
	\end{align}
		If \eqref{no lower bound for sigma average V} holds, then by Theorem \ref{thm:lip Lq bound}, we have 
		\begin{align*}
			\left\lVert \X u \right\rVert_{\infty, \frac{1}{4}\sigma R_{i_{0}+1}} \leq c_{p/2}\left(\fint_{B_{\frac{1}{2}\sigma R_{i_{0}+1}}} \left\lvert \X u \right\rvert^{\frac{p}{2}} \right)^{\frac{2}{p}} \leq c_{p/2}\left( \fint_{B_{\frac{1}{2}\sigma R_{i_{0}+1}}} \left\lvert V\left(\X u\right) \right\rvert^{q}\right)^{\frac{2}{pq}}. 
		\end{align*}
Now, denoting $	\varUpsilon := \left(V\left( \X u\right)\right)_{\frac{1}{2}\sigma R_{i_{0}+1}} $ as a temporary shorthand and using the bound \begin{align}\label{lower sup controls upper sup V}
	\frac{1}{4}\left\lVert \X u \right\rVert_{\infty, R_{i_{0}}} \leq \left\lVert \X u \right\rVert_{\infty, \frac{1}{4}\sigma R_{i_{0}+1}},
\end{align}
which is implied by \eqref{lower bound in any smaller ball V},  in the last line of the following,  we deduce  
	\begin{align}
		\left\lVert \X u \right\rVert_{\infty, \frac{1}{4}\sigma R_{i_{0}+1}} &\leq c_{p/2}\left[\left( \fint_{B_{\frac{1}{2}\sigma R_{i_{0}+1}}} \left\lvert V\left(\X u\right)  - \varUpsilon\right\rvert^{q}\right)^{\frac{1}{q}} + \left\lvert \varUpsilon \right\rvert\right]^{\frac{2}{p}}  \notag\\
		&\stackrel{\eqref{no lower bound for sigma average V}}{\leq}2^{\frac{1}{p}}c_{p/2}\left( \fint_{B_{\frac{1}{2}\sigma R_{i_{0}+1}}} \left\lvert V\left(\X u\right)  - \varUpsilon\right\rvert^{q}\right)^{\frac{2}{pq}} + \frac{2^{\frac{1}{p}}c_{p/2}}{8B}\left\lVert \X u \right\rVert_{\infty, R_{i_{0}}} \notag\\
		&\leq 2^\frac{1}{p}c_{p}\left( \fint\limits_{B_{\frac{1}{2}\sigma R_{i_{0}+1}}} \left\lvert V\left(\X u\right)  - \varUpsilon\right\rvert^{q}\right)^{\frac{2}{pq}}+ \frac{2^\frac{1}{p}c_{p/2}}{2B}\left\lVert \X u \right\rVert_{\infty, \frac{1}{4}\sigma R_{i_{0}+1}} \label{prelim estimate for sup in 1/2sigma V}. 
	\end{align}
Now by our choice of $B=2^{1+\frac{1}{p}}c_{p/2}$ together with \eqref{prelim estimate for sup in 1/2sigma V} implies 
	\begin{align*}
		\left\lVert \X u \right\rVert_{\infty, \frac{1}{4}\sigma R_{i_{0}+1}/2} \leq \frac{2^{2+\frac{1}{p}}c_{p/2}}{3} \left( \fint_{B_{\frac{1}{2}\sigma R_{i_{0}+1}}} \left\lvert V\left(\X u\right)  - \left(V\left( \X u\right)\right)_{\frac{1}{2}\sigma R_{i_{0}+1}}\right\rvert^{q}\right)^{\frac{2}{pq}}.  
	\end{align*}
	Once again, combined with Theorem \ref{thm:holder V}, this implies, for any $0 < \rho \leq \sigma R_{i_a}/4\le \sigma R_{i_{0}+1}/4,$ we have  
	\begin{align}
		&\left( \fint_{B_{\rho}} \left\lvert V\left(\X u\right) - \left(V\left( \X u\right)\right)_{\rho}\right\rvert^{q}\, dx\right)^{\frac{1}{q}} \notag\\
		&\qquad\qquad\quad\leq c\left(\frac{\rho}{\sigma R_{i_{0}+1}}\right)^{\beta_{1}}\left\lVert \X u \right\rVert_{\infty, \frac{1}{4}\sigma R_{i_{0}+1}}^{\frac{p}{2}} \notag\\
		&\qquad\qquad\quad\leq \frac{2^{p+2}c_{p/2}^{\frac{p}{2}}c}{3^{\frac{p}{2}}}\left(\frac{\rho}{\sigma R_{i_{0}+1}}\right)^{\beta_{1}} \left( \fint_{B_{\frac{1}{2}\sigma R_{i_{0}+1}}} \left\lvert V\left(\X u\right)  - \left(V\left( \X u\right)\right)_{\frac{1}{2}\sigma R_{i_{0}+1}}\right\rvert^{q} \right)^{\frac{1}{q}} \notag
		\\&\qquad\qquad\quad\stackrel{\eqref{minimality of mean}}{\leq} \frac{2^{p+3}c_{p/2}^{\frac{p}{2}}c}{3^{\frac{p}{2}}}\left(\frac{\rho}{\sigma R_{i_{0}+1}}\right)^{\beta_{1}} \left( \fint_{B_{\frac{1}{2}\sigma R_{i_{0}+1}}} \left\lvert V\left(\X u\right)  - \left(V\left( \X u\right)\right)_{R}\right\rvert^{q} \right)^{\frac{1}{q}} \notag\\		
		&\qquad\qquad\quad
			 \leq c\left(\frac{R}{\sigma R_{i_{0}+1}}\right)^{\frac{Q}{q} +\beta_{1}} \left(\frac{\rho}{R}\right)^{\beta_{1}} \left( \fint_{B_{R}} \left\lvert V\left(\X u\right)  - \left(V\left( \X u\right)\right)_{R}\right\rvert^{q}\right)^{\frac{1}{q}}. 
		\label{prefinal estimate for with no lower bound on sigma average V}\raisetag{-23pt}
	\end{align}
	Now, since $\sigma,$ $\tau_{1}, \beta_{1}$ are fixed constants and $i_{a}$ is a fixed integer, all depending only on $n$ and $p,$ with $\sigma,\tau_{1} \in (0,1)$ and $i_{0} < i_{a},$ we note that 
	\begin{align*}
		\left(\frac{R}{\sigma R_{i_{0}+1}}\right)^{\frac{Q}{q} +\alpha_{1}} = 	\left(\frac{1}{\sigma \tau_{1}^{i_{0}+1}}\right)^{\frac{Q}{q} +\alpha_{1}} \leq \left(\frac{1}{\sigma \tau_{1}^{i_{a}}}\right)^{\frac{Q}{q} +\alpha_{1}} \leq \left(\frac{1}{\sigma \tau_{1}^{i_{a}}}\right)^{Q +\alpha_{1}} ,
	\end{align*}
	where the last term on the right is again a fixed constant, depending only on $n$ and $p.$ Hence, \eqref{prefinal estimate for with no lower bound on sigma average V} now easily yields our desired estimate. Hence it only remains to settle the case when \eqref{lower bound for sigma average V} holds. Observe that \eqref{oscillation by sup in sigma ball V} implies 
	\begin{align*}
		\left( \fint_{B_{\sigma R_{i_{0}+1}}} \left\lvert V\left(\X u\right) - \left(V\left( \X u\right)\right)_{\sigma R_{i_{0}+1}}\right\rvert^{q} \right)^{\frac{1}{q}} \leq \frac{\theta_{0}}{16} \left\lVert \X u \right\rVert^{\frac{p}{2}}_{\infty, R_{i_{0}}} \leq \theta_{0} \left\lVert \X u \right\rVert^{\frac{p}{2}}_{\infty, R_{i_{0}+1}},
	\end{align*}
where once again we have used the bound \eqref{lower sup controls upper sup V}. Combining this with the H\"{o}lder and reverse H\"{o}lder inequality, we deduce  
		\begin{align*}
		&\left( \fint_{B_{\frac{1}{2}\sigma R_{i_{0}+1}}} \left\lvert V\left(\X u\right) - \left(V\left( \X u\right)\right)_{\frac{1}{2}\sigma R_{i_{0}+1}}\right\rvert^{2} \right)^{\frac{1}{2}} \\
		&\quad  \stackrel{\eqref{minimality of mean}}{\leq} 2\left( \fint_{B_{\frac{1}{2}\sigma R_{i_{0}+1}}} \left\lvert V\left(\X u\right) - \left(V\left( \X u\right)\right)_{\sigma R_{i_{0}+1}}\right\rvert^{2} \right)^{\frac{1}{2}}\notag 
		\\&\quad \stackrel{\text{H\"{o}lder}}{\leq} 2\left( \fint_{B_{\frac{1}{2}\sigma R_{i_{0}+1}}} \left\lvert V\left(\X u\right) - \left(V\left( \X u\right)\right)_{\sigma R_{i_{0}+1}}\right\rvert^{\frac{2Q}{Q-2}} \right)^{\frac{Q-2}{2Q}}\\
		&\quad \stackrel{\eqref{lower bound in any smaller ball V}, \eqref{bound for xi}, \eqref{rev holder bound V 1 to 2 star}}{\leq} 2C_{RH, V}\fint_{B_{\sigma R_{i_{0}+1}}} \left\lvert V\left(\X u\right) - \left(V\left( \X u\right)\right)_{\sigma R_{i_{0}+1}}\right\rvert \\
		&\quad \leq 2C_{RH, V}\left( \fint_{B_{\sigma R_{i_{0}+1}}} \left\lvert V\left(\X u\right) - \left(V\left( \X u\right)\right)_{\sigma R_{i_{0}+1}}\right\rvert^{q} \right)^{\frac{1}{q}} \leq 2C_{RH, V}\theta_{0} \left\lVert \X u \right\rVert^{\frac{p}{2}}_{\infty, R_{i_{0}+1}}.
	\end{align*}
		Thus, recalling $\theta_{0} \in (0,1),$ we have 
	\begin{align}\label{oscillation control for using main lemma V}
		\fint_{B_{\frac{1}{2}\sigma R_{i_{0}+1}}} \left\lvert V\left(\X u\right) - \left(V\left( \X u\right)\right)_{\frac{1}{2}\sigma R_{i_{0}+1}}\right\rvert^{2} &\leq \theta_{0}^{2} (2C_{RH, V})^2\left\lVert \X u \right\rVert_{\infty,  R_{i_{0}+1}}^{p} \notag\\
		&\leq \theta_{0} (2C_{RH,V})^p\left\lVert \X u \right\rVert_{\infty,  R_{i_{0}+1}}^{p}. 
	\end{align}
	By virtue of  \eqref{lower bound in any smaller ball V}, we also trivially have 
	\begin{align}\label{upper and lower bound for Xu V}
		\frac{1}{4B}\left\lVert \X u \right\rVert_{\infty, R_{i_{0}+1}}	\leq \frac{1}{4}\left\lVert \X u \right\rVert_{\infty, R_{i_{0}}}	\leq \left\lvert \X u \right\rvert &=  \left\lvert V\left( \X u \right) \right\rvert^{\frac{2}{p}} \notag \\ 
		&\leq 2C_{RH, V}\left\lVert \X u \right\rVert_{\infty, R_{i_{0}+1}}   \text{ in } B_{R_{i_{0}+1}}.						     
	\end{align}
	Now note that \eqref{lower bound for sigma average V} immediately implies 
	\begin{align}\label{lower upper bound sigma average V}
		\frac{1}{8B}\left\lVert \X u \right\rVert_{\infty, R_{i_{0}+1}}	\leq \frac{1}{8B}\left\lVert \X u \right\rVert_{\infty, R_{i_{0}}}	&\leq \left\lvert \left(V\left( \X u\right)\right)_{\frac{1}{2}\sigma R_{i_{0}+1}}\right\rvert^{\frac{2}{p}} \notag \\ &\leq 2C_{RH, V}\left\lVert \X u \right\rVert_{\infty, R_{i_{0}+1}}. 
	\end{align} 
	Now, \eqref{upper and lower bound for Xu V}, \eqref{lower upper bound sigma average V} and \eqref{oscillation control for using main lemma V} implies that the hypotheses of Lemma \ref{decay estimate in the nondegenerate regime for V} are satisfied with $A=2 C_{RH, V}$ and $B= 2^{1+\frac{1}{p}}c_{p/2}$, $R$ replaced by $R_{i_{0}+1}/2,$ and $\lambda = \left\lVert \X u \right\rVert_{\infty, R_{i_{0}+1}},$ which we can always assume to be positive, as otherwise there is nothing to prove. Hence, Lemma \ref{decay estimate in the nondegenerate regime for V} implies that for any $0 < \rho \leq \sigma R_{i_{a}}/2 \leq \sigma R_{i_{0}+1}/2,$ we have 
	\begin{align}\label{holder rev holder depending on q V}
		\fint_{B_{\rho}} &\left\lvert V\left(\X u\right) - \left(V\left( \X u\right)\right)_{\rho}\right\rvert^{2} \notag\\&\leq C \sigma^{-Q-2\alpha_{0}} \left(\frac{\rho}{R_{i_{0}+1}}\right)^{2\alpha_{0}} \fint_{B_{\frac{1}{2}R_{i_{0}+1}}} \left\lvert V\left(\X u\right) - \left(V\left( \X u\right)\right)_{\frac{1}{2}R_{i_{0}+1}}\right\rvert^{2}.
	\end{align}
	From here, to derive our desired estimate, we use H\"{o}lder inequality on one side and reverse H\"{o}lder inequality on the other. More precisely, by \eqref{holder rev holder depending on q V}, for any $0 < \rho \leq \sigma R_{i_{a}}/2 \leq \sigma R_{i_{0}+1}/2,$ we have  
	\begin{align*}
		&\left( 	\fint_{B_{\rho}} \left\lvert V\left(\X u\right) - \left(V\left( \X u\right)\right)_{\rho}\right\rvert^{q}  \right)^{\frac{1}{q}}\\
		&\stackrel{\text{H\"{o}lder}}{\leq}	\left( 	\fint_{B_{\rho}} \left\lvert V\left(\X u\right) - \left(V\left( \X u\right)\right)_{\rho}\right\rvert^{2}  \right)^{\frac{1}{2}} \\
		&\stackrel{\eqref{holder rev holder depending on q V}}{\leq} C \sigma^{-\frac{Q}{2}-\alpha_{0}} \left(\frac{\rho}{R_{i_{0}+1}}\right)^{\alpha_{0}} \left(\fint_{B_{\frac{1}{2}R_{i_{0}+1}}} \left\lvert V\left(\X u\right) - \left(V\left( \X u\right)\right)_{\frac{1}{2}R_{i_{0}+1}}\right\rvert^{2}\right)^\frac{1}{2}\notag\\
		&\stackrel{\text{H\"{o}lder}}{\leq} C\sigma^{-\frac{Q}{2}-\alpha_{0}} \left(\frac{\rho}{R_{i_{0}+1}}\right)^{\alpha_{0}} \left(\fint_{\frac{1}{2}B_{R_{i_0+1}}} \left\lvert V\left(\X u\right) - \left(V\left( \X u\right)\right)_{\frac{1}{2}R_{i_0+1}}\right\rvert^{\frac{2Q}{Q-2}}\right)^{\frac{Q-2}{2Q}} \\
		&\stackrel{\eqref{minimality of mean}}{\leq} 2 C\sigma^{-\frac{Q}{2}-\alpha_{0}} \left(\frac{\rho}{R_{i_{0}+1}}\right)^{\alpha_{0}} \left(\fint_{B_{\frac{1}{2}R_{i_0+1}}} \left\lvert V\left(\X u\right) - \left(V\left( \X u\right)\right)_{R_{i_0+1}}\right\rvert^{\frac{2Q}{Q-2}}\right)^{\frac{Q-2}{2Q}}.
	\end{align*}
	Now by using \eqref{lower bound in any smaller ball V}, \eqref{bound for xi} and \eqref{rev holder bound V 1 to 2 star}, we continue the estimation
	
	\begin{align*}
		&\left( 	\fint_{B_{\rho}} \left\lvert V\left(\X u\right) - \left(V\left( \X u\right)\right)_{\rho}\right\rvert^{q}  \right)^{\frac{1}{q}}\\
		&\leq 2C_{RH, V} C\sigma^{-\frac{Q}{2}-\alpha_{0}} \left(\frac{\rho}{R_{i_{0}+1}}\right)^{\alpha_{0}} \fint_{B_{R_{i_0+1}}} \left\lvert V\left(\X u\right) - \left(V\left( \X u\right)\right)_{R_{i_0+1}}\right\rvert \\
		&\stackrel{\text{H\"{o}lder}}{\leq} 2C_{RH,V} C\sigma^{-\frac{Q}{2}-\alpha_{0}} \left(\frac{\rho}{R_{i_{0}+1}}\right)^{\alpha_{0}} \left(\fint_{B_{R_{i_0+1}}} \left\lvert V\left(\X u\right) - \left(V\left( \X u\right)\right)_{R_{i_0+1}}\right\rvert^{q}\right)^{\frac{1}{q}} \\
		&\leq 4C_{RH, V}C\sigma^{-\frac{Q}{2}-\alpha_{0}} \left(\frac{\rho}{R_{i_{0}+1}}\right)^{\alpha_{0}} \left(\fint_{B_{R_{i_0+1}}} \left\lvert V\left(\X u\right) - \left(V\left( \X u\right)\right)_{R}\right\rvert^{q}\right)^{\frac{1}{q}}\\
		& \leq 4C_{RH, V}C\left(\frac{R}{R_{i_0+1}}\right)^\frac{Q}{q}\sigma^{-\frac{Q}{2}-\alpha_{0}} \left(\frac{\rho}{R_{i_{0}+1}}\right)^{\alpha_{0}} \left(\fint_{B_R}\left\lvert V\left( \X u \right) - \left(V\left( \X u\right)\right)_{R}\right\rvert^{q}\right)^{\frac{1}{q}}\\
		& \leq 4C_{RH,V}C\left(\frac{R}{R_{i_0+1}}\right)^{\frac{Q}{q}+\alpha_0}\sigma^{-\frac{Q}{2}-\alpha_{0}} \left(\frac{\rho}{R}\right)^{\alpha_{0}} \left(\fint_{B_R}\left\lvert V\left(\X u\right) - \left(V\left( \X u\right)\right)_{R}\right\rvert^{q}\right)^{\frac{1}{q}}\\
		& \leq 4C_{RH, V}C\left(\frac{2}{\tau_1^{i_0+1}}\right)^{\frac{Q}{q}+\alpha_0}\sigma^{-\frac{Q}{2}-\alpha_{0}} \left(\frac{\rho}{R}\right)^{\alpha_{0}} \left(\fint_{B_R}\left\lvert V\left(\X u\right) - \left(V\left( \X u\right)\right)_{R}\right\rvert^{q}\right)^{\frac{1}{q}}\\
		&\leq 4C_{RH, V}C\left(\frac{2}{\tau_1^{i_a+1}}\right)^{\frac{Q}{q}+\alpha_0}\sigma^{-\frac{Q}{2}-\alpha_{0}} \left(\frac{\rho}{R}\right)^{\alpha_{0}} \left(\fint_{B_R}\left\lvert V\left(\X u\right) - \left(V\left( \X u\right)\right)_{R}\right\rvert^{q}\right)^{\frac{1}{q}}.
	\end{align*}
	Since $C_{RH, V}, \sigma, \alpha_{0}, \tau_{1}, i_{a}$ are all fixed, which ultimately depend only on $n$ and $p,$ we derive our desired estimate. This completes the proof.
\end{proof}
\section{H\texorpdfstring{\"{o}}{o}lder continuity}\label{Holder continuity section}
\subsection{Preliminary estimates}
Let $x_{0} \in \Omega$ and $0 < R <1 $ be such that $B(x_{0}, 2R) \subset \Omega.$ Clearly, if $u \in HW_{\text{loc}}^{1,p} \left(\Omega \right)$ is a weak solution to \eqref{p subLaplace divergenceform}, then $u \in HW^{1,p}\left(B_{R}\right)$ is a weak solution to 
\begin{align}\label{p subLaplace divergenceform Ball}
	\operatorname{div}_{\mathbb{H}} \left( a(x) \lvert \X u \rvert^{p-2} \X u \right)  &= \operatorname{div}_{\mathbb{H}}F   &&\text{ in } B_{R}.
\end{align}
Now we define $w \in u + HW_{0}^{1,p}\left( B_{R} \right)$ to be the unique solution of 
\begin{align}\label{p subLaplace homogeneous general campanato}
	\left\lbrace \begin{aligned}
		\operatorname{div}_{\mathbb{H}} \left( a(x) \lvert \X w \rvert^{p-2} \X w \right)  &= 0   &&\text{ in } B_{R},\\
		w &=u &&\text{  on } \partial B_{R},
	\end{aligned} 
	\right. 
\end{align}
and $v \in w + HW_{0}^{1,p}\left( B_{R} \right)$ to be the unique solution of 
\begin{align}\label{p subLaplace homogeneous frozen general campanato}
	\left\lbrace \begin{aligned}
		\operatorname{div}_{\mathbb{H}} \left( a(x_{0}) \lvert \X v \rvert^{p-2} \X v \right) &= 0   &&\text{ in }  B_{R},\\
		v &= w &&\text{  on } \partial B_{R}.
	\end{aligned} 
	\right. 
\end{align}
Now we record an easy comparison estimate, whose proof follows from the weak formulation of \eqref{p subLaplace divergenceform Ball} and \eqref{p subLaplace homogeneous general campanato}, the estimates \eqref{constant cv}, \eqref{monotonicity} and Young's inequality.  
\begin{lemma}\label{firstcomparisoncampanato}
	Let $1<p<\infty$, $u$ be as in \eqref{p subLaplace divergenceform Ball} and $w$ be as in \eqref{p subLaplace homogeneous general campanato}.  Then for any $\xi \in \mathbb{R}^{2n},$ we have the following inequality
	\begin{align}\label{firstcomparisonestimate campanato}
		\int_{B_{R}} \left\lvert V( \X u) - V( \X w) \right\rvert^{2}  \leq c \int_{B_{R}} \left\lvert F - \xi \right\rvert \left\lvert \Xu - \X w \right\rvert,
	\end{align}
where $c>0$ is a constant depending on $n,p,\nu$ and $L$.
\end{lemma}
The following is proved analogously to Lemma 3 in \cite{Sil_nonlinearStein}. 
\begin{lemma}\label{secondcomparison1campanato}
	Let $1<p<\infty$, $w$ be as in \eqref{p subLaplace homogeneous general campanato} and $v$ be as in \eqref{p subLaplace homogeneous frozen general campanato}.  Then there exists a constant $c \equiv c\left( n, p, \gamma, L \right)$ such that we have the inequality
	\begin{align}\label{l1omega campanato}
		\fint_{B_{R}} \lvert V(\X v) - V(\X w) \rvert^{2}  &\leq c\left[ \omega\left(R\right) \right]^{2}  
		\fint_{B} \lvert \X w \rvert^{p}.
	\end{align}
\end{lemma}
\subsection{Proofs of Theorem \ref{campanato estimates} and Corollary \ref{campanato estimates nondivergence}}
Now we are in a position to prove the H\"{o}lder continuity results. 
\begin{proof}{\textbf{( of Theorem \ref{campanato estimates}} ) } Since all our estimates are local, we can assume $a \in C^{0, \mu_{1}} \left(\Omega\right)$, $F \in C^{0, \mu_{2}}\left(\Omega; \mathbb{R}^{2n}\right)$ and $u \in HW^{1,p} \left(\Omega\right).$ We first focus on the case $1< p \le 2,$ the other case being much easier.  We divide the proof in two steps.  \smallskip 
	
	\noindent \emph{Step 1:}
	First, we are going to show $\X u \in \mathrm{L}_{\text{loc}}^{p,Q-\delta} \left(\Omega, \mathbb{R}^{2n}\right)$ for every $\delta >0$,  under the assumptions of the theorem.  We first choose $x_{0} \in \Omega$ and radii $\rho, R >0$ such that $0 < 4\rho < R$ and $B_{R}(x_{0}) \subset \subset \Omega. $ Now we define the functions $w$ and $v$ the same way as before in $B_{R}$ by using \eqref{p subLaplace homogeneous general campanato} and \eqref{p subLaplace homogeneous frozen general campanato} respectively.  Now using \eqref{Xu:bdd}, by standard calculations, for any $1<p<\infty$, we have 
	\begin{align}
		\int_{B_{\rho}}\left\lvert \X u \right\rvert^{p} \leq 	c\left( \frac{\rho}{R}\right)^{Q}\int_{B_{R}}\left\lvert \X u \right\rvert^{p}  + 	c\int_{B_{R}}\left\lvert \X u - \X v\right\rvert^{p}. \label{Morrey bound p less 2 case}
	\end{align}
	Now, since $1 < p \le 2,$ using \eqref{v estimate p less 2}, we deduce 
	\begin{multline*}
		\int_{B_{R}}\left\lvert \X w - \X v\right\rvert^{p} \leq  c\int_{B_{R}}\left\lvert V\left( \X w \right) - V \left(\X v \right) \right\rvert^{2} \\+c \int_{B_{R}} \left\lvert V\left( \X w \right) - V \left(\X v \right) \right\rvert^{p}\left\lvert \X w\right\rvert^{\frac{p(2-p)}{2}}.  
	\end{multline*} 
	Using Young's inequality with $\varepsilon>0,$ together with \eqref{l1omega campanato}, for $1<p< 2$,  we deduce 
	\begin{align*}
		\int_{B_{R}}\left\lvert \X w - \X v\right\rvert^{p} &\leq \varepsilon\int_{B_{R}} \left\lvert \X w\right\rvert^{p} + c \left( 1+ \varepsilon^{\frac{p-2}{p}} \right)\int_{B_{R}}\left\lvert V\left( \X w \right) - V \left(\X v \right) \right\rvert^{2} \\
		&\stackrel{\eqref{l1omega campanato}}{\leq} \varepsilon\int_{B_{R}} \left\lvert \X w\right\rvert^{p} + c \left( 1+ \varepsilon^{\frac{p-2}{p}} \right)\left[ \omega\left(R\right) \right]^{2}  
		\int_{B_{R}} \lvert \X w \rvert^{p}. 
	\end{align*}
The inequality above trivially holds for $p=2$. Note that $w$ minimizes the functional 
	\begin{align*}
		w \mapsto \frac{1}{p} \int_{{B_{R}}} a\left(x\right)\left\lvert \X w \right\rvert^{p} \quad \text{ in  } u + HW_{0}^{1,p}\left(B_{R}\right).
	\end{align*}
	Thus, by minimality and the bounds on $a$, we have 
	\begin{align}\label{w minimizes}
		\frac{\gamma}{p} \int_{{B_{R}}} \left\lvert \X w \right\rvert^{p} \leq \frac{1}{p} \int_{{B_{R}}} a\left(x\right)\left\lvert \X w \right\rvert^{p} \leq \frac{1}{p} \int_{{B_{R}}} a\left(x\right)\left\lvert \X u \right\rvert^{p} \leq \frac{L}{p} \int_{{B_{R}}} \left\lvert \X u \right\rvert^{p}.  
	\end{align} 
	Using this in the last estimate, we arrive at 
	\begin{align}\label{comparison of w and v in terms of u}
		\int_{B_{R}}\left\lvert \X w - \X v\right\rvert^{p} \leq \frac{cL}{\gamma} \left(\varepsilon + \varepsilon^{\frac{p-2}{p}}\left[ \omega\left(R\right) \right]^{2}\right)\int_{{B_{R}}} \left\lvert \X u \right\rvert^{p}.
	\end{align}
	For the other term, we have, again using \eqref{v estimate p less 2}
	\begin{align*}
		\int_{B_{R}}\left\lvert \X u - \X  w\right\rvert^{p} &\leq c \int_{B_{R}}\left\lvert V\left( \X u \right) - V \left(\X w \right) \right\rvert^{2} + \int_{B_{R}} \left\lvert V\left( \X u \right) - V \left(\X w \right) \right\rvert^{p}\left\lvert \X u\right\rvert^{\frac{p(2-p)}{2}} \\
		&\leq \varepsilon\int_{B_{R}} \left\lvert \X u\right\rvert^{p} + c \left( 1+ \varepsilon^{\frac{p-2}{p}} \right)\int_{B_{R}}\left\lvert V\left( \X u \right) - V \left(\X w \right) \right\rvert^{2} \\
		&\stackrel{\eqref{firstcomparisonestimate campanato}}{\leq} \varepsilon\int_{B_{R}} \left\lvert \X u\right\rvert^{p} + c\varepsilon^{\frac{p-2}{p}} \int_{B_{R}} \left\lvert F -\xi\right\rvert \left\lvert \X u - \X w\right\rvert \\
		&\leq \frac{1}{2}\int_{B_{R}}\left\lvert \X u - \X  w\right\rvert^{p} +  \varepsilon\int_{B_{R}} \left\lvert \X u\right\rvert^{p} + c\varepsilon^{\frac{p-2}{p-1}}\int_{B_{R}} \left\lvert F -\xi\right\rvert^{\frac{p}{p-1}}. 
	\end{align*}
	Thus, we deduce 
	\begin{align}\label{comparison of u and w in terms of u}
		\int_{B_{R}}\left\lvert \X u - \X  w\right\rvert^{p} \leq 2\varepsilon\int_{B_{R}} \left\lvert \X u\right\rvert^{p} + c\varepsilon^{\frac{p-2}{p-1}}\int_{B_{R}} \left\lvert F -\xi\right\rvert^{\frac{p}{p-1}}. 
	\end{align}
	Combining \eqref{comparison of u and w in terms of u} and \eqref{comparison of w and v in terms of u},  with the shorthand 
	\begin{align}\label{def of C_epsilon_R}
		C_{\varepsilon, R}:=2^{p-1}\left[ \left( 2+ \frac{cL}{\gamma}\right) \varepsilon + \frac{cL}{\gamma \varepsilon^{\frac{2-p}{p}}}\left[ \omega\left(R\right) \right]^{2}\right], 
	\end{align} for some $c\equiv c(n,p, \gamma, L)>0,$ coming from \eqref{comparison of u and w in terms of u} and \eqref{comparison of w and v in terms of u}, we have,
	\begin{align}\label{comaprison of u and v in terms of u}
		\int_{B_{R}}\left\lvert \X u - \X v\right\rvert^{p} &\leq 2^{p-1} \left( \int_{B_{R}}\left\lvert \X u - \X w \right\rvert^{p} + \int_{B_{R}}\left\lvert \X w - \X v\right\rvert^{p}\right) \notag \\
		&\leq C_{\varepsilon, R}\int_{B_{R}} \left\lvert \X u\right\rvert^{p}  + c\varepsilon^{\frac{p-2}{p-1}}\int_{B_{R}} \left\lvert F -\xi\right\rvert^{\frac{p}{p-1}}. 
		\end{align}
	Now, plugging this with the choice $\xi =0,$ in \eqref{Morrey bound p less 2 case}  we have 
	\begin{align*}
		\int_{B_{\rho}}\left\lvert \X u \right\rvert^{p} &\leq  c\left( \frac{\rho}{R}\right)^{Q}\int_{B_{R}}\left\lvert \X u \right\rvert^{p}    + 	C_{\varepsilon, R} \int_{B_{R}} \left\lvert \X u\right\rvert^{p}  + c\varepsilon^{\frac{p-2}{p-1}}\int_{B_{R}} \left\lvert F\right\rvert^{\frac{p}{p-1}} \\
		&\leq  c\left( \frac{\rho}{R}\right)^{Q}\int_{B_{R}}\left\lvert \X u \right\rvert^{p}    + 	C_{\varepsilon, R} \int_{B_{R}} \left\lvert \X u\right\rvert^{p}  + c\varepsilon^{\frac{p-2}{p-1}} R^{Q} \left\lVert F\right\rVert_{L^{\infty}\left(\Omega; \mathbb{R}^{2n}\right)}^{\frac{p}{p-1}}. 
	\end{align*}
	Now we use the standard iteration lemma ( see Lemma 5.13 in \cite{giaquinta-martinazzi-regularity} ). We fix $\delta >0$ and use Lemma 5.13 in \cite{giaquinta-martinazzi-regularity} with $\alpha = Q,$ $\beta = Q-\delta$ and 
	\begin{align*}
		\phi \left(r\right):= \int_{B_{r}}\left\lvert \X u \right\rvert^{p} . 
	\end{align*} 
	Now, in view of \eqref{def of C_epsilon_R},  we first choose $\varepsilon>0$ small enough and then choose $R$ small enough ( depending on $\varepsilon$ ) such that $C_{\varepsilon, R}$ above is less then the threshold $\varepsilon_{0},$ given by the interation Lemma. This yields, 
	\begin{align*}
		\frac{1}{\rho^{Q-\delta}}\int_{B_{\rho}}\left\lvert \X u \right\rvert^{p} &\leq c \left( \left\lVert \X u\right\rVert_{L^{p}\left(\Omega; \mathbb{R}^{2n}\right)}^{p}  +  \left\lVert F\right\rVert_{L^{\infty}\left(\Omega; \mathbb{R}^{2n}\right)}^{\frac{p}{p-1}} \right).
	\end{align*}
	Note that, here $c>0$ depends on $n,p,\nu,L$ and $\delta$.
	By the standard covering argument, this implies
	\begin{align*}
		\left\lVert \X u \right\rVert_{\mathrm{L}^{p, Q-\delta}\left(\Omega_{2}; \mathbb{R}^{2n}\right)} \leq c \left( \left\lVert \X u\right\rVert_{L^{p}\left(\Omega_{1}; \mathbb{R}^{2n}\right)}  +  \left\lVert F\right\rVert_{L^{\infty}\left(\Omega_{1}; \mathbb{R}^{2n}\right)}^{\frac{1}{p-1}} \right),
	\end{align*}
	for any $\Omega_{2} \subset \subset \Omega_{1} \subset \subset \Omega,$ where the constant $c$ now depends also on $\Omega_{2}$ and $\Omega_{1}.$ \smallskip 
	
	\noindent \emph{Step 2:} Now we finish the proof of H\"{o}lder continuity of $\X u.$ With the Morrey bound at our disposal, we return to the comparison estimates and estimate them differently this time. We have the following estimate for some constant $c\equiv c(n,p, \nu,L)>0$, 
	\begin{align*}
		\int_{B_{R}}&\left\lvert \X w - \X v\right\rvert^{p} \notag\\&\leq c \int_{B_{R}}\left\lvert V\left( \X w \right) - V \left(\X v \right) \right\rvert^{2} + c\int_{B_{R}} \left\lvert V\left( \X w \right) - V \left(\X v \right) \right\rvert^{p}\left\lvert \X w\right\rvert^{\frac{p(2-p)}{2}} \notag\\
		&\leq c \int_{B_{R}}\left\lvert V\left( \X w \right) - V \left(\X v \right) \right\rvert^{2}  + c\left(\int_{B_{R}}\left\lvert V\left( \X w \right) - V \left(\X v \right) \right\rvert^{2} \right)^{\frac{p}{2}}\left( \int_{B_{R}} \left\lvert \X w\right\rvert^{p}\right)^{\frac{2-p}{2}}
		\notag\\
		&\stackrel{\eqref{l1omega campanato}}{\leq} c \left[ \omega\left(R\right) \right]^{2}  
		\int_{B_{R}} \lvert \X w \rvert^{p}  + \left(\left[ \omega\left(R\right) \right]^{2}  
		\int_{B_{R}} \lvert \X w \rvert^{p} \right)^{\frac{p}{2}}\left( \int_{B_{R}} \left\lvert \X w\right\rvert^{p}\right)^{\frac{2-p}{2}}
		\notag\\ 
		&\leq c\left( \left[ \omega\left(R\right) \right]^{2} + \left[ \omega\left(R\right) \right]^{p}\right)\int_{B_{R}} \lvert \X w \rvert^{p}.   
		\end{align*}
	In view of \eqref{w minimizes}, this implies 
	\begin{align}\label{decay of w -v term1}
		\int_{B_{R}}\left\lvert \X w - \X v\right\rvert^{p} \leq c\left( \left[ \omega\left(R\right) \right]^{2} + \left[ \omega\left(R\right) \right]^{p}\right)\int_{B_{R}} \lvert \X u \rvert^{p}. 
	\end{align}
 As $a$ is $C^{0, \mu}$ and $p \le 2,$ this implies 
	\begin{align}\label{decay of w -v term}
		\int_{B_{R}}\left\lvert \X w - \X v\right\rvert^{p} \leq c(n, p, \nu, L, \delta)R^{p\mu + Q-\delta}	\left\lVert \X u \right\rVert_{\mathrm{L}^{p, Q-\delta}}^{p}.  
	\end{align}
	For the other comparison estimate, for some constant $c\equiv c(n,p,\nu,L)>0,$ we have 
	\begin{align*}
			&\int_{B_{R}}\left\lvert \X u - \X  w\right\rvert^{p} \\&\quad\leq c \int_{B_{R}}\left\lvert V\left( \X u \right) - V \left(\X w \right) \right\rvert^{2} + c\int_{B_{R}} \left\lvert V\left( \X u \right) - V \left(\X w \right) \right\rvert^{p}\left\lvert \X u\right\rvert^{\frac{p(2-p)}{2}} \\
		&\quad\leq c \int_{B_{R}}\left\lvert V\left( \X u \right) - V \left(\X w \right) \right\rvert^{2}  + c\left(\int_{B_{R}}\left\lvert V\left( \X u \right) - V \left(\X w \right) \right\rvert^{2} \right)^{\frac{p}{2}}\left( \int_{B_{R}} \left\lvert \X u\right\rvert^{p}\right)^{\frac{2-p}{2}}. 
	\end{align*}
	Now using \eqref{firstcomparisonestimate campanato} on the right, we deduce   
	\begin{align*}
		\int_{B_{R}}&\left\lvert \X u - \X  w\right\rvert^{p} 	\\	&\leq
			c\int_{B_{R}} \left\lvert F -\xi\right\rvert \left\lvert \X u - \X w\right\rvert   + c\left(\int_{B_{R}} \left\lvert F -\xi\right\rvert \left\lvert \X u - \X w\right\rvert\right)^{\frac{p}{2}}\left( \int_{B_{R}} \left\lvert \X u\right\rvert^{p}\right)^{\frac{2-p}{2}}	\\
		&\leq \begin{aligned}[t]
			c\int_{B_{R}} &\left\lvert F -\xi\right\rvert \left\lvert \X u - \X w\right\rvert \\ &+ 	c\left(\int_{B_{R}} \left\lvert F -\xi\right\rvert^{\frac{p}{p-1}}\right)^{\frac{p-1}{2}} \left( \int_{{B_{R}}}\left\lvert \X u - \X w\right\rvert^{p}\right)^{\frac{1}{2}}\left( \int_{B_{R}} \left\lvert \X u\right\rvert^{p}\right)^{\frac{2-p}{2}} 
		\end{aligned}\\ 
		&\leq \begin{aligned}[t]
			\frac{1}{2}\int_{B_{R}}\left\lvert \X u - \X  w\right\rvert^{p} &+ c \int_{B_{R}} \left\lvert F -\xi\right\rvert^{\frac{p}{p-1}} \\ &+ c \left(\int_{B_{R}} \left\lvert F -\xi\right\rvert^{\frac{p}{p-1}}\right)^{p-1} \left( \int_{B_{R}} \left\lvert \X u\right\rvert^{p}\right)^{2-p}. 
		\end{aligned}
	\end{align*}
	Thus, we arrive at 
	\begin{align}\label{decay of u -w term1}
		\int_{B_{R}}\left\lvert \X u - \X  w\right\rvert^{p} \leq c &\int_{B_{R}} \left\lvert F -\xi\right\rvert^{\frac{p}{p-1}} \notag \\ &\qquad + c \left(\int_{B_{R}} \left\lvert F -\xi\right\rvert^{\frac{p}{p-1}}\right)^{p-1} \left( \int_{B_{R}} \left\lvert \X u\right\rvert^{p}\right)^{2-p}
	\end{align}
	for some constant $c\equiv c(n, p, \nu, L)>0$.	 
	Now choosing $\xi = \left(F\right)_{B_{R}}$ and recalling the fact that $F \in C^{0, \mu}$ and the Campanato characterization of H\"{o}lder continuity, we have 
	\begin{align*}
		&\int_{B_{R}}\left\lvert \X u - \X  w\right\rvert^{p} \notag \\&\leq
			R^{Q+ \frac{p\mu}{p-1}}\left[ F\right]^{\frac{p}{p-1}}_{\mathcal{L}^{\frac{p}{p-1},Q+ \frac{p\mu}{p-1} }}  + c R ^{\left[ (p-1)Q + p\mu \right]}\left[ F\right]^{p}_{\mathcal{L}^{\frac{p}{p-1},Q+ \frac{p\mu}{p-1}}}\cdot R^{(Q-\delta)(2-p)}\left\lVert \X u \right\rVert_{\mathrm{L}^{p, Q-\delta}}^{p\left(2-p\right)}\notag \\
		&\leq R^{Q+ \frac{p\mu}{p-1}}\left[ F\right]^{\frac{p}{p-1}}_{\mathcal{L}^{\frac{p}{p-1},Q+ \frac{p\mu}{p-1} }} + R^{Q +p\mu - \delta (2-p)}\left[ F\right]^{p}_{\mathcal{L}^{\frac{p}{p-1},Q+ \frac{p\mu}{p-1}}}\left\lVert \X u \right\rVert_{\mathrm{L}^{p, Q-\delta}}^{p\left(2-p\right)}.  
	\end{align*}	
	Now if $1<p <2$, we apply Young's inequality on the last term on the right with the the exponents $(1/(p-1), 1/(2-p)),$ to arrive at 
	\begin{align}\label{decay of u -w term}
		\int_{B_{R}}\left\lvert \X u - \X  w\right\rvert^{p} \leq cR^{Q +p\mu - \delta (2-p)} \left( \left[ F\right]^{\frac{p}{p-1}}_{\mathcal{L}^{\frac{p}{p-1},Q+ \frac{p\mu}{p-1}}} + \left\lVert \X u \right\rVert_{\mathrm{L}^{p, Q-\delta}}^{p} \right), 
	\end{align}
	where we used the fact that since $p < 2,$ we have $Q+\frac{p\mu}{p-1} > Q +p\mu - \delta (2-p).$ Note that, \eqref{decay of u -w term} holds trivially if $p=2$.
	Now, since $1 < p \leq 2, $ using \eqref{Main excess decay estimate Intro estimate} of  Theorem \ref{Main excess decay estimate Intro} with $q=p$ we have, in the standard way,   
	\begin{align*}
		\int_{B_{\rho}} \left\lvert \X u - \left(\X u\right)_{\rho}\right\rvert^{p} &\leq c \left(\frac{\rho}{R}\right)^{Q+ p\alpha}	\int_{B_{R}} \left\lvert \X u - \left(\X u\right)_{R}\right\rvert^{p} + 	\int_{B_{R}} \left\lvert \X u - \X v \right\rvert^{p}.  
	\end{align*}
	Utilizing \eqref{decay of u -w term} and \eqref{decay of w -v term} in this, we deduce 
	\begin{align*}
		\int_{B_{\rho}}\left\lvert \X u - \left(\X u\right)_{\rho}\right\rvert^{p} &\leq c \left(\frac{\rho}{R}\right)^{Q+ p\alpha}	\int_{B_{R}} \left\lvert \X u - \left(\X u\right)_{R}\right\rvert^{p} + cR^{ Q+ p\mu-\delta}	\left\lVert \X u \right\rVert_{\mathrm{L}^{p, Q-\delta}}^{p} \\&\qquad \qquad + cR^{Q +p\mu - \delta (2-p)} \left( \left[ F\right]^{\frac{p}{p-1}}_{\mathcal{L}^{\frac{p}{p-1},Q+ \frac{p\mu}{p-1}}} + \left\lVert \X u \right\rVert_{\mathrm{L}^{p, Q-\delta}}^{p} \right). 
	\end{align*} 
	Note that for $1<p\le2,$ the smallest exponent of $R$ on the right is $Q+p\mu -\delta.$ Thus, choosing $0 < \delta < p \mu, $	 using Lemma 5.13 in \cite{giaquinta-martinazzi-regularity}  and standard covering arguments, together with Camapanto's characterization of H\"{o}lder continuity, the estimate above implies 
	$\X u \in C^{0, \mu -\frac{\delta}{p}}_{\text{loc}}$ for any $0 < \delta < p\mu. $ But then we also have $\X u \in L^{\infty}_{\text{loc}}$ and consequently, $\X u \in \mathrm{L}^{p, Q}_{\text{loc}}.$ Now we go back to \eqref{decay of w -v term1} and \eqref{decay of u -w term1} and derive \eqref{decay of w -v term} and \eqref{decay of u -w term} with $\delta =0$ with constants depending only on $n,p,\nu$ and $L$.  Consequently, we have proved $\X u \in C^{0, \mu}_{\text{loc}}.$ Tracing back our estimates, we have the estimate 
	\begin{align*}
		\left\lVert \X u \right\rVert_{C^{0, \mu}\left( \overline{\Omega_{2}}; \mathbb{R}^{2n}\right)} \leq C \left( \left\lVert \X u \right\rVert_{L^{p}\left( \Omega_{1}; \mathbb{R}^{2n}\right)} + \left\lVert F \right\rVert_{C^{0, \mu}\left( \overline{\Omega_{1}}; \mathbb{R}^{2n}\right)}\right),  
	\end{align*}
	for any $\Omega_{2} \subset \subset \Omega_{1} \subset \subset \Omega,$ where the constant $C$ now depends also on $\Omega_{2}$ and $\Omega_{1}.$ 	
	
	Now we prove the case $p > 2.$ This is quite easy. We use  Lemma \ref{firstcomparisoncampanato} to derive
	\begin{align*}
		\int_{{B_{R}}} \left\lvert V(\X u) - V(\X w) \right\rvert^{2}&\leq c\int_{{B_{R}}} \left\lvert F - \xi \right\rvert \left\lvert \X u - \X w \right\rvert \\
		&\leq \varepsilon \int_{{B_{R}}}\left\lvert \Xu - \X w \right\rvert^{p} + c \int_{{B_{R}}}\left\lvert  F - \xi  \right\rvert^{\frac{p}{p-1}}\\
		&\stackrel{\eqref{constant cv}}{\leq}c_V \varepsilon\int_{{B_{R}}} \left\lvert V(\X u) - V(\X w) \right\rvert^{2} + c \int_{{B_{R}}}\left\lvert  F - \xi  \right\rvert^{\frac{p}{p-1}} 
	\end{align*}
	Choosing $\varepsilon >0$ small enough, we have 
	\begin{align}\label{estimate vxu-vxw}
		\int_{{B_{R}}} \left\lvert V(\X u) - V(\X w) \right\rvert^{2} \leq c \int_{{B_{R}}}\left\lvert  F - \xi  \right\rvert^{\frac{p}{p-1}}.
	\end{align}
	Similarly, using \eqref{w minimizes} in \eqref{l1omega campanato}, we obtain 
	\begin{align}\label{estimate vxw-vxv}
		\int_{{B_{R}}} \left\lvert V(\X w) - V(\X v) \right\rvert^{2} \le c\left[ \omega\left(R\right) \right]^{2} \int_{{B_{R}}}\left\lvert  \X u  \right\rvert^{p}= c\left[ \omega\left(R\right) \right]^{2} \int_{{B_{R}}}\left\lvert  V(\X u)  \right\rvert^{2} .  
	\end{align}
	Now, using \eqref{Morrey bound p less 2 case} and choosing $\xi=0$ we have 
	\begin{align*}
		\int_{{B_{\rho}}}\left\lvert  V(\X u)  \right\rvert^{2} &=\int_{B_{\rho}}\left\lvert \X u \right\rvert^{p} \leq  c\left( \frac{\rho}{R}\right)^{Q}\int_{B_{R}}\left\lvert \X u \right\rvert^{p}    +  	c\int_{B_{R}}\left\lvert \X u - \X v\right\rvert^{p} \\
		&\stackrel{\eqref{constant cv}}{\leq}  c\left( \frac{\rho}{R}\right)^{Q}\int_{B_{R}}\left\lvert V(\X u) \right\rvert^{2}    +  	c\int_{B_{R}}\left\lvert V(\X u) - V(\X v)\right\rvert^{2}\\
		&\stackrel{\eqref{estimate vxu-vxw}, \eqref{estimate vxw-vxv}}{\leq} c\left[\left( \frac{\rho}{R}\right)^{Q}+ [\omega(R)]^2\right] \int_{B_R}\left\lvert V(\X u) \right\rvert^{2}+ c \int_{{B_{R}}}\left\lvert  F  \right\rvert^{\frac{p}{p-1}}.
	\end{align*}
As before, choosing $R$ small enough and using the iteration Lemma 5.13 in \cite{giaquinta-martinazzi-regularity}, this implies $V(\Xu) \in \mathrm{L}^{2, Q-\delta}_{\text{loc}}$ for every $\delta >0.$ Now, we have, using \eqref{Main excess decay estimate Intro estimate V} of  Theorem \ref{Main excess decay estimate V Intro} with $q=2$, 
	\begin{align*}
		\int_{B_{\rho}}&\left\lvert V(\X u) - \left(V(\X u)\right)_{\rho}\right\rvert^{2}  \\ &\qquad \begin{aligned}
			\leq c \left(\frac{\rho}{R}\right)^{Q+ 2\beta}	\int_{B_{R}} \left\lvert V(\X u) - \left(V(\X u)\right)_{R}\right\rvert^{2}    &+ cR^{ Q+ 2\mu-\delta}	\left\lVert V(\X u) \right\rVert_{\mathrm{L}^{2, Q-\delta}}^{2} \\& + cR^{Q +\frac{p\mu}{p-1}}  \left[ F\right]^{\frac{p}{p-1}}_{\mathcal{L}^{\frac{p}{p-1},Q+ \frac{p\mu}{p-1}}} .
		\end{aligned} 
	\end{align*}
	Recall that $\mu<\beta$. Also, since $p>2$ implies $p/(p-1) < 2,$ we can choose $\delta >0$ small enough such that the minimum power of $R$ on the right is  $Q +\frac{p\mu}{p-1}.$ This, by standard iteration lemma and covering arguments imply $V(\X u) \in \mathcal{L}_{\text{loc}}^{2, Q+ \frac{p\mu}{p-1}}$. Hence, using \eqref{constant cv} we conclude that $\X u \in \mathcal{L}_{\text{loc}}^{p, Q+ \frac{p\mu}{p-1}}.$ By Campanato's characterization, this implies $\X u \in C^{0, \frac{\mu}{p-1}}_{\text{loc}}.$ 
	This completes the proof. 
\end{proof}

\begin{proof}{\textbf{(of Corollary \ref{campanato estimates nondivergence})}}
	This follows immediately from Theorem \ref{campanato estimates} as given any $f \in L^{q}_{\text{loc}}\left( \Omega\right),$ with $q>Q,$ we can always write it as a divergence of a locally H\"{o}lder continuous vector field. More precisely, we can find $F \in C^{0, \frac{q-Q}{q}}_{\text{loc}}\left( \Omega; \mathbb{R}^{2n}\right)$ such that $ \operatorname{div}_{\mathbb{H}} F = f $ in $\Omega.$ Indeed, we just find a solution of 
	\begin{align*}
		\operatorname{div}_{\mathbb{H}} \left( \X \psi \right) = f \text{ in } \Omega, 
	\end{align*}
	and set $F = \X \psi.$ By standard $L^{p}$ estimates for the subLaplacian, (see \cite[page 917]{folland-cz-linear}) we deduce $F \in HW^{1,q}_{\text{loc}}\left( \Omega; \mathbb{R}^{2n}\right).$ Since $q>Q,$ by Proposition \ref{Morrey's theorem}, we have $F \in C^{0, \frac{q-Q}{q}}_{\text{loc}}\left( \Omega; \mathbb{R}^{2n}\right).$ This completes the proof. \end{proof}

\section{Homogeneous equation with Dini coefficients}\label{Homogeneous Dini section} 
\subsection{General setting and comparison estimates}
In this section, our goal is to prove Theorem \ref{dwcontinuityhomogeneousDini}. Let $B(x_{0}, 2R) \subset \Omega$ be a fixed ball and for $i \geq 0,$ we set 
\begin{equation}\label{shrinkingballs1}
	B_{i} \equiv B(x_{0}, R_{i}), \qquad R_{i}:= \sigma^{i} R, \qquad \sigma \in (0, \frac{1}{2}) .
\end{equation}
Let $w$ be a weak solution of \eqref{p subLaplace with dini coeff}. We define the maps $v_{i} \in w + HW_{0}^{1,p}\left( B_{i} \right)$ to be the unique solution of 
\begin{align}\label{p subLaplace homogeneous frozen i}
	\left\lbrace \begin{aligned}
		\operatorname{div}_{\mathbb{H}}	 ( a(x_{0}) \lvert \X v_{i} \rvert^{p-2} \X v_{i}) )  &= 0   &&\text{ in } B_{i},\\
		v_{i} &=w &&\text{  on } \partial B_{i}.                \end{aligned} 
	\right. 
\end{align}
Also, we define the quantities, for $i \geq 0$ and $r \geq 1,$ 
$$ m_{i}(G) := \left\lvert \left( G \right)_{B_{i}}\right\rvert \qquad \text{ and } \qquad 
E_{r}(G, B_{i}):=  \left( \fint_{B_{i}} \lvert G -  \left( G \right)_{B_{i}} \rvert^{r}\right)^{\frac{1}{r}}.$$
Now we record the following comparison estimate, which is proved exactly as Lemma 3 in \cite{Sil_nonlinearStein}. 
\begin{lemma}\label{secondcomparison1}
	Assume that $w,v_{i}$ satisfy \eqref{p subLaplace with dini coeff} and \eqref{p subLaplace homogeneous frozen i} respectively and  $i \geq 0.$ Then there exists a constant $c_{4} \equiv c_{4}\left( n, p, \gamma, L \right)$ such that we have the inequality
	\begin{align}\label{l1omega}
		\fint_{B_{i}} \lvert V( \X v_{i}) - V(\X w) \rvert^{2}  &\leq c_{4}\left[ \omega\left(R_{i}\right) \right]^{2}  
		\fint_{B_{i}} \lvert \X w \rvert^{p}.
	\end{align}
\end{lemma}
\subsection{Proof of Theorem \ref{dwcontinuityhomogeneousDini}}
Now we prove the following pointwise estimate, which immediately implies the sup estimate. 
\begin{theorem}\label{dwsupboundhomogeneousDini}
	Let $w \in HW_{\text{loc}}^{1,p}\left( \Omega \right)$ be be a local weak solution of \eqref{p subLaplace with dini coeff} with $1<p<\infty.$ Then there exists a constant 
	$c = c(n, p, \gamma, L, \omega( \cdot )) \geq 1$ and a positive radius $R_{1}= R_{1}(n, p, \gamma, L, \omega( \cdot )) >0$ such that the pointwise estimate 
	\begin{align}\label{pointwise bound dini}
		\left\lvert \X  w (x) \right\rvert \leq c \left( \fint_{B(x, R)} \left\lvert \X w \right\rvert^{p}\right)^{\frac{1}{p}},
	\end{align}
	holds whenever $B(x,2R) \subset \Omega$, $2R \leq R_{1}$ and $x$ is a Lebesgue point of $\X w.$ If $a(\cdot )$ is a constant function, the estimate holds without any restriction on $R.$
\end{theorem}

\begin{proof}  \emph{ \textbf{Step 1: Choice of constants.}} We pick an arbitrary point $x_{0} \in \Omega$ and a arbitrary positive radius $R_{1} > 0$ such that $B(x_{0}, R_{1}) \subset \Omega.$ We pick $0< R < R_{1}/2$ and for now set 
	$B(x_{0},R)$ as our starting ball and consider the chain of shrinking balls as explained in \eqref{shrinkingballs1} for some parameter $\sigma \in (0,\frac{1}{4}).$ 
	We shall soon make specific choices of both the parameters $R_{1}$ and $\sigma.$
	We define the constant $\lambda$ as 
\begin{align}\label{lambda def}
	 \lambda^{\frac{p}{2}}:= H_{1}\left( \fint_{B(x_{0},R)} \left\lvert V(\X w)\right\rvert^{2}\right)^{\frac{1}{2}} ,
	 \end{align}
	where $H_{1}$ will be chosen soon. 
	Clearly, we can assume $\lambda > 0.$ In view of Theorem \ref{Main excess decay estimate V Intro}, we choose $\sigma \in (0,\frac{1}{4})$ small enough such that 
	\begin{equation}\label{choiceofsigma1}
		C_{0,V} \ \sigma^{\beta} \leq 4^{-(Q+4)}.
	\end{equation}
	Now that we have chosen $\sigma,$ we set 
	\begin{equation}\label{choiceofH11}
		H_{1} : = 10^{5Q} \sigma^{-2Q}.
	\end{equation}
	Note that $H_{1}$ depends only on $n$ and $p.$ 
	Now, we chose the radius $R_{1} > 0$ small enough such that we have 
	\begin{equation}\label{smallnessradius1}
		\omega (R_{1})  + \int_{0}^{2R_{1}} \omega(\varrho) \frac{d \varrho}{\varrho} \leq 
		\frac{\sigma^{2Q}}{6^Q10^6 c_{4}} .
	\end{equation}
	Note that $R_{1}$ depends on $n,p,\gamma, L$ \emph{and} $\omega(\cdot).$ Also, if $a(\cdot)$ is a constant function, the dependence on $\omega( \cdot )$ is redundant. With this, we have chosen all the relevant parameters.\smallskip 
	
	\emph{ \textbf{Step 2: Excess decay.}} We now want to prove 
	\begin{claim}[excess decay estimate]
		If 
		\begin{align}\label{upper bound}
			\left( \fint_{B_{i}} \left\lvert \X w \right\rvert^{p}\right)^{\frac{1}{p}} \leq \lambda  ,
		\end{align} 
		then 
		\begin{equation}\label{excess decay eqn}
			E_{2}(V(\X w), B_{i+1}) \leq \frac{1}{4} E_{2}(V(\X w), B_{i}) + \frac{2c_{4}\lambda^{\frac{p}{2}}}{\sigma^{Q}}\omega\left( R_{i} \right), 
		\end{equation}
		where $c_4$ is given by Lemma \ref{secondcomparison1}.
	\end{claim}
	By the excess decay estimate \eqref{Main excess decay estimate Intro estimate V} of  Theorem \ref{Main excess decay estimate V Intro} with $q=2$ and the choice of $\sigma$ in \eqref{choiceofsigma1}, we have, 
	\begin{equation}\label{e1}
		E_{2}(V(\X v_{i}), B_{i+1}) \leq \frac{1}{4^{Q+4}} E_{2}(V(\X v_{i}), B_{i}).
	\end{equation}
	Now using the property of the mean, triangle inequality and elementary estimates along with Lemma \ref{secondcomparison1}, we have, 
	\begin{align}
		E_{2}(V(\X w), B_{i+1}) &\leq  E_{2}(V(\X v_{i}), B_{i+1}) + \sigma^{-Q/2} \left( \fint_{B_{i}} \lvert V(\X v_{i}) - V(\X w )\rvert^{2} \right)^{\frac{1}{2}} \notag \\
		&\stackrel{\eqref{l1omega},\eqref{upper bound}}{\leq} E_{2}(V(\X v_{i}),B_{i+1}) + \sigma^{-Q} c_{4}\omega\left(R_{i} \right) \lambda^{\frac{p}{2}} \label{e2}.
	\end{align}
	Also similarly, 
	\begin{align}\label{e3}
		E_{2}(V(\X v_{i}),B_{i}) &\stackrel{\eqref{l1omega},\eqref{upper bound}}{\leq} E_{2}(V(\X w),B_{i}) + c_{4}\omega\left(R_{i} \right) \lambda^{\frac{p}{2}} .
	\end{align}
	Combining \eqref{e1}, \eqref{e2} and \eqref{e3}, we have \eqref{excess decay eqn}.\smallskip

	\emph{\textbf{Step 3: Control on composite quantities.}} Now we want to show, by induction that 
	\begin{equation}\label{induction for m +E}
		m_{i}(V(\X w)) + E_{2}(V(\X w), B_{i}) \leq \lambda^{\frac{p}{2}}  \qquad \text{ for all } i \geq 1.
	\end{equation}
	This is true for $i = 1$ by our choice of $H_{1}$ in \eqref{choiceofH11} and elementary estimates. Thus we assume this is true for all $j \in \lbrace 1, \ldots, i\rbrace$ and prove it for $i+1.$ But we have , 
	\begin{align*}
		\fint_{B_{j}} \left\lvert \X w \right\rvert^{p} = \fint_{B_{j}} \left\lvert V(\X w ) \right\rvert^{2} &\leq \left[ m_{j}(V(\X w))\right]^{2} + \left[ E_{2}(V(\X w), B_{j})\right]^{2}\\ &\leq \left[ m_{j}(V(\X w)) +  E_{2}(V(\X w), B_{j})\right]^{2}. 
	\end{align*} Thus the induction hypotheses 
	implies the bound 
	\begin{align*}
		\left( \fint_{B_{j}} \left\lvert \X w \right\rvert^{p}\right)^{\frac{1}{p}} \leq \lambda \text{ for all } j \in \lbrace 1, \ldots, i\rbrace. 
	\end{align*}
	Thus, by the estimates \eqref{excess decay eqn} for each $j$ and summing, we obtain, 
	\begin{align*}
		\sum_{j=2}^{i+1} E_{2}(V(\X w), B_{j}) \leq \frac{1}{2} \sum_{j=1}^{i} E_{2}(V(\X w), B_{j}) + \frac{2c_{4}\lambda^{\frac{p}{2}}}{\sigma^{Q}} \sum_{j=1}^{i}  
		\omega\left( R_{j} \right).
	\end{align*}
	Note that 
	\begin{align}\label{sum modulus of continuity}
		\int_{0}^{2R} \omega(\varrho) \frac{d \varrho}{\varrho} &= \sum_{i=0}^{\infty} \int_{R_{i+1}}^{R_{i}} \omega(\varrho) \frac{d \varrho}{\varrho} 
		+ \int_{R}^{2R} \omega(\varrho) \frac{d \varrho}{\varrho} \notag \\
		&\geq \sum_{i=0}^{\infty} \omega(R_{i+1}) \int_{R_{i+1}}^{R_{i}} \frac{d \varrho}{\varrho} + \omega(R) \int_{R}^{2R} \frac{d \varrho}{\varrho} \notag \\
		&= \log(\frac{1}{\sigma}) \sum_{i=0}^{\infty} \omega(R_{i+1}) + \omega(R) \log 2  \geq \log 2\sum_{i=0}^{\infty} \omega(R_{i}). 
	\end{align}
	
	Thus, we deduce, using \eqref{lambda def}, \eqref{choiceofH11}, \eqref{sum modulus of continuity} and \eqref{smallnessradius1}, 
	\begin{align}\label{e4}
		\sum_{j=1}^{i+1} E_{2}(V(\X w), B_{j}) \leq 2 E_{2}(V(\X w), B_{1}) + \frac{4c_{4}\lambda^{\frac{p}{2}}}{\sigma^{Q}} \sum_{j=1}^{i}\omega\left( R_{j} \right)
		\leq \frac{\sigma^{Q}\lambda^{\frac{p}{2}}}{100}.
	\end{align}
	Now, using this and writing the difference of averages as a telescoping sum, we deduce, 
	\begin{align}
		m_{i+1}(V(\X w)) -  m_{1}(V(\X w)) &= \sum_{j=1}^{i} \left( m_{j+1}(V(\X w)) -  m_{j}(V(\X w))\right) \notag\\
		&\leq \sum_{j=1}^{i} \fint_{B_{j+1}}\left\lvert V(\X w) - (V(\X w))_{B_{j}} \right\rvert \notag \\
		&\leq \sigma^{-Q}\sum_{j=1}^{i} E_{2}(V(\X w), B_{j}) \leq \frac{\lambda^{\frac{p}{2}}}{100}. \label{e5}
	\end{align}
	Now \eqref{e4} and \eqref{e5} together yields \eqref{induction for m +E}. \smallskip

	\emph{\textbf{Step 4:}} Finally, by \eqref{induction for m +E}, we have, for $i \geq 1,$  
	\begin{align*}
		\fint_{B_{i}}	\left\lvert   \X  w  \right\rvert^{p} =  \fint_{B_{i}}\left\lvert V(\X w) \right\rvert^{2}  
		&\leq \left[ m_{i}(V(\X w))\right]^{2} + \left[ E_{2}(V(\X w), B_{i})\right]^{2}\\ &\leq \left[ m_{i}(V(\X w)) +  E_{2}(V(\X w), B_{i})\right]^{2} \leq \lambda^{p}.
	\end{align*}
	But this implies, if $x_{0}$ is a Lebesgue point of $\left\lvert \X w \right\rvert^{p},$ then we have, 
	\begin{align}
		\left\lvert \X w(x_{0})\right\rvert^{p} =  \lim_{i \rightarrow \infty} 	\fint_{B_{i}}	\left\lvert   \X  w  \right\rvert^{p}  \leq \lambda^{p}.
	\end{align}
	This immediately implies \eqref{pointwise bound dini} and completes the proof. \end{proof}

\begin{proof}[Proof of Theorem \ref{dwcontinuityhomogeneousDini}] Pick $\Omega_1 \subset \subset \Omega$ and set $\lambda := \lVert \X w \rVert_{L^{\infty}(\Omega_1)} + 1.$ Clearly, $\lambda <\infty$ by Theorem \ref{dwsupboundhomogeneousDini}.  We first prove that $V(\X w)$ is continuous. The strategy of the proof is to show that 
	$V(\X w)$ is the locally uniform limit of a net of continuous maps, defined by the averages 
	$$ x \mapsto \left( V(\X w)\right)_{B(x, \rho)}.$$ To do this, we pick any $\Omega' \subset \subset \Omega_1$. We show that for every 
	$\varepsilon > 0,$ there exists a radius $$0 < r_{\varepsilon} \leq \operatorname*{dist}(\Omega', \partial\Omega_1)/1000 = R^{\ast} ,$$ depending only on 
	$n,p, \gamma, L, \omega( \cdot ), \varepsilon$ such that for every $x \in \Omega',$ the estimate 
	\begin{equation}\label{oscillationVdw}
		\left\lvert \left( V(\X w)\right)_{B(x, \rho)} - \left( V(\X w)\right)_{B(x, \varrho)} \right\rvert \leq \lambda^{\frac{p}{2}} \varepsilon \qquad \text{ holds for every } \rho, \varrho \in 
		(0, r_{\varepsilon}].
	\end{equation}
	This would imply that the sequence of maps $x \mapsto \left( V(\X w)\right)_{B(x, \rho)}$ are uniformly Cauchy and would conclude the continuity of $V(\X w).$ \smallskip 
	
	\emph{ \textbf{Step 1: Choice of constants.}} We fix $\varepsilon >0.$ Now we choose the constants as in the proof of boundedness, but in the the scale $\varepsilon.$ More precisely, we choose now, 
	we choose $\sigma \in (0,\frac{1}{4})$ small enough such that 
	\begin{equation}\label{choiceofsigma1epsilon}
		C_{0,V} \ \sigma^{\beta} \leq \frac{\varepsilon}{4^{(Q+4)}},
	\end{equation}
	where $C_{0,V}$ and $\beta$ are constants given by Theorem \ref{Main excess decay estimate V Intro}. Clearly, $\sigma$ depend only on $\varepsilon, n$ and $p$.
	Now, we fix a radius $R_{1} > 0$ small enough such that we have 
	\begin{equation}\label{smallnessradius1epsilon}
		\omega (R_{1})  + \int_{0}^{2R_{1}} \omega(\varrho) \frac{d \varrho}{\varrho} \leq 
		\frac{\varepsilon\sigma^{2Q}}{6^Q10^6 c_{4}} .
	\end{equation}
	Note that $R_{1}$ depends on $n, p,\gamma, L$ ,$\omega(\cdot),$ and this time, also on $\varepsilon.$ Also, if $a(\cdot)$ is a constant function, the dependence on $\omega( \cdot )$ is redundant. 
	With this, we have chosen all the relevant parameters.\smallskip

	\emph{ \textbf{Step 2: Smallness of the excess.}} Fix $R< \min \{R^*, R_1\}$ and consider the chain of shrinking balls as described in \eqref{shrinkingballs1}.  Proceeding exactly as in Step 2 of the proof of Theorem \ref{dwsupboundhomogeneousDini}, we obtain the following.  
	
	For any $i \geq 1,$ we have 
	\begin{equation}\label{i+1 to i estimate for excess of xw}
		E_{2}(V(\X w), B_{i+1}) \leq \frac{\varepsilon}{4} E_{2}(V(\X w), B_{i}) + \frac{2c_{4}\lambda^{\frac{p}{2}}}{\sigma^{Q}}\omega\left( R_{i} \right). 
	\end{equation}
	Now in view of \eqref{sum modulus of continuity} and \eqref{smallnessradius1epsilon} we have from \eqref{i+1 to i estimate for excess of xw}
	\begin{align*}
	E_{2}(V(\X w), B_{i+1}) \leq \frac{\varepsilon}{4} E_{2}(V(\X w), B_{i}) + \frac{\varepsilon\sigma^{Q} \lambda^{\frac{p}{2}}}{3^Q10^6}. 
	\end{align*}
Now iterating this we obtain, 
\begin{align*}
E_{2}(V(\X w), B_{i+1}) \leq \left(\frac{\varepsilon}{4}\right)^{i} E_{2}(V(\X w), B_{1}) + \frac{\varepsilon\sigma^{Q} \lambda^{\frac{p}{2}}}{3^Q10^6}\sum_{j=0}^\infty \left(\frac{\varepsilon}{4}\right)^{i}.
\end{align*}
Now we choose $i_0\equiv i_0(n, p,\gamma, L$, $\omega(\cdot), \varepsilon)\geq 2$, such that $\varepsilon^{i_0-1} \leq \sigma^Q$. Then this implies for any $i\geq i_0+1 $ we have
 \begin{align*}
 E_{2}(V(\X w), B_{i}) \leq \varepsilon \sigma^Q \lambda^{\frac{p}{2}}.
 \end{align*}
	This now easily implies that, given $\varepsilon \in (0,1),$ there exists a positive radius $r_{\varepsilon} = r_{\varepsilon} ( n, p,\gamma, L$, $\omega(\cdot), \varepsilon )$ such that we have 
	\begin{equation}\label{smallness of the excess dini}
		E_{2}(V(\X w), B_{\rho}) \leq \lambda^{\frac{p}{2}} \varepsilon, 
	\end{equation}
	whenever $0 < \rho \leq r_{\varepsilon}$ and $B_{\rho} \subset \subset \Omega.$
	
	\emph{ \textbf{Step 3: Cauchy estimate.}}\smallskip 
	
	\noindent Now we finish the proof. First we fix a radius one last time. 
	Given $\varepsilon \in (0, 1),$ using \eqref{smallness of the excess dini} we choose a radius $R_{3} \equiv R_{3}(n,p,\gamma, L$, $\omega(\cdot), \varepsilon )$ such that 
	\begin{equation}\label{excess very small for epsilon}
		\sup_{0 < \rho \leq R_{3}} \sup_{x \in \Omega_{0}} E_{2}(V(\X w), B_{\rho}) \leq \frac{\sigma^{4Q}\lambda^{\frac{p}{2}} \varepsilon}{10^{10}}.
	\end{equation}
	We set $$ R_{0} := \frac{\min{\lbrace R^{\ast}, R_{3}, R_{1} \rbrace} }{16}$$
	and again consider the chain of shrinking balls as described in \eqref{shrinkingballs1} with the starting radius $R_0$.
	Now we want to show that given two integers $2 \leq i_{1} < i_{2},$ we have the estimate 
	\begin{equation}\label{integer ball average oscillation}
		\left\lvert (V(\X w))_{B_{i_{1}}} - (V(\X w))_{B_{i_{2}}} \right\rvert \leq \frac{\lambda^{\frac{p}{2}} \varepsilon}{10}.
	\end{equation}
	Note that this will complete the proof, since for any $0 < \rho < \varrho \leq \sigma^2 R_0,$ there exist integers such that 
	$$ \sigma^{i_{1}+1}R_{0} < \varrho \leq \sigma^{i_{1}}R_{0} \quad \text{ and } \sigma^{i_{2}+1}R_{0} < \rho \leq \sigma^{i_{2}}R_{0}. $$
	Also, we have the easy estimates
	\begin{align*}
		\left\lvert (V(\X w))_{B_{\varrho}} - (V(\X w))_{B_{i_{1}+1}} \right\rvert &\leq \fint_{B_{i_{1}+1}} \left\lvert  V(\X w) - (V( \X w))_{B_{\varrho}}  \right\rvert \\
		&\leq \frac{\lvert B_{\varrho}\rvert}{\lvert B_{i_{1}+1}\rvert} \fint_{B_{\varrho}} \left\lvert  V(\X w) - (V(\X w))_{B_{\varrho}}  \right\rvert \\
		& \le\sigma^{-Q} E_{2}(V(\X w), B_{\varrho}) \stackrel{\eqref{excess very small for epsilon}}{\leq}\frac{\lambda^{\frac{p}{2}} \varepsilon}{1000}
	\end{align*}
	and similarly 
	$$  \left\lvert (V(\X w))_{B_{\rho}} - (V(\X w))_{B_{i_{2}+1}} \right\rvert \leq \frac{\lambda^{\frac{p}{2}} \varepsilon}{1000}.$$
	These two estimates combined with \eqref{integer ball average oscillation} will establish \eqref{oscillationVdw}. Thus it only remains to establish \eqref{integer ball average oscillation}. Summing up \eqref{i+1 to i estimate for excess of xw} from 
	$i = i_{1}-1$ to $i_{2}-2,$ and using \eqref{excess very small for epsilon}, \eqref{smallnessradius1epsilon} and \eqref{sum modulus of continuity} we obtain 
	$$\sum_{i=i_{1}}^{i_{2}-1}E_{2}(V(\X w), B_{i}) \leq 2 E_{2}(V(\X  w), B_{i_{1}-1}) + \frac{4c_{4}\lambda^{\frac{p}{2}}}{\sigma^{Q}}\sum_{i=1}^{\infty} \omega\left( R_{i} \right)
	\leq \frac{ \sigma^{2Q} \lambda^{\frac{p}{2}} \varepsilon}{50}. $$
	This yields \eqref{integer ball average oscillation} via the elementary estimate 
	\begin{align*}
		\left\lvert (V(\X w))_{B_{i_{1}}} - (V(\X w))_{B_{i_{2}}} \right\rvert &\leq \sum_{i=i_{1}}^{i_{2}-1} \fint_{B_{i+1}} \left\lvert V(dw) - (V(\X w))_{B_{i}} \right\rvert \\
		&\leq \sigma^{-Q} \sum_{i=i_{1}}^{i_{2}-1}E_{2}(V(\X  w), B_{i}) .
	\end{align*}
	This establishes \eqref{oscillationVdw}.\smallskip 
	
	\emph{ \textbf{Step 4: Final conclusions.}} Now we prove Theorem \ref{dwcontinuityhomogeneousDini}. Note that \eqref{oscillationVdw} implies $V(\X w)$ is continuous. \eqref{sup estimate dini} follows from Theorem \ref{dwsupboundhomogeneousDini}. Next, fix $x\in \Omega$, $A, R$ and $\lambda$ satisfying \eqref{sup bound for osc dini}. Now there exists  $0<r_\varepsilon<R/16$ which depends only on $n,p, \gamma, L, \omega( \cdot ),$ $ \varepsilon$, such that for every $\tilde{x}\in B(x,R)$
	\begin{align}
		\left\lvert \left( V(\X w)\right)_{B(\tilde{x}, \rho)} - \left( V(\X w)\right)_{B(\tilde{x}, \varrho)} \right\rvert &\leq (A\lambda)^{\frac{p}{2}} \varepsilon \qquad  \text{ and } \label{smallness of the excess diniFC1}\\ E_{2}(V(\X w), B_{\rho}) &\leq (A\lambda)^{\frac{p}{2}} \varepsilon, \label{smallness of the excess diniFC}
	\end{align}	
	whenever $0 < \rho, \varrho \leq r_{\varepsilon}$.  The proof is similar to  \eqref{oscillationVdw} and \eqref{smallness of the excess dini}.  
	Since $V \left(\X w\right)$ is continuous, by letting $\rho \rightarrow 0$ in \eqref{smallness of the excess diniFC1}, we obtain
	\begin{align*}
		\left\lvert  V(\X w)\left( \tilde{x}\right) - \left( V(\X w)\right)_{B(\tilde{x}, \varrho)} \right\rvert \leq \left(A\lambda\right)^{\frac{p}{2}} \varepsilon \qquad \text{ holds for every }  \varrho \in 
		(0, r_{\varepsilon}].
	\end{align*}
	Choose $\sigma_1 >0$ sufficiently small such that $\sigma_1 R \leq r_{\varepsilon}/{64}.$
	Now we have 
	\begin{align*}
		B(\tilde{y}, r_{\varepsilon}/8)  \subset B(\tilde{x}, r_{\varepsilon }) \subset  B( x, R/4 ), \quad \text{ for any } \tilde{y}, \tilde{x} \in B(x , \sigma_1 R).
	\end{align*}
	We estimate 
	\begin{multline*}
		\left\lvert   \left( V(\X w)\right)_{B(\tilde{y},  r_{\varepsilon}/8)} - \left( V(\X w)\right)_{B(\tilde{x},  r_{\varepsilon}/8)} \right\rvert  \\ \begin{aligned}
		&\leq \fint_{B(\tilde{y},  r_{\varepsilon}/8)}\left\lvert  V(\X w) - \left( V(\X w)\right)_{B(\tilde{x},  r_{\varepsilon}/8)} \right\rvert \\
		&\leq 8^Q\fint_{B(\tilde{x},  r_{\varepsilon})}\left\lvert  V(\X w) - \left( V(\X w)\right)_{B(\tilde{x},  r_{\varepsilon}/8)} \right\rvert \\
		& \leq 8^Q \fint_{B(\tilde{x},  r_{\varepsilon})}\left\lvert  V(\X w) - \left( V(\X w)\right)_{B(\tilde{x},  r_{\varepsilon})} \right\rvert\\
		&\qquad+ 8^Q \left\lvert \left( V(\X w)\right)_{B(\tilde{x},  r_{\varepsilon})} - \left( V(\X w)\right)_{B(\tilde{x},  r_{\varepsilon}/8)} \right\rvert  \\
		&\leq 2\cdot8^{2Q} \fint_{B(\tilde{x},  r_{\varepsilon})}\left\lvert  V(\X w) - \left( V(\X w)\right)_{B(\tilde{x},  r_{\varepsilon})} \right\rvert \\
		&\leq 16^{2Q} \left(\fint_{B(\tilde{x},  r_{\varepsilon})}\left\lvert  V(\X w) - \left( V(\X w)\right)_{B(\tilde{x},  r_{\varepsilon})} \right\rvert^{2} \right)^{\frac{1}{2}} \stackrel{\eqref{smallness of the excess diniFC}}{\leq}16^{2Q}\left(A\lambda\right)^{\frac{p}{2}}\varepsilon.
		\end{aligned}
	\end{multline*} 
Thus, by triangle inequality, we have 
	\begin{multline*}
		\left\lvert  V(\X w)\left( \tilde{x}\right) -  V(\X w)\left( \tilde{y}\right) \right\rvert \\ \begin{multlined}[t]
			\leq 	\left\lvert  V(\X w)\left( \tilde{x}\right) - \left( V(\X w)\right)_{B(\tilde{x}, r_\varepsilon/8)} \right\rvert + 	\left\lvert  V(\X w)\left( \tilde{y}\right) - \left( V(\X w)\right)_{B(\tilde{y}, r_\varepsilon/8)} \right\rvert \\ + \left\lvert  \left( V(\X w)\right)_{B(\tilde{y},  r_{\varepsilon}/8)} - \left( V(\X w)\right)_{B(\tilde{x},  r_{\varepsilon}/8)} \right\rvert, 
		\end{multlined}
	\end{multline*}
	for any $\tilde{x}, \tilde{y} \in B(x , \sigma_1 R).$
	Thus, we deduce 
	\begin{align}\label{prefinal osc estimate dini}
		\left\lvert  V(\X w)\left( \tilde{x}\right) -  V(\X w)\left( \tilde{y}\right) \right\rvert \leq 48^{2Q} \left(A\lambda\right)^{\frac{p}{2}}\varepsilon. 
	\end{align}
	Now, we first prove \eqref{osc estimate dini} when $p \geq 2.$ For any given $\delta >0,$ clearly we can assume 
	\begin{align*}
		\max \left\lbrace \left\lvert \X w\left( \tilde{x}\right)\right\rvert , \left\lvert \X w\left( \tilde{y}\right)\right\rvert\right\rbrace  \geq \frac{\delta \lambda}{2}, 
	\end{align*}
	as the estimate follows trivially otherwise. Using the fact that $p\geq 2,$ we deduce 
	\begin{align*}
		\left\lvert \X w\left( \tilde{x}\right) - \X w\left( \tilde{y}\right)\right\rvert &\leq c_{V} \frac{\left\lvert  V(\X w)\left( \tilde{x}\right) -  V(\X w)\left( \tilde{y}\right) \right\rvert}{\left(\left\lvert \X w\left( \tilde{x}\right)\right\rvert + \left\lvert \X w\left( \tilde{y}\right)\right\rvert\right)^{\frac{p-2}{2}}} \\
		&\leq c_{V}2^{\frac{p}{2}-1} \left( \delta \lambda \right)^{1-\frac{p}{2}}\left\lvert  V(\X w)\left( \tilde{x}\right) -  V(\X w)\left( \tilde{y}\right) \right\rvert \\
		&\stackrel{\eqref{prefinal osc estimate dini}}{\leq} c_{V}2^{\frac{p}{2}-1}48^{2Q} A^{\frac{p}{2}} \delta^{1-\frac{p}{2}}\lambda \varepsilon. 
	\end{align*}
	On the other hand, for $1< p <2,$ we have
	\begin{align*}
		\left\lvert \X w\left( \tilde{x}\right) - \X w\left( \tilde{y}\right)\right\rvert &\leq c_{V} \frac{\left\lvert  V(\X w)\left( \tilde{x}\right) -  V(\X w)\left( \tilde{y}\right) \right\rvert}{\left(\left\lvert \X w\left( \tilde{x}\right)\right\rvert + \left\lvert \X w\left( \tilde{y}\right)\right\rvert\right)^{\frac{p-2}{2}}} \\
		&\stackrel{\eqref{sup bound for osc dini}}{\leq} c_{V}\left(2A\lambda\right)^{\frac{2-p}{2}}\left\lvert  V(\X w)\left( \tilde{x}\right) -  V(\X w)\left( \tilde{y}\right) \right\rvert \stackrel{\eqref{prefinal osc estimate dini}}{\leq} c_{V} 2^{\frac{p-2}{2}}48^{2Q} A\lambda\varepsilon. 
	\end{align*}
	Thus, we choose 
	\begin{align*}
		\varepsilon := \left\lbrace \begin{aligned}
			&\frac{\delta^{\frac{p}{2}}}{c_{V}2^{\frac{p}{2}-1}48^{2Q} A^{\frac{p}{2}}} &&\text{ if } p\geq 2, \\
			&\frac{\delta}{c_{V} 2^{\frac{p-2}{2}}48^{2Q} A} &&\text{ if } 1 < p < 2. 
		\end{aligned}\right. 
	\end{align*}
	Then we can determine the radius $r_{\varepsilon}$ and consequently $\sigma_1$ depending on this choice of $\varepsilon$ which depends only on $n,p, \gamma, L, \omega( \cdot ), \delta$ as desired. This proves \eqref{osc estimate dini}, which implies the continuity of $\X w$. This completes the proof.
\end{proof}
\section{Nonlinear Stein theorem}\label{Nonlinear Stein section}
\subsection{General setting}
 Let $u \in HW_{\text{loc}}^{1,p} \left(\Omega \right)$ be a weak solution to \eqref{Main equation Stein Heisenberg}. Pick $x_{0} \in \Omega$ and $0 < R <1 $ such that $B(x_{0}, 2R) \subset \subset \Omega.$ Then clearly, $u \in HW^{1,p} \left(B_{2R} \right)$ and satisfies 
\begin{align}\label{main equation in a ball}
	\operatorname{div}_{\mathbb{H}}\left(  a(x) \lvert \X u \rvert^{p-2} \X u) \right)   &= f   &&\text{ in } B_{2R}. 
\end{align}
For $j\geq 0,$ we set 

\begin{align}\label{concentric balls}
	B_{j}:= B(x_{0}, r_{j}), \qquad r_{j}:= \sigma^{j}R, \quad \sigma \in (0,1/4).
\end{align} 
With a slight abuse of notation we denote 
$$ cB_j: = B(x_0,cr_j), \text{ for any } c>0.$$
\noindent For $j \geq 0,$ we define $w_{j} \in u + HW_{0}^{1,p}\left( B_{j}\right)$ to be the unique solution of 
\begin{align}\label{quasilinear homogeneous general}
	\left\lbrace \begin{aligned}
		\operatorname{div}_{\mathbb{H}}\left(  a(x) \lvert \X w_{j} \rvert^{p-2} \X w_{j}) \right) &= 0   &&\text{ in } B_{j},\\
		w_{j} &= u &&\text{  on } \partial B_{j},
	\end{aligned} 
	\right. 
\end{align}
and $v_{j} \in w_{j} + HW_{0}^{1,p}\left( B(x_0,R_j/2)\right)$ to be the unique solution of 
\begin{align}\label{quasilinear homogeneous frozen general}
	\left\lbrace \begin{aligned}
		\operatorname{div}_{\mathbb{H}} \left(  a(x_{0}) \lvert \X v_{j} \rvert^{p-2} \X v_{j}) \right)  &= 0   &&\text{ in }  B(x_0,R_j/2),\\
		v_{j} &= w_{j} &&\text{  on } \partial \left(  B(x_0,R_j/2)\right).
	\end{aligned} 
	\right. 
\end{align}
For $1<p<\infty$, $p'$ denotes the number $\frac{p}{p-1}$ and let $p_{0}>p$ be any real number. Set 
\begin{align}\label{qdef}
		p^{\ast} := \left\lbrace \begin{aligned}
		&\frac{Qp}{Q-p}, &&\text{if } p < Q \\
		&p_{0}, &&\text{if } p\geq Q,
	\end{aligned}\right. \quad \text{ and } \quad 	q := \left\lbrace \begin{aligned}
	&\frac{Qp}{Q+p} &&\text{if } p \geq 2 \\
	&\frac{Qp}{Qp -Q +p} &&\text{if } 1 < p < 2 .
\end{aligned}\right.
\end{align}
\subsection{New ingredients}
Now we record the new ingredients supplied by our analysis in the present work. The first one is a consequence of Theorem \ref{dwcontinuityhomogeneousDini}. 
\begin{lemma}\label{estimate for wj}
Let $w_{j} \in u + HW_{0}^{1,p}\left( B_{j}\right)$ to be the unique weak solution of \eqref{quasilinear homogeneous general} for $j \geq 0$. Then there exists a constant 
	$C_{2} \equiv C_{2}(n, p, \gamma, L, \omega( \cdot )) \geq 1$ and a positive radius $R_{1}= R_{1}(n, p, \gamma, L, \omega( \cdot )) >0$ such that if $R \leq R_{1}$, then the following holds:  
	\begin{align}\label{sup estimate dini inside proof}
		\sup\limits_{\frac{1}{2}B_{j}} \left\lvert \X w_{j} \right\rvert \leq C_{2}  \fint_{B_{j}} \left\lvert \X w_{j} \right\rvert.
	\end{align}
Furthermore, for any $A \geq 1$ and any $\delta \in (0,1), $ there exists a positive constant 
$\sigma_{1} \equiv \sigma_{1} (n, p, \gamma, L, \omega( \cdot ), A, \delta ) \in (0, \frac{1}{4})$ such that for any $\lambda>0, $ 
	\begin{align}\label{osc estimate dini inside proof}
		\sup\limits_{\frac{1}{2}B_{j}} \left\lvert \X  w_{j} \right\rvert \leq A\lambda \implies \sup_{x,y \in \sigma_{1}B_{j}} \left\lvert \X w_{j}(x) - \X w_{j}(y) \right\rvert \leq \delta \lambda.
	\end{align}
\end{lemma}
The following follows from Theorems \ref{thm:lip Lq bound} and \ref{thm:holder} and does not need our results. 
\begin{lemma}\label{estimates for vj}
	Let $v_{j} \in w_{j} + HW_{0}^{1,p}\left( B(x_0,R_j/2)\right)$ to be the unique weak solution of \eqref{quasilinear homogeneous frozen general} for $j \geq 0$. Then there exists a constant $C_3 \equiv C_3(n, p)\ge 1$ such that 
	\begin{align}
		\sup_{B(x_0, R_j/4)} \lvert \X v_j \rvert \leq C_3 \fint _{B_(x_0, R_j/2)} \lvert \X v_j \rvert .\label{sup estimate for vj} 
	\end{align}
	For every $A\geq 1$ there exists a constants $C_4\equiv C_4(n, p, A) \equiv \tilde{C}_4(n, p) A$ and  $\alpha\equiv \alpha(n,p) \in (0,1)$, such that for any $\lambda>0$ and for any $\tau \in (0, 1/4),$ we have 
	\begin{align}\label{osc estimate for vj}
		\sup_{B(x_0, R_j/4)} \lvert \X v_j \rvert  \leq A\lambda \implies  \sup_{x, y \in B(x_0, \tau R_j)} \lvert \X v_j(x) -\X v_j(y) \lvert  \leq C_4 \tau^\alpha \lambda. 
	\end{align}
\end{lemma}
The second new ingredient is the following, which is just Theorem \ref{Main excess decay estimate Intro}. 
\begin{lemma}\label{smallness from excess decay}
	Let $v_j$ be as in \eqref{quasilinear homogeneous frozen general} for $j\geq 0$. Given $\bar{\varepsilon}>0$, there exists $\sigma_2\equiv \sigma_2(n,p, \gamma, L, \bar{\varepsilon})\in (0,1/4)$, such that if $\sigma \in (0,\sigma_2]$ then for all $t\in [0,1]$, we have
\begin{align}\label{smallness from excess decay estimate}
		\left(\fint_{B_{\sigma R_j}} \left \lvert \X v_j - \left( \X v_j\right)_{\sigma R_j} \right\rvert^t \right)^\frac{1}{t} \leq \bar{\varepsilon} \left(\fint_{B_{\frac{1}{2} R_j}} \left \lvert \X v_j - \left( \X v_j\right)_{\frac{1}{2} R_j} \right\rvert^t \right)^\frac{1}{t}.
	\end{align}
\end{lemma}
\subsection{Comparison estimates}
We begin by recording a number of comparison estimates. The following two can be proved analogously to  Lemma 5 in \cite{Sil_nonlinearStein} and Lemma 3 in \cite{Sil_nonlinearStein} and is valid for $1 < p < \infty.$
\begin{lemma}\label{u w comparison any p easy}
	Let $u$ be as in \eqref{main equation in a ball} and $w_{j}$ be as in \eqref{quasilinear homogeneous general} and $j \geq 0.$ There exists a constant 
	$c_{5} \equiv c_{5}\left( n, p, \gamma, L, \right)$ such that for any $p >1$, the following inequality holds. 
	\begin{equation}\label{estimateuwforanyp}
		\fint_{B_{j}}\left( \left\lvert \X u \right\rvert + \left\lvert \X w_{j} \right\rvert  \right)^{p-2}\left\lvert \X u  - \X w_{j}\right\rvert^{2}  
		\leq c_{5} \fint_{B_{j}} \lvert f \rvert \lvert u -w_{j}\rvert .
	\end{equation}
\end{lemma}
\begin{lemma}\label{secondcomparison2}
	Let $w_{j},v_{j}$ be as before and  $j \geq 0.$ Then there exists a constant $c \equiv c\left( n, p, \gamma, L \right)$ such that we have the inequality 
	\begin{align}\label{l1omegageneralV}
		\fint_{\frac{1}{2}B_{j}} \lvert V\left( \X v_{j} \right)  - V \left( \X w_{j} \right)  \rvert^{2}   &\leq c\left[ \omega\left(r_{j}\right) \right]^{2}  
		\fint_{\frac{1}{2} B_{j}} \lvert \X w_{j} \rvert^{p}.
	\end{align}
\end{lemma}
The following three lemmas are proved analogously to Lemma 7, Lemma 8 and Lemma 9 of \cite{Sil_nonlinearStein}, respectively,  using Lemma \ref{estimate for wj}, Lemma \ref{estimates for vj} and Lemma \ref{smallness from excess decay} in place of Theorem 9, Theorem 8 and Lemma 2, respectively, of \cite{Sil_nonlinearStein}. These provide the comparison estimates for the case $p>2.$ 
\begin{lemma}\label{secondcomparisonpgeq2}
	Let $w_{j}$ be as in \eqref{quasilinear homogeneous general}, $v_{j}$ be as in \eqref{quasilinear  homogeneous frozen general} with $p > 2 $ and $j \geq 0.$ Then there exists a constant $C_{7} \equiv C_{7}\left( n,  p, \gamma, L \right)$ such that we have the inequality
	\begin{align}\label{l1omegageneral}
		\left( \fint_{\frac{1}{2}B_{j}} \lvert \X v_{j} - \X w_{j} \rvert^{p} \right)^{\frac{1}{p}}  &\leq C_{7}\left[ \omega\left(r_{j}\right) \right]^{\frac{2}{p}}  
		\left( \fint_{\frac{1}{2} B_{j}} \lvert \X w_{j} \rvert^{p}\right)^{\frac{1}{p}}.
	\end{align}
\end{lemma}
\begin{lemma}\label{firstcomparison}
	Let $u$ be as in \eqref{main equation in a ball} with $p >2 $ and $w_{j}$ be as in \eqref{quasilinear homogeneous general} and $j \geq 0.$ There exists a constant
	$C_{5} \equiv C_{5}\left( n, p,\tilde{q}, \gamma, L \right)$
	such that the following inequality
	\begin{equation}\label{fq u and w}
		\left( \fint_{B_{j}} \lvert \X  u - \X  w_{j}\rvert^{p} \right)^{\frac{1}{p}}  \leq C_{5}
		\left( r_{j}^{q} \fint_{B_{j}} \lvert f \rvert^{q} \right)^{\frac{1}{q(p-1)}},
	\end{equation}
	holds for every $\tilde{q} \geq ( p^{*})^{'} $ when $p<Q$ and for every $\tilde{q} >1$ when $ p \geq Q.$ Moreover, when $j \geq 1,$ there exists another constant 
	$C_{6} \equiv C_{6}\left( n,  p, \tilde{q}, \gamma, L, \sigma  \right) > 0$ such that the following inequality holds 
	\begin{equation}\label{fq j j-1}
		\left( \fint_{B_{j}} \lvert \X  w_{j-1} - \X  w_{j}\rvert^{p} \right)^{\frac{1}{p}}  \leq C_{6}
		\left( r_{j-1}^{q} \fint_{B_{j-1}} \lvert f \rvert^{\tilde{q}} \right)^{\frac{1}{q(p-1)}}. 
	\end{equation}
\end{lemma}
\begin{lemma}\label{finalcomparispbigger2}
	Let $u$ be as in \eqref{main equation in a ball} with $p >2 $ and $v_{j},w_{j}$ be as before in 
	\eqref{quasilinear homogeneous general}, \eqref{quasilinear homogeneous frozen general}, respectively, for $j \geq 1$ and let $q$ be as in 
	\eqref{qdef}. Suppose for some $\lambda >0$ we have  
	\begin{equation}\label{upper bound of f}
		r_{j-1} \left( \fint_{B_{j-1}}\lvert f \rvert^{q} \right)^{\frac{1}{q}} \leq \lambda^{p-1}
	\end{equation}
	and for some constants $A, B \geq 1,$ the estimates 
	\begin{equation}\label{both side bounds on w}
		\sup\limits_{\frac{1}{2}B_{j}} \left\lvert \X w_{j} \right\rvert \leq A\lambda \qquad \text{ and } \qquad \frac{\lambda}{B} \leq \left\lvert \X w_{j-1} \right\rvert 
		\leq A\lambda \quad \text{ in } B_{j}
	\end{equation}
	hold. Then there exists a constant $C_{8} \equiv C_{8}\left( n, p, \gamma, L, A, B, \sigma \right) >0$ such that 
	\begin{align}\label{grad u -grad v pbigger2}
		\left( \fint_{\frac{1}{2}B_{j}} \left\lvert \X  u - \X  v_{j} \right\rvert^{p'}\right)^{\frac{1}{p'}} \leq C_{8}\omega(r_{j})\lambda 
		+ C_{8}\lambda^{2-p}r_{j-1}\left( \fint_{B_{j-1}} \left\lvert f \right\rvert^{q} \right)^{\frac{1}{q}}.
	\end{align}
\end{lemma}
The following Lemma is proved analogously to Lemma 10 in \cite{Sil_nonlinearStein}, using Lemma \ref{estimates for vj} and Lemma \ref{smallness from excess decay} in place  of Theorem 8 and Lemma 2, respectively, of \cite{Sil_nonlinearStein}. This provides the comparison estimate for the case $1 < p \leq 2.$ 
\begin{lemma}\label{finalcomparispsmaller2}
	Let $u$ be as in \eqref{main equation in a ball} with $1 < p \leq 2 $ and $v_{j},w_{j}$ be as before in 
	\eqref{quasilinear homogeneous general}, \eqref{quasilinear homogeneous frozen general}, respectively, for $j \geq 1$ and let $q$ be as in 
	\eqref{qdef}. Suppose for some $\lambda >0$, we have  
	\begin{align*}
		r_{j} \left( \fint_{B_{j}}\lvert f \rvert^{q} \right)^{\frac{1}{q}} \leq \lambda^{p-1} \qquad \text{ and } \qquad \left( \fint_{B_{j}}\lvert \X u \rvert^{p} \right)^{\frac{1}{p}} \leq \lambda. 
	\end{align*}
Then there exists a constant $C_{8} \equiv C_{8}\left( n,  p, \gamma, L \right) >0$ such that 
	\begin{equation}\label{lambdafirstpsmaller2d}
		\left( \fint_{\frac{1}{2}B_{j}} \left\lvert \X  u - \X  v_{j} \right\rvert^{p}\right)^{\frac{1}{p}} \leq C_{8}\omega(r_{j})\lambda 
		+ C_{8}\lambda^{2-p}r_{j}\left( \fint_{B_{j}} \left\lvert f \right\rvert^{q} \right)^{\frac{1}{q}}. 
	\end{equation}
	Moreover, there exists another constant $C_{5} =C_{5}(n, p, \nu, L ) >0 $ such that  
	\begin{equation}\label{lambdasecondpsmaller2}
		\left( \fint_{B_{j}} \left\lvert \X u - \X  w_{j} \right\rvert^{p}\right)^{\frac{1}{p}} \leq 
		C_{5}\lambda^{2-p}r_{j}\left( \fint_{B_{j}} \left\lvert f \right\rvert^{q} \right)^{\frac{1}{q}}.
	\end{equation} 
\end{lemma}
\subsection{Proof of Theorem \ref{Nonlinear Stein theorem intro}}
\begin{theorem}\label{pointwise gradient bound theorem}
	Let $u$ be as in Theorem \ref{Nonlinear Stein theorem intro}, solves the equation \eqref{Main equation Stein Heisenberg}. Then $\Xu$ is locally bounded in $\Omega$. Moreover, there exists a constant $c\ge 1$ and a positive radius $R_0$, both depending only on $Q, p , \gamma, L$ and $\omega(\cdot)$ such that the pointwise estimate 
	
	\begin{equation}\label{pointwise gradient bound estimate}
		\lvert \Xu (x_0) \rvert \leq c \left( \fint_{B(x_0,r)} \lvert \Xu \rvert^s \right)^\frac{1}{s} + c \lvert \lvert f \rvert \rvert^{\frac{1}{p-1}}_{L^{(Q,1)}(B(x_0,r))}
	\end{equation}
	holds whenever $B(x_0, 2r) \subset \Omega$, $x_0$ is a Lebesgue point of $\Xu$ and $2r\leq R_0$. Estimate \eqref{pointwise gradient bound estimate} holds with no restriction on $r$ when the coefficient function $a(\cdot)$ is constant.
\end{theorem}

\begin{proof}
	The proof of the theorem now proceeds analogously to the proof of Theorem 4 in \cite{KuusiMingione_nonlinearStein}, where $D$ is replaced by $\X,$ $n$ is replaced by $Q$ and the constants are replaced by analogous constants. We use 
	Lemma \ref{estimate for wj}, Lemma \ref{estimates for vj}, Lemma \ref{smallness from excess decay}, Lemma \ref{firstcomparison}, Lemma \ref{secondcomparisonpgeq2}, Lemma \ref{finalcomparispbigger2}, Lemma \ref{finalcomparispsmaller2} and Lemma \ref{Lorenz series estimate lemma} as substitutes for Theorem 2, Theorem 3, Lemma 3, Lemma 4, Lemma 5, Lemma 6, Lemma 7 and Lemma 1, respectively, of \cite{KuusiMingione_nonlinearStein}. We remark that our excess decay estimate is crucial for proving these replacements. \end{proof}  

\begin{proof}[Proof of Theorem \ref{Nonlinear Stein theorem intro}]  Now the proof of Theorem \ref{Nonlinear Stein theorem intro} follows exactly analogously to the proof of Theorem 1 in \cite{KuusiMingione_nonlinearStein}, with the same replacements as above. 
\end{proof}

\end{document}